\documentclass[12pt,fleqn,a4paper]{amsart}

\usepackage[utf8]{inputenc}
\usepackage[T1]{fontenc}
\usepackage{a4wide}

\usepackage{amsmath,amssymb,amsthm,latexsym,hyperref,enumitem}
\usepackage{mathrsfs}
\usepackage{lmodern}
\usepackage{microtype,booktabs,longtable}
\hypersetup{colorlinks=true,linkcolor=blue,citecolor=blue,urlcolor=blue}
\usepackage{url}

\renewcommand{\le}{\leqslant}
\renewcommand{\ge}{\geqslant}

\setlist[enumerate,1]{label=\textup{(\roman*)},leftmargin=*}
\setlist[itemize,1]{leftmargin=*}

\numberwithin{equation}{section}
\newtheorem{theorem}{Theorem}[section]
\newtheorem{mainthm}[theorem]{Main Theorem}
\newtheorem*{theoremA}{Theorem A}
\newtheorem*{theoremB}{Theorem B}
\newtheorem*{theoremC}{Theorem C}
\newtheorem*{theoremD}{Theorem D}
\newtheorem*{theoremE}{Theorem E}
\newtheorem{lemma}[theorem]{Lemma}
\newtheorem{corollary}[theorem]{Corollary}
\newtheorem{proposition}[theorem]{Proposition}
\theoremstyle{definition}
\newtheorem{definition}[theorem]{Definition}
\newtheorem{remark}[theorem]{Remark}
\newtheorem{example}[theorem]{Example}

\newtheorem{convention}[theorem]{Convention}

\newcommand{\N}{\mathbb{N}}
\newcommand{\Q}{\mathbb{Q}}
\newcommand{\C}{\mathbb{C}}
\newcommand{\Z}{\mathbb{Z}}
\newcommand{\R}{\mathbb{R}}

\newcommand{\SB}{\mathrm{SB}}
\newcommand{\Ideal}{\mathrm{Ideal}}
\newcommand{\dist}{\mathrm{dist}}

\newcommand{\Min}{\mathrm{Min}}
\newcommand{\Con}{\mathrm{Con}}
\newcommand{\NSub}{\mathrm{NSub}}
\newcommand{\Aut}{\mathrm{Aut}}
\newcommand{\Inn}{\mathrm{Inn}}
\newcommand{\Id}{\mathrm{Id}}
\newcommand{\Sub}{\mathrm{Sub}}

\newcommand{\CTbl}{\mathrm{CTbl}}

\title[Quotient codings and descriptive complexity]
{Quotient codings and the descriptive complexity\\ of algebraic and analytic structures}

\author[T.~Kania]{Tomasz Kania}
\address[T.~Kania]{Mathematical Institute\\Czech Academy of Sciences\\\v Zitn\'a 25 \\115 67 Praha 1\\Czech Republic  and  Institute of Mathematics and Computer Science\\ Jagiellonian University\\ {\L}ojasiewicza 6, 30-348 Krak\'{o}w, Poland
}
\email{kania@math.cas.cz, tomasz.marcin.kania@gmail.com}
\thanks{RVO: 67985840.}
\date{\today}

\subjclass[2020]{Primary 03E15, 46L05; Secondary 03C15, 46H05, 46L80, 54B20}
\keywords{Descriptive set theory, Borel complexity, Polish space, Wijsman topology, $C^*$-algebra, Banach algebra, hyperspace, countable structure, quotient encoding, operator algebra $K$-theory, $\Pi^1_1$-complete}

\begin{document}

\begin{abstract}
We study the descriptive complexity of algebraic and analytic properties by
presenting structures as quotients of fixed generators. For separable
Banach-type structures, admissible kernels form Polish spaces and quotient
norms are continuous in the Wijsman topology. For countable algebraic
structures, congruence spaces are compact and atomic predicates are clopen.

In the unital $C^*$-algebra coding, commutativity, MF-ness, stable finiteness
and existence of a tracial state are $\Pi^0_1$-complete. AF-ness, nuclearity, simplicity,
property~(SP), approximate divisibility, fixed real- and stable-rank bounds, and fixed nuclear-dimension and
decomposition-rank bounds are $\Pi^0_2$-complete. Finiteness of each of these four ranks is $\Sigma^0_3$-complete.
Finite dimensionality and existence of a non-zero finite-dimensional
representation are $\Sigma^0_2$-complete. Residual finite dimensionality
is $\Pi^0_3$-complete, already among AF algebras, while quasidiagonality is
$\Pi^0_3$.
For countable abelian groups, slenderness is $\Pi^0_3$-complete even within
the torsion-free locus. We also obtain a $\Sigma^0_3$ upper bound for uniformly
open multiplication in Banach algebras.

An internal $K_0$ presentation has $F_\sigma$ membership coordinates and is
$\Sigma^0_3$-measurable. We establish Borel tensor-ideal assignments,
continuity for a fixed nuclear tensor factor, and a $G_\delta$ bound for
absorption of a strongly self-absorbing algebra. Finally, separability of the dual is $\Pi^1_1$-complete in the full
unital $C^*$-quotient coding, already among commutative AF quotients;
superatomicity of countable Boolean algebras is also $\Pi^1_1$-complete.
\end{abstract}

\maketitle

\section{Introduction}

A parameter space for mathematical structures carries two distinct kinds
of descriptive questions. One asks how difficult it is to classify its
points up to isomorphism. The other asks how difficult it is to recognise
a property of one coded object. The first question can be extremely
complicated even when the second has a low-rank answer. For example,
Sabok proved that isomorphism of separable simple AI $C^*$-algebras is
complete among orbit equivalence relations~\cite{SabokCompleteness}.
Nevertheless, in a suitable quotient parameter space, simplicity itself
is a $G_\delta$ condition, and stable finiteness is closed.

The purpose of this paper is to make such property computations systematic
and, where possible, sharp. We fix a generator $P$, present an object as
$P/K$, and topologise the collection of admissible kernels. The main
advantage is that the norm of a named quotient element is already a
coordinate of the parameter space:
\[
 K\longmapsto\|x+K\|=\operatorname{dist}(x,K).
\]
In the Wijsman topology these coordinates are continuous. Stable finite
relations turn existential assertions into open conditions. Uniform
families of existential formulas sharpen the nuclearity and dimension
bounds using the model-theoretic results
of~\cite{FarahHartLupiniRobertTikuisisVignatiWinter}, while compact
spaces of completely positive maps give controlled projections of Borel
sets. On the discrete side, equality of two named quotient elements is
clopen. The same approach therefore treats both continuous and countable
algebraic structures, although their atomic topologies differ.

Quotient generators are available in situations where universal embedding
objects are not. There is no universal separable Banach algebra for
isomorphic embeddings~\cite{KaniaNoUniversalBanachAlg}, but every separable
Banach algebra is a quotient of an appropriate free $\ell_1$-algebra.
Similarly, $C^*_{\max}(F_\infty)$ is a quotient generator for separable
unital $C^*$-algebras, and countably generated free objects give quotient
codings for countable groups, rings and other varieties.

The existence of quotient and seminorm codings is not new. We build on
Kechris's parametrisation of separable $C^*$-algebras and the explicit
Borel translations and invariant computations of
Farah--Toms--T\"ornquist~\cite{KechrisAsianCstar,FarahTomsTornquist,FarahTomsTornquistCrelle}.
Godefroy--Saint-Raymond and C\'uth--Dole\v{z}al--Doucha--Kurka develop
related hyperspace and norm codings for Banach spaces
\cite{GodefroySaintRaymond2018,CuthDolezalDouchaKurka2022,CuthDolezalDouchaKurka2024}.
The contribution here is the explicit topology of quotient presentations,
together with matching continuous reductions for several natural classes.
A Borel change of coding need not preserve a finite Borel rank; all
completeness statements below refer to the specified topologies.

\begin{theoremA}
Let $P$ be a separable Banach-type quotient generator with countably many
continuous operations. Its admissible kernels form a closed Polish
subspace for every hyperspace topology satisfying the admissibility
conditions used below. In the Wijsman topology, quotient-norm predicates
are continuous, and countable formulas built from these predicates and
stable existential relations define Borel sets.

For a countable free algebra, the congruence space is compact,
zero-dimensional and Polish. Atomic predicates are clopen. On the locus
of infinite quotients, the canonical passage to structures on $\N$, and
the reverse passage by term evaluation, are continuous.
\end{theoremA}

The continuous and discrete statements are proved in
Theorems~\ref{thm:continuous-main}, \ref{thm:definability-rank}
and~\ref{thm:countable-main}. The hypotheses on admissible topologies
include closedness of the membership relation; we state this explicitly
to distinguish them from broader conventions in the literature.

\begin{theoremB}
In the Wijsman quotient coding based on $A=C^*_{\max}(F_\infty)$:
\begin{enumerate}
\item Commutativity, MF-ness, stable finiteness, existence of a tracial
state and existence of a character are $\Pi^0_1$-complete.
\item Finite dimensionality and existence of a non-zero finite-dimensional
representation are $\Sigma^0_2$-complete, already among simple AF quotients.
\item AF-ness, nuclearity, simplicity, property~\textup{(SP)}, approximate divisibility,
real rank, nuclear dimension and decomposition rank at most $d$
($d\geqslant0$), and topological stable rank at most $n$ ($n\geqslant1$) are $\Pi^0_2$-complete.
\item Finiteness of real rank, topological stable rank, nuclear dimension
and decomposition rank is $\Sigma^0_3$-complete.
\item Residual finite dimensionality is $\Pi^0_3$-complete, already among
AF quotients. Quasidiagonality is $\Pi^0_3$.
\end{enumerate}
\end{theoremB}

The proofs occupy Sections~\ref{sec:Cstar-properties}
and~\ref{sec:D-stability}. Several explicit
reductions account for the sharp bounds. A continuous deformation of any
unital algebra to the scalars proves completeness for the closed classes.
Compact-space reductions give sharp rank bounds in the commutative
subfamily. Explicit AF inductive limits detect simplicity, and a family
of finite matrix algebras converging through marked presentations to the
CAR algebra detects approximate divisibility. These reductions identify
the obstruction behind each upper bound. Nuclearity is detected by
a family assembled from finite-dimensional quotients of $C^*(F_2)$.
The matrix-to-CAR construction also detects finite dimensionality and
existence of finite-dimensional representations. Applying it independently
in countably many summands gives the third-level lower bound for residual
finite dimensionality.

\begin{theoremC}
In the compact quotient coding of countable abelian groups, slenderness
and vanishing of the first Ulm subgroup are $\Pi^0_3$-complete, with
hardness already on the closed locus of torsion-free groups. Reducedness
is $\Pi^0_3$-complete on that locus.

For countable groups, abelianity, torsion-freeness, soficity, each fixed nilpotency
bound and each fixed derived-length bound are $\Pi^0_1$-complete.
Finiteness is $\Sigma^0_2$-complete for countable groups, abelian groups,
rings, lattices and Boolean algebras. Dedekind finiteness of countable
unital rings is $\Pi^0_1$-complete. Local finiteness of countable groups
is $\Pi^0_2$-complete, while nilpotence and solubility are
$\Sigma^0_2$-complete.

Fields, division rings, simple rings, local rings, von Neumann regular
rings and unit-regular rings are $\Pi^0_2$-complete, with all six lower
bounds holding among characteristic-zero integral domains. For countable
Boolean algebras, atomlessness is $\Pi^0_2$-complete and atomicity is
$\Pi^0_3$-complete.
\end{theoremC}

Section~\ref{sec:countable-complexity} proves these assertions. The
slenderness reduction uses only direct sums of rank-one groups: a finite
row of a binary array produces an infinite cyclic summand, and an
infinite row produces a rational summand. Thus the third-level lower
bound persists in a particularly elementary family of torsion-free groups.
One localisation family proves all six ring lower bounds. Stone duality
turns the Boolean reductions into elementary tests for isolated points
in compact subsets of Cantor space.

\begin{theoremD}
There is a Borel assignment
\[
 \kappa_{K_0}:\Ideal(A)\longrightarrow\Sub(\Z^{(\N)})
\]
whose quotient group is $K_0(A/K)$. Every membership coordinate is
$F_\sigma$, and the map is $\Sigma^0_3$-measurable. The established unital $K_1$ assignment gives a Borel map in the
same group coding; complex Bott periodicity gives the higher groups.
Elementary tensor-ideal assignments are $\Sigma^0_2$-measurable;
tensoring with a fixed nuclear $C^*$-algebra gives a continuous kernel map.
For fixed separable unital exact $D$, $D$-absorption is analytic; for
strongly self-absorbing $D$ it is $G_\delta$. Absorption of the CAR algebra
is $\Pi^0_2$-complete.
\end{theoremD}

The $K$-theory assertions refine the known Borel computability results
by giving an internal presentation and explicit coordinate bounds.
We distinguish measurability ranks from Baire classes defined through
pointwise limits, and elementary tensor ideals from kernels of minimal
tensor quotient maps. The proofs appear in
Sections~\ref{sec:K0}--\ref{sec:D-stability}.

\begin{theoremE}
Having separable dual is $\Pi^1_1$-complete in the full Wijsman ideal
space of $C^*_{\max}(F_\infty)$, with hardness already among
commutative AF quotients. Superatomicity is $\Pi^1_1$-complete in the
Boolean-algebra quotient coding.
\end{theoremE}

The coanalytic upper bound for separability of the dual follows by an
explicit translation to Bossard's Banach-space quotient coding.
The lower bounds are continuous transfers of the classical complexity
of countable compacta. They show that the low Borel bounds in the
preceding results do not extend to every natural property of a quotient.
Section~\ref{sec:Pi11-example} supplies the precise statements.
The Banach-algebra applications in Section~\ref{sec:banach-complexity}
include a $\Sigma^0_3$ upper bound for uniformly open multiplication,
as well as closedness of Dedekind finiteness. Metric quotient maps transfer
the sharp stable-rank bounds to Banach algebras and give
$\Pi^0_2$-hardness for uniformly open multiplication; on the commutative
$C^*$-locus, open and uniformly open multiplication are both
$\Pi^0_2$-complete (Corollary~\ref{cor:Banach-sharp-ranks}); Section~\ref{sec:TRO}
records the corresponding kernel construction for TROs.

Related functorial questions have also received recent attention.
Vaes--Wouters develop a general theory of Borel fields of Banach spaces
and operator algebras~\cite{VaesWouters2025}. Pi--Szachniewicz--Tartarotti
prove analyticity of $KK$-equivalence and of the UCT class
\cite{PiSzachniewiczTartarotti2026}. Shiner has independently announced
ideal-based parametrisations and Borel constructions involving the
Kasparov product~\cite{ShinerForthcoming}. Our focus is the finite
Borel hierarchy of concrete properties in a fixed quotient topology.

\section{Preliminaries and notation}
\label{sec:prelim}

Throughout, $X$ denotes a separable Banach space.  For a non-empty closed set
$F\subseteq X$ and $x\in X$ we write
\[
d(x,F)=\inf\{\|x-y\|:y\in F\}.
\]
If $F$ is a closed linear subspace then $d(x,F)=\|x+F\|_{X/F}$.
For each separable space $X$ under discussion we fix a countable dense set
$D^X=\{u_m:m\in\N\}$.

We shall use standard descriptive set theory; see~\cite{Kechris,BeckerKechris,GaoDST}.
We recall the key notions for the reader's convenience.

\subsection{The Borel and projective hierarchies}\label{subsec:borel-hierarchy}

Let $X$ be a Polish space (separable, completely metrisable).
The \emph{Borel hierarchy} stratifies the Borel $\sigma$-algebra of $X$ into
countable levels.  At the base,
$\Sigma^0_1$ denotes the open sets and $\Pi^0_1$ the closed sets.
Inductively, for $n\ge 1$:
$\Sigma^0_{n+1}$ consists of countable unions of $\Pi^0_n$ sets, and
$\Pi^0_{n+1}$ consists of countable intersections of $\Sigma^0_n$ sets.
Equivalently, $\Pi^0_n$ is the class of complements of $\Sigma^0_n$ sets.
In low levels:
$\Sigma^0_2 = F_\sigma$ (countable unions of closed sets),
$\Pi^0_2 = G_\delta$ (countable intersections of open sets),
$\Sigma^0_3 = G_{\delta\sigma}$, and $\Pi^0_3 = F_{\sigma\delta}$.

Beyond the Borel hierarchy lie the \emph{projective} (or \emph{analytic}) classes.
A set is $\Sigma^1_1$ (\emph{analytic}) if it is the continuous image of a Borel set
in another Polish space, or equivalently the projection of a closed subset of
$X\times\N^\N$.  A set is $\Pi^1_1$ (\emph{coanalytic}) if its complement is analytic.
Suslin's theorem asserts that a set is Borel iff it is both analytic and coanalytic.

\smallskip
\noindent\textbf{Hardness and completeness.}
Let $\Gamma$ be a pointclass ($\Sigma^0_n$, $\Pi^0_n$, $\Sigma^1_1$, $\Pi^1_1$, etc.).
A set $A\subseteq X$ is \emph{$\Gamma$-hard} if for every set $B\in\Gamma(Y)$ (where $Y$
is any zero-dimensional Polish space) there exists a continuous map $f:Y\to X$ with $f^{-1}(A)=B$;
informally, $A$ is ``at least as complex as any $\Gamma$ set''.
The set $A$ is \emph{$\Gamma$-complete} if it belongs to $\Gamma$ and is $\Gamma$-hard.
Completeness for $\Gamma$ implies that $A$ does not belong to the dual class
$\check\Gamma$ (e.g.\ a $\Sigma^0_2$-complete set is not $\Pi^0_2$, and a
$\Pi^1_1$-complete set is not Borel).
Proving a property $\Gamma$-complete therefore establishes that its upper bound
within the hierarchy is \emph{optimal} and cannot be improved.

\smallskip
\noindent\textbf{Connection to definability.}
In this paper, upper bounds on Borel complexity arise from
expressing properties via countable quantifier alternations over
continuous predicates.  Each universal quantifier (over a countable domain) preserves
$\Pi^0_n$; each existential quantifier preserves $\Sigma^0_n$; alternation
increments the level by one.  Lower bounds (completeness results) are proved
by explicit continuous reductions from canonical complete sets, such as the set of
eventually-zero sequences in $2^\N$ ($\Sigma^0_2$-complete) or the set of
well-founded trees ($\Pi^1_1$-complete).

\subsection{Hyperspaces and admissible topologies}\label{sec:hyperspace}

Let $X$ be a separable metric space and let $\mathscr F(X)$ denote the hyperspace of
non-empty closed subsets of $X$.

\begin{definition}\label{def:wijsman}
The \emph{Wijsman topology} on $\mathscr F(X)$ is the coarsest topology making each map
$F\mapsto d(x,F)$ continuous for every fixed $x\in X$.
Equivalently, a net $(F_\alpha)$ converges to $F$ in the Wijsman topology iff
$d(x,F_\alpha)\to d(x,F)$ for every $x\in X$; in other words, the distance functions
converge pointwise.

If $X$ is separable, the Wijsman topology is metrisable
(indeed, fixing a countable dense set $\{u_m\}$ one can take
$d_W(F,G)=\sum_m 2^{-m}\min\bigl(1,|d(u_m,F)-d(u_m,G)|\bigr)$
as a compatible metric).
If the metric on $X$ is complete and separable, then
$(\mathscr F(X),\mathrm{Wijsman})$ is Polish;
see Beer~\cite{Beer}.
\end{definition}

The Wijsman topology is the natural topology on our parameter spaces throughout
this paper: for Banach-type structures, it makes the quotient-norm functional
$K\mapsto\|x+K\|$ continuous by definition, and for countable structures
for the discrete $0$--$1$ metric, the Wijsman topology is the product topology
on membership indicators, restricted to the non-empty subsets.  This is the
topology used in the countable congruence spaces below.

\begin{lemma}\label{lem:wijsman-approx}
Let $F_n\to F$ in the Wijsman topology.
\begin{enumerate}
\item If $x\in F$, there exist $x_n\in F_n$ with $x_n\to x$.
\item If $F_n\subseteq G_n$ for all $n$ and $G_n\to G$ (Wijsman), then $F\subseteq G$.
\end{enumerate}
\end{lemma}

\begin{proof}
(i) Since $d(x,F)=0$ and $d(x,F_n)\to d(x,F)=0$, choose $x_n\in F_n$ with
$\|x-x_n\|\le d(x,F_n)+1/n$.  Then $\|x-x_n\|\to 0$.

(ii) Let $x\in F$.  Wijsman convergence $F_n\to F$ gives $d(x,F_n)\to d(x,F)=0$.
Since $F_n\subseteq G_n$ for all $n$, we have $d(x,G_n)\le d(x,F_n)$ for all $n$,
hence $d(x,G_n)\to 0$.  As $G_n\to G$ in the Wijsman topology,
$d(x,G_n)\to d(x,G)$, so $d(x,G)=0$ and therefore $x\in G$.
\end{proof}

The Wijsman topology is one instance of a broader class of
hyperspace topologies introduced by Godefroy and
Saint-Raymond~\cite{GodefroySaintRaymond2018}.

\begin{definition}
\label{def:admissible}
For this paper, a Polish topology $\tau$ on $\mathscr F(X)$ is
\emph{admissible} if it satisfies the following conditions, including the
closed-membership condition that is optional in
\cite{GodefroySaintRaymond2018,CuthDolezalDouchaKurka2022}:
\begin{enumerate}
\item for each open $U\subseteq X$,
$E^+(U)=\{F\in\mathscr F(X):F\cap U\ne\varnothing\}$ is $\tau$-open;
\item $\tau$ has a subbase consisting of countable unions of sets of the form
$E^+(U)\setminus E^+(V)$ with $U,V$ open in $X$;
\item the membership relation $\{(x,F):x\in F\}\subseteq X\times\mathscr F(X)$
is $\tau$-closed.
\end{enumerate}
\end{definition}

The hit condition permits approximation of each point of a limiting closed
set by points of nearby closed sets.  Closedness of membership supplies the
converse implication: limits of convergent selections remain in the limiting
set.  These are the two facts used in the closure arguments below.

The Wijsman topology is admissible.
The first two conditions are standard (see~\cite{Beer}).  For the third:
if $(x_n,F_n)\to(x,F)$ in $X\times\mathscr F(X)$ and $x_n\in F_n$ for all~$n$,
then $d(x,F_n)\le d(x,x_n)+d(x_n,F_n)=d(x,x_n)\to 0$,
so $d(x,F)=\lim_n d(x,F_n)=0$, hence $x\in F$ and the membership relation is closed.

Let $\SB(X)$ denote the family of closed linear subspaces of~$X$.  For admissible~$\tau$,
$\SB(X)$ is closed, hence Polish, under the closed-membership hypothesis
used here; see Remark~\ref{rem:SBX-closed}. Whenever an admissible topology
is mentioned on $\SB(X)$ or an ideal space, we mean this induced topology.

\begin{proposition}
\label{prop:SBX-Wijsman-closed}
Let $F_n\in\SB(X)$ and $F_n\to F$ in the Wijsman topology on $\mathscr F(X)$.
Then $F\in\SB(X)$; hence $\SB(X)$ is closed in $(\mathscr F(X),\mathrm{Wijsman})$.
\end{proposition}

\begin{proof}
By Wijsman convergence, $d(0,F_n)=0$ for all $n$ implies $d(0,F)=\lim_n d(0,F_n)=0$,
so $0\in F$.

Let $x,y\in F$ and let $a,b$ belong to $\Q$ in the real case or $\Q(i)$
in the complex case.  By the Wijsman approximation lemma, there exist
$x_n,y_n\in F_n$ with $x_n\to x$ and $y_n\to y$ in $X$.
Since each $F_n$ is a linear subspace, we have $z_n:=ax_n+by_n\in F_n$.
By continuity of addition and scalar multiplication, $z_n\to z:=ax+by$.
Therefore $d(z,F_n)\le \|z-z_n\|\to 0$,
and Wijsman convergence yields
$d(z,F)=\lim_n d(z,F_n)=0$, so $z\in F$.
Hence $F$ is closed under rational or complex-rational linear combinations.
Finally, for an arbitrary scalar $t$ and $x\in F$, choose scalars $t_k\to t$
from the corresponding countable field;
then $t_k x\in F$ for all~$k$, and since $F$ is closed, $tx=\lim_k t_k x\in F$.
Thus $F$ is a closed linear subspace.
\end{proof}

\paragraph{Substructure variant.}
The same argument does not use linearity per se.  If $X$ is a separable metric
space equipped with countably many continuous operations
$\omega_i:X^{n_i}\to X$, then the family of all non-empty closed subsets
$F\subseteq X$ that are closed under each $\omega_i$ is Wijsman-closed in
$\mathscr F(X)$.  Indeed, if $F_n\to F$ in the Wijsman topology and each $F_n$
is a closed substructure, then for any $x_1,\dots,x_{n_i}\in F$ one may choose
$x_j^{(n)}\in F_n$ with $x_j^{(n)}\to x_j$ by
Lemma~\ref{lem:wijsman-approx}.  Continuity gives
$\omega_i(x_1^{(n)},\dots,x_{n_i}^{(n)})\to \omega_i(x_1,\dots,x_{n_i})$, and since
$\omega_i(x_1^{(n)},\dots,x_{n_i}^{(n)})\in F_n$ for all $n$, Wijsman convergence
implies $\omega_i(x_1,\dots,x_{n_i})\in F$.  (For nullary operations, the same
argument applies to the distinguished constants.)  In particular, for a separable
$C^*$-algebra $A$, the family of closed $*$-subalgebras of $A$ is Wijsman-closed.

\begin{remark}\label{rem:SBX-closed}
Our definition includes the closed-membership condition, which is optional in
the general definition of admissibility in
\cite{GodefroySaintRaymond2018,CuthDolezalDouchaKurka2022}.
Under this additional hypothesis, closedness holds for every admissible
topology used here.  Indeed, if $F_n\to F$ and $x\in F$, openness of all hit
sets gives points $x_n\in F_n$ with $x_n\to x$.  Closedness of membership gives
the converse limit implication: if $x_n\in F_n$ and $x_n\to x$, then $x\in F$.
These two facts show that any condition of closure under continuous operations
passes to limits.  In particular, $\SB(X)$ is closed.
\end{remark}

\subsection{Matrix-level quotient norms}\label{subsec:matrix-Phi}

In the $C^*$-algebra parts of this paper, several arguments require norms
in \emph{matrix amplifications} $M_n(A/K)$ rather than merely in~$A/K$ itself
(see \cite[II.6.6]{BlackadarOA} or \cite[Chapter~1]{BrownOzawa} for background on
matrix norms and operator-space structure).
We introduce a systematic notation.

\begin{definition}\label{def:Phi-matrix}
Let $A$ be a separable $C^*$-algebra and $K\in\Ideal(A)$.
Because $M_n$ is nuclear, the canonical identification $M_n(A)\cong A\otimes M_n$
carries $M_n(K)=K\otimes M_n$ isometrically.  For $X\in M_n(A)$ define
\[
\Phi^{(n)}(K,X):=\dist\bigl(X,\,M_n(K)\bigr)=\|X+M_n(K)\|_{M_n(A/K)}.
\]
\end{definition}

\begin{lemma}\label{lem:Phi-n-continuous}
For each fixed $n\in\N$ and $X\in M_n(A)$, the map
$K\mapsto \Phi^{(n)}(K,X)$ is continuous in the Wijsman topology on $\Ideal(A)$.
\end{lemma}

\begin{proof}
Suppose $K_m\to K$ Wijsman and write $B_m=A/K_m$.  The homomorphism
\[
 a\longmapsto[(a+K_m)_m]
 \quad\text{from }A\text{ into }\prod_m B_m/\bigoplus_m B_m
\]
has kernel $K$, since its norm at $a$ is
$\limsup_m\|a+K_m\|=\|a+K\|$.  It therefore induces an injective
$*$-homomorphism on $A/K$, which is completely isometric.  Its amplification
gives, for $X\in M_n(A)$,
\[
 \|X+M_n(K)\|=\limsup_m\|X+M_n(K_m)\|.
\]
Applying the same argument to every subsequence proves convergence, as
required.  Here finite matrix amplification commutes with the product and
the $c_0$-sum, and the norm in the resulting quotient is the limsup norm.
\end{proof}

\begin{remark}\label{rem:Phi-n-dense}
Fix a countable dense $\Q(i)$-$*$-subalgebra $D^A\subseteq A$.  Then the set 
\[
    M_n(D^A)=\{(a_{ij})_{i,j\le n}:a_{ij}\in D^A\}
\]
is a countable dense
subset of~$M_n(A)$, and it is closed under multiplication, $*$, and rational
linear combinations.
All arguments below that invoke $\Phi^{(n)}$ may therefore be evaluated on
this countable dense set with no loss.
\end{remark}

\subsection{Projection of locally closed sets along compact spaces}

The following projection lemma is used in the quasidiagonality and
residual finite-dimensionality arguments
(Proposition~\ref{prop:QD-MF} and Theorem~\ref{thm:RFD-Pi3-complete}).

\begin{lemma}\label{lem:proj-locally-closed}
Let $X$ be a compact metric space, let $Y$ be a Polish space, and let $C\subseteq X\times Y$ be closed and $U\subseteq X\times Y$ open.  Then $\pi_Y(C\cap U)$ is $F_\sigma$ in~$Y$.
\end{lemma}

\begin{proof}
Write
\[
U=\bigcup_{m\ge 1}F_m,
\qquad
F_m:=\bigl\{z\in X\times Y:d\bigl(z,(X\times Y)\setminus U\bigr)\ge 1/m\bigr\}.
\]
Each $F_m$ is closed in $X\times Y$ and satisfies $F_m\subseteq U$.
(We do \emph{not} claim that $F_m$ is compact, since only the $X$-factor is compact.)
Now $C\cap F_m$ is closed in $X\times Y$, and because $X$ is compact and $Y$ is
Hausdorff, the projection $\pi_Y(C\cap F_m)$ is closed in~$Y$.
Therefore
\[
\pi_Y(C\cap U)=\bigcup_{m\ge 1}\pi_Y(C\cap F_m)
\]
is $F_\sigma$.
\end{proof}

\begin{convention}\label{conv:enumeration}
Each countably infinite free object is identified with $\N$ using a fixed
enumeration; a finite object uses a finite initial segment. This makes
$\min([a]_S)$ well-defined for each congruence class. No computability
assumption is imposed. Indices start at zero when convenient, and an index
appearing as a reciprocal tolerance is always positive.
\end{convention}

\section{Main theorem for separable continuous structures}\label{sec:continuous}

This section develops the analytic half of our general framework:
we construct Polish parameter spaces for separable Banach-type structures
(such as Banach algebras, $C^*$-algebras, Banach lattices, and ternary rings of
operators) and establish a general definability theorem linking algebraic formulae
to Borel complexity.  The key idea is to view such structures not abstractly, but
concretely as \emph{quotients} of a single universal object endowed with
finitely or countably many continuous multilinear operations.

\subsection{Categories with continuous operations}

\begin{definition}\label{def:category-ops}
A \emph{concrete category $\mathcal{C}$ with continuous operations} consists of:
\begin{itemize}
\item \textbf{Objects}: Separable Banach spaces $X$ equipped with countably many
continuous operations
$\omega_i: X^{n_i} \to X$, $i\in I$ (countable), where $n_i\in\N\cup\{0\}$ is the arity
of~$\omega_i$.
\item \textbf{Morphisms}: Bounded linear maps $T: X \to Y$ preserving all operations:
\[
T(\omega_i^X(x_1, \ldots, x_{n_i}))
= \omega_i^Y(T(x_1), \ldots, T(x_{n_i})).
\]
\item \textbf{Quotients}: Every operation-compatible closed linear
subspace is a kernel whose quotient, with its quotient norm and descended
operations, is an object of the category. If $T: X \to Y$ is surjective with kernel $K = \ker(T)$,
then $Y \cong X/K$ as objects in $\mathcal{C}$, where operations descend to quotients.
\end{itemize}
\end{definition}

\begin{definition}\label{def:admissible-kernel}
Let $X\in\mathcal{C}$.  A closed linear subspace $K\subseteq X$ is an
\emph{admissible kernel} (or \emph{$\mathcal{C}$-ideal}) if for every $i\in I$ and all
$x_1,\ldots,x_{n_i},x_1',\ldots,x_{n_i}'\in X$ with $x_j-x_j'\in K$ for
$1\le j\le n_i$, one has
\[
\omega_i(x_1,\dots,x_{n_i})-\omega_i(x_1',\dots,x_{n_i}')\in K.
\]
Equivalently, $\omega_i$ descends to a well-defined operation on $X/K$;
in the Banach setting, this means that the equivalence relation
$E_K$ on $X$ given by $x\,E_K\,y$ iff $x-y\in K$ is a congruence for all the
operations $\omega_i$.  We write $\Ideal_{\mathcal C}(X)$ for the collection of
admissible kernels.
\end{definition}

\paragraph{Congruence viewpoint.}
The relation $x\,E_K\,y$ if and only if $x-y\in K$ is a closed congruence,
and Section~\ref{sec:countable} gives the discrete counterpart of this
construction.  The analogy does not assert a general closure theorem for
closed congruences on arbitrary Polish structures: transitivity need not
survive hyperspace limits.  Our continuous arguments use the linear
structure and the quotient norm.

\begin{example}\label{ex:standard-cats}
We list canonical examples fitting the framework of
Definition~\ref{def:category-ops}.

\begin{enumerate}
\item \emph{Banach algebras.}  $\omega_1(x,y)=xy$ (multiplication).
Admissible kernels are closed two-sided ideals.

\item \emph{$C^*$-algebras.}  $\omega_1(x,y)=xy$, $\omega_2(x)=x^*$.
Admissible kernels are closed $*$-ideals.

\item \emph{Banach lattices.}
$\omega_1(x,y)=x\vee y$, $\omega_2(x,y)=x\wedge y$, $\omega_3(x)=|x|$.
Admissible kernels are closed lattice ideals.

\item \emph{Ternary rings of operators (TROs).}
$\omega_1(x,y,z)=[x,y,z]=xy^*z$.
Admissible kernels are closed ternary ideals.

\item \emph{Operator systems (not treated here).}
Operator systems do not fit Definition~\ref{def:category-ops} as stated.
Their natural quotient theory is by completely order ideals (kernels of unital
completely positive maps), and an Archimedeanisation step is generally needed to
form the quotient operator system~\cite{KPTT}.  We therefore exclude operator
systems from the present framework.

\item \emph{Banach $A$-modules (encoding caveat).}
Fix a separable Banach algebra $A$ and a countable dense subset
$(a_m)$ of its unit ball.
A left Banach $A$-module may be encoded by the countable family of
unary operations $\omega_m(x)=a_mx$.  We do not use this variant below.

\item \emph{Jordan--Banach algebras (JB or JB$^*$).}
The fundamental operation is the Jordan product
$\omega_1(x,y)=\tfrac12(xy+yx)$; in the JB$^*$ case one adds
$\omega_2(x)=x^*$.
Admissible kernels are closed Jordan ideals.

\item \emph{Banach--Lie algebras.}
$\omega_1(x,y)=[x,y]$ (the Lie bracket).
Admissible kernels are closed Lie ideals.

\item \emph{Operator spaces (encoding caveat).}
Operator spaces are naturally described by a countable family of matrix-norm
predicates, rather than only by operations $X^{n_i}\to X$.  They therefore fall
outside the literal scope of Definition~\ref{def:category-ops}; a multi-sorted
or predicate-enriched version of the framework should cover them.
\end{enumerate}
\end{example}

\begin{remark}\label{rem:BL-separate}
While Banach lattices fit the general quotient-kernel framework at the level of
closed lattice ideals, the descriptive set-theoretic analysis of the natural Polish
parameter spaces of separable Banach lattices involves additional lattice-specific
topological subtleties.  These are treated separately by
M.~Niwi{\'n}ski~\cite{NiwinskiBL}.
\end{remark}

\subsection{Closedness of the admissible-kernel space}

\begin{proposition}\label{prop:ideal-closed}
Let $X\in\mathcal C$ be separable and let $\tau$ be the restriction to
$\SB(X)$ of an admissible topology on $\mathscr F(X)$.
Then $\Ideal_{\mathcal C}(X)$ is $\tau$-closed in $\SB(X)$, hence Polish.
\end{proposition}

\begin{proof}
Suppose $K_n\to K$ in the chosen admissible topology.  The preceding
approximation argument shows that $K$ is a closed linear subspace.  Fix an
operation $\omega$, an ambient tuple $\bar x$, a coordinate $j$, and $h\in K$.
Choose $h_n\in K_n$ with $h_n\to h$.  Since $K_n$ is an admissible kernel,
\[
 \omega(x_1,\ldots,x_j+h_n,\ldots,x_r)-\omega(\bar x)\in K_n.
\]
Continuity and closedness of membership imply the corresponding inclusion
with $h$ and $K$.  Replacing the coordinates successively proves the full
congruence condition for $K$.  Thus the kernel locus is closed and Polish.
\end{proof}

\subsection{Quotient norms and definability}

\begin{definition}\label{def:Phi}
For $x\in X$ and $F\in\Ideal_{\mathcal C}(X)$, define
\[
\Phi(F,x)=d(x,F)=\inf_{z\in F}\|x-z\|.
\]
If $F$ is a subspace, $\Phi(F,x)=\|x+F\|_{X/F}$.
\end{definition}

\begin{lemma}\label{lem:Phi-continuous}
Under the Wijsman topology, for each fixed $x\in X$, the map
$F\mapsto \Phi(F,x)$ is continuous on $\Ideal_{\mathcal C}(X)$.
\end{lemma}

\begin{proof}
Immediate from Definition~\ref{def:wijsman}: Wijsman convergence
$F_n\to F$ means $d(x,F_n)\to d(x,F)$ for all~$x$.
\end{proof}

\begin{lemma}\label{lem:wijsman-subbasic}
Let $X$ be separable and let $\tau_W$ be the Wijsman topology on $\mathscr F(X)$.
For $x\in X$ and $r\in\Q$ the sets
\[
U(x,r):=\{F:\ d(x,F)<r\},\qquad
V(x,r):=\{F:\ d(x,F)\le r\}
\]
are $\tau_W$-open and $\tau_W$-closed respectively.
If $F\subseteq X$ is a closed subspace, then $d(x,F)=\Phi(F,x)$.
\end{lemma}

\begin{proof}
For each fixed $x$, the map $F\mapsto d(x,F)$ is continuous (Definition~\ref{def:wijsman}).
Thus $U(x,r)$ is the preimage of the open ray $(-\infty,r)$ and $V(x,r)$ is the
preimage of the closed ray $(-\infty,r]$.
\end{proof}

\begin{lemma}\label{lem:dense-op-closed}
Let $X\in\mathcal C$ be separable with countably many
operations $\omega_i$ of finite arity.  There exists a countable dense subset
$D\subseteq X$ closed under all operations $\omega_i$ and under rational
linear combinations.
\end{lemma}

\begin{proof}
Start with any countable dense $D_0$.  Define inductively $D_{n+1}$ to be the
set obtained from $D_n$ by closing under all operations~$\omega_i$ applied to tuples
from~$D_n$ and under rational linear combinations.  Put $D=\bigcup_n D_n$.
\end{proof}

\begin{definition}\label{def:projective}
An object $P\in\mathcal C$ is a \emph{separable quotient generator} if
$P$ is separable and every separable object $Y\in\mathcal C$ is isomorphic to
$P/K$ for some $K\in\Ideal_{\mathcal C}(P)$.
\end{definition}

\begin{remark}\label{rem:terminology}
In the categorical literature, a \emph{projective generator} additionally satisfies
a lifting property with respect to epimorphisms.  We do not require such a lifting
property; our arguments use only the quotient-universality encoded in
Definition~\ref{def:projective}.  We therefore use the more descriptive term
`separable quotient generator' throughout.
\end{remark}

\subsection{Main theorem and definability scheme}

\begin{mainthm}\label{thm:continuous-main}
Let $\mathcal{C}$ be a~concrete category with continuous operations admitting a
separable quotient generator~$P$.  Fix:
\begin{itemize}
\item a countable dense subset $D^P=\{u_m:m\in\N\}\subseteq P$ closed under all
operations~$\omega_i$ and rational (or complex-rational, depending on the ground field) linear combinations;
\item the restriction $\tau$ to $\SB(P)$ of an admissible topology on
$\mathscr F(P)$ (e.g.\ Wijsman).
\end{itemize}
Then:
\begin{enumerate}
\item \textbf{(Polishness.)}
$\Ideal_{\mathcal C}(P)$ is a Polish space with the topology~$\tau$.

\item \textbf{(Borel encoding.)}
Define a code $c(K)\in\R^\N$ for $K\in\Ideal_{\mathcal C}(P)$ by
concatenating all values
\[
\Phi\Bigl(K,\sum_{j=1}^\ell q_j u_{m_j}\Bigr)
\quad\text{and}\quad
\Phi\bigl(K,\omega_i(u_{m_1},\dots,u_{m_{n_i}})-u_r\bigr),
\]
enumerated over all rational data.  For any admissible topology~$\tau$,
each coordinate $K\mapsto\Phi(K,x)=d(x,K)$ is upper semicontinuous:
its strict sublevel set at $r>0$ is the open hit set $E^+(B(x,r))$.
It is therefore Borel, and
$c:\Ideal_{\mathcal C}(P)\to\R^\N$ is Borel.
Under the \emph{Wijsman} topology, each coordinate is continuous.

\item \textbf{(Faithfulness.)}
If $c(K)=c(L)$, then $K=L$.  In particular, $P/K\cong P/L$ as objects in~$\mathcal C$
via the unique isometric isomorphism sending $u_m+K$ to $u_m+L$.
(\emph{Proof:} $c(K)=c(L)$ implies $\Phi(K,u_m)=\Phi(L,u_m)$ for all~$m$, hence
$d(\cdot,K)=d(\cdot,L)$ on a dense set, hence everywhere, so $K=L$.)

\item \textbf{(Surjectivity.)}
Every separable object $Y\in\mathcal C$ is isomorphic to $P/K$ for some
$K\in\Ideal_{\mathcal C}(P)$.  In particular, $c$ gives a Borel coding of
quotient presentations of separable objects in~$\mathcal C$.
We do \emph{not} claim injectivity on isomorphism classes: distinct kernels may
still yield isomorphic quotients.

\item \textbf{(Definability scheme.)}
Any property $\mathcal P$ of $P/K$ expressible by a formula built from:
\begin{itemize}
\item atomic predicates of the form
$F(\Phi(K,x_1),\ldots,\Phi(K,x_m))\bowtie r$,
where $x_j\in D^P$, $F:\R^m\to\R$ is continuous,
$\bowtie\in\{<,\le,=,\ge,>\}$, and $r\in\Q$;
\item Boolean connectives;
\item countable quantification over variables in $D^P$ or $\N$;
\end{itemize}
defines a Borel subset of $\Ideal_{\mathcal C}(P)$ for \emph{any} admissible
topology~$\tau$.  Under the \emph{Wijsman} topology specifically,
the Borel rank is bounded by the quantifier alternation depth:
$\forall$ alone gives~$\Pi^0_2$,
$\forall\exists$ gives~$\Pi^0_3$,
$\forall\exists\forall$ gives~$\Pi^0_4$, and so on.
\end{enumerate}
\end{mainthm}

\begin{proof}
Part~(1) is Proposition~\ref{prop:ideal-closed}.
For~(2), each coordinate map $K\mapsto d(x,K)$ is upper semicontinuous for any
admissible topology by openness of hit sets, hence Borel; under the Wijsman topology it is continuous by
Lemma~\ref{lem:Phi-continuous}.
Part~(3): as argued in the statement, $c(K)=c(L)$ forces $K=L$.
Part~(4) holds because $P$ is a quotient generator.
Part~(5): for any admissible $\tau$, each atomic predicate is Borel
(since $K\mapsto\Phi(K,x)$ is u.s.c.\ by~(2)),
Boolean combinations and countable quantification preserve Borelness.
Under the Wijsman topology, each $\Phi$-coordinate is continuous,
so each atomic predicate is open (for~$<,>$)
or closed (for~$\le,\ge,=$), and countable
quantification over a countable set raises the Borel rank by at most one per alternation.
\end{proof}

\begin{definition}\label{def:stable-rel}
Fix $m\in\N$ and a finite family of $\mathcal C$-terms
$t_1(\bar x),\dots,t_r(\bar x)$ in variables $\bar x=(x_1,\dots,x_m)$.
For an object $Y\in\mathcal C$ and $\bar y\in Y^m$ set the \emph{defect}
\[
\mathrm{def}_{\bar t}^Y(\bar y):=\max_{\ell\le r}\|t_\ell^Y(\bar y)\|.
\]
Given a positive integer $M$, we say that the relation
\[
\Sigma_{\bar t,M}(\bar x):\ \ \bigl(\|x_j\|\le M\ (1\le j\le m)\bigr)\ \wedge\
\bigl(t_1(\bar x)=\cdots=t_r(\bar x)=0\bigr)
\]
is \emph{stable on quotients of $P$} if for every $\varepsilon>0$ there exists
$\delta>0$ such that for every $K\in\Ideal_{\mathcal C}(P)$ and every
$\bar y\in (P/K)^m$ with $\|y_j\|\le M$,
\[
\mathrm{def}_{\bar t}^{P/K}(\bar y)<\delta\ \Longrightarrow\
\exists\bar z\in(P/K)^m:
\ \Sigma_{\bar t,M}(\bar z)\ \ \text{and}\ \
\max_j\|y_j-z_j\|<\varepsilon.
\]
\end{definition}

\begin{lemma}\label{lem:stable-exists-dense}
Assume $\Sigma_{\bar t,M}$ is stable on quotients of $P$.
Fix $\varepsilon=1$ and let $\delta>0$ be the corresponding stability constant.
Then for every $K\in\Ideal_{\mathcal C}(P)$ the following are equivalent:
\begin{enumerate}
\item there exists $\bar x\in (P/K)^m$ with $\Sigma_{\bar t,M}(\bar x)$;
\item there exists $\bar u\in (D^P)^m$ such that
$\max_j\Phi(K,u_j)< M$ and $\mathrm{def}_{\bar t}^{P/K}(\bar u+K)<\delta$.
\end{enumerate}
In particular, the set of $K$ satisfying $\exists\bar x\,\Sigma_{\bar t,M}(\bar x)$
is \emph{open} in the Wijsman topology on $\Ideal_{\mathcal C}(P)$.
\end{lemma}

\begin{proof}
(1)$\Rightarrow$(2).
Let $\bar x\in (P/K)^m$ be an exact solution with $\|x_j\|\le M$.
By continuity of the finitely many term maps $t_\ell^{P/K}$ at $\bar x$,
choose $0<\eta<M$ such that
\[
\max_j\|y_j-x_j\|<\eta \ \Longrightarrow\ 
\mathrm{def}_{\bar t}^{P/K}(\bar y)<\delta.
\]
Pick $\bar v\in(D^P)^m$ with $\max_j\|v_j+K-x_j\|<\eta/4$.
Then for each $j$,
\[
\Phi(K,v_j)=\|v_j+K\|\le \|x_j\|+\eta/4\le M+\eta/4.
\]

For each $j$, if $\Phi(K,v_j)<M$ set $u_j=v_j$.
Otherwise $M\le \Phi(K,v_j)\le M+\eta/4$.
Consider the open interval
\[
\Bigl(1-\frac{\eta/2}{\Phi(K,v_j)},\,\frac{M}{\Phi(K,v_j)}\Bigr).
\]
Its length equals
\[
\frac{M-\Phi(K,v_j)+\eta/2}{\Phi(K,v_j)}
\ \ge\ \frac{M-(M+\eta/4)+\eta/2}{\Phi(K,v_j)}
\ =\ \frac{\eta/4}{\Phi(K,v_j)}\ >\ 0,
\]
so the interval is non-empty. Choose a rational $\lambda_j$ in this interval and set
$u_j=\lambda_j v_j\in D^P$ (using closure of $D^P$ under rational scalars).
Then $\Phi(K,u_j)=\lambda_j\Phi(K,v_j)<M$ and
\[
\|u_j+K-(v_j+K)\|=|1-\lambda_j|\Phi(K,v_j)<\eta/2.
\]
Hence
\[
\|u_j+K-x_j\|\le \|u_j+K-(v_j+K)\|+\|v_j+K-x_j\|
<\eta/2+\eta/4<\eta,
\]
and therefore $\mathrm{def}_{\bar t}^{P/K}(\bar u+K)<\delta$.

(2)$\Rightarrow$(1).
Let $\bar u\in(D^P)^m$ satisfy the inequalities in~(2), and set $\bar y=\bar u+K$.
Then $\|y_j\|=\Phi(K,u_j)<M$ and $\mathrm{def}_{\bar t}^{P/K}(\bar y)<\delta$.
By stability there exists $\bar z$ with $\Sigma_{\bar t,M}(\bar z)$.

Finally, for fixed $\bar u$ the conditions
$\max_j\Phi(K,u_j)<M$ and $\mathrm{def}_{\bar t}^{P/K}(\bar u+K)<\delta$
are strict inequalities of continuous functions of $K$, hence open.
Taking the union over countably many $\bar u\in(D^P)^m$ yields openness.
\end{proof}

\begin{theorem}\label{thm:definability-rank}
Work in the setting of Theorem~\ref{thm:continuous-main} with the Wijsman topology.
Let $\mathcal L$ be the language generated by:

\begin{enumerate}[label=\textup{(\alph*)}]
\item \emph{quotient-norm atoms:}
predicates of the form
\[
F\bigl(\Phi(K,t_1),\dots,\Phi(K,t_m)\bigr)\ \bowtie\ r,
\]
where $t_j$ are $\mathcal C$-terms over $D^P$, $F:\R^m\to\R$ is continuous,
$r\in\Q$, and $\bowtie\in\{<,\le,=,\ge,>\}$;

\item \emph{stable-existence atoms:}
for each stable relation $\Sigma_{\bar t,M}$ (Definition~\ref{def:stable-rel}),
a predicate asserting the existence of a solution in $P/K$:
\[
\exists \bar x\ \Sigma_{\bar t,M}(\bar x).
\]
Equivalently (Lemma~\ref{lem:stable-exists-dense}), one may use a fixed rational
tolerance $\delta>0$ and quantify only over tuples from $D^P$.
\end{enumerate}

Build formulae from these atoms using Boolean connectives and countable quantifiers
over $D^P$ and~$\N$.

Then for every such formula $\varphi$, the satisfaction set
\[
[\varphi]=\{K\in\Ideal_{\mathcal C}(P):\ P/K\models\varphi\}
\]
is Borel in $(\Ideal_{\mathcal C}(P),\mathrm{Wijsman})$.
Moreover, if $\varphi$ is in prenex normal form with $s$ alternations, then the
following uniform upper bound holds:
$[\varphi]\in\Sigma^0_{2+s}$ or $[\varphi]\in\Pi^0_{2+s}$ according to the leading
quantifier block (with the convention that stable-existence atoms count as
\emph{open} atoms in the rank bookkeeping).  The estimate is deliberately an upper
bound; closed atomic predicates, such as equalities, often lower the actual rank.
\end{theorem}

\begin{proof}
By Lemma~\ref{lem:Phi-continuous}, each map $K\mapsto\Phi(K,t)$ is continuous.
Quotient-norm atoms of type~(a) are therefore either open (for~$<,>$) or closed
(for~$\le,\ge,=$) by Lemma~\ref{lem:wijsman-subbasic}.
Stable-existence atoms of type~(b) are open by
Lemma~\ref{lem:stable-exists-dense}.
Boolean combinations preserve Borelness.
Countable $\exists$ corresponds to a countable union, countable $\forall$ to a
countable intersection.  Each alternation raises the Borel rank by at most one.
\end{proof}

\begin{remark}\label{rem:dense-method-scope}
Theorem~\ref{thm:definability-rank} applies to properties that can be expressed using:
\begin{itemize}
\item norms of \emph{named terms} in the quotient (via $\Phi$), and
\item existential assertions whose defining relations are \emph{stable} on bounded balls,
so that exact witnesses can be replaced by approximate witnesses from the fixed
countable dense set $D^P$ (Lemma~\ref{lem:stable-exists-dense}).
\end{itemize}
This applies, for instance, to the stable relations for projections, partial
isometries and matrix units, and hence to suitable finite-dimensional approximation
predicates; see Lemma~\ref{lem:stable-fd} and Remark~\ref{rem:stable-relations}.
Properties involving completely positive maps or isomorphism witnesses require the
additional coding arguments given in the later sections; they are not automatic
consequences of this theorem.

By contrast, properties whose definitions require \emph{unbounded} existential
quantification without a uniform stability modulus (\emph{e.g.}\ ring-theoretic phenomena
with unbounded witnesses in reduced products) or that intrinsically quantify over
\emph{uncountable} parameter sets may fall outside this scheme.
\end{remark}

\begin{remark}\label{rem:upper-bounds}
The ranks given in Theorem~\ref{thm:definability-rank} are \emph{uniform upper bounds}
that accommodate both open and closed atomic predicates.  In many concrete instances
the actual Borel class is lower, because particular atoms may be open (strict
inequalities) rather than merely closed (equalities or non-strict inequalities).
Throughout this paper, each stated complexity estimate should be read as a safe upper
bound unless explicitly accompanied by a matching hardness or completeness result.

When we write ``closed (so in particular $\Pi^0_2$)'', we mean the set is $\Pi^0_1$ (closed) and
therefore \emph{a fortiori} $\Pi^0_2$; we sometimes report $\Pi^0_2$ in
such cases because our bookkeeping is geared to quantifier alternation
starting at the $\Pi^0_2$ level, and we do not always optimise the base level.
\end{remark}

\begin{remark}\label{rem:encoding}
The code $c(K)$ captures the \emph{norm structure} of $P/K$ via quotient norms of
all rational combinations, and the \emph{operations} via
$\Phi(K,\omega_i(u_{m_1},\ldots)-u_r)$ (measuring when operations take specific values).
Since $D^P$ is dense and operation-closed, this information determines $P/K$ up to
isometric isomorphism.
\end{remark}

\section{Main theorem for countable algebraic structures}\label{sec:countable}

\subsection{Congruences on free algebras}

Let $L=\{f_i:i\in I\}$ be a countable finitary signature and let $F_\infty$ be the
free $L$-algebra on generators $\{x_n:n\in\N\}$.

\begin{theorem}
\label{thm:free-objects-closed}
Let $\mathcal V$ be a finitary algebraic variety over a countable signature
$L=\{f_i:i\in I\}$.  Let $F_\infty$ be the free $L$-algebra on the countable set of
generators $\{x_n:n\in\N\}$.  Consider the following Polish spaces, all with the
product (Cantor) topology:

\begin{enumerate}
\item \textbf{Congruences (general case).}
$\Con_L(F_\infty)\subseteq 2^{F_\infty\times F_\infty}$, the family of
$L$-congruences on~$F_\infty$.

\item \textbf{Groups.}
$\NSub(F_\infty)\subseteq 2^{F_\infty}$, the family of normal subgroups of the
free group on countably many generators, $F_\infty=\langle x_n:n\in\N\rangle$.

\item \textbf{Rings.}
$\Id(R_\infty)\subseteq 2^{R_\infty}$, the family of two-sided ideals in the
free unital ring generated by countably many free variables, $R_\infty=\Z\langle X_n:n\in\N\rangle$.

\item \textbf{Abelian groups.}
$\Sub(\Z^{(\N)})\subseteq 2^{\Z^{(\N)}}$, the family of subgroups of
$\Z^{(\N)}=\bigoplus_{n\in\N}\Z e_n$.

\item \textbf{Lattices.}
$\Con_{\wedge,\vee}(L_\infty)\subseteq 2^{L_\infty\times L_\infty}$, the family of
lattice congruences on the free lattice~$L_\infty$.
\end{enumerate}

Then each of the above parameter spaces is a \emph{closed} subset of the ambient
Cantor cube, hence compact, zero-dimensional, and Polish.
\end{theorem}

\begin{proof}
We treat~(1) in full generality; the remaining cases are specialisations.

A congruence $R\subseteq F_\infty\times F_\infty$ must be an equivalence relation
compatible with each operation.  Reflexivity, symmetry, and transitivity are
universal Horn conditions: for instance, transitivity states that for all $a,b,c$,
if $(a,b)\in R$ and $(b,c)\in R$ then $(a,c)\in R$; the violation for a fixed
triple $(a,b,c)$ forms the clopen cylinder
$\{R:(a,b)\in R,\ (b,c)\in R,\ (a,c)\notin R\}$.
Taking complements and intersecting over all triples yields a closed set.

For $L$-compatibility, fix an operation $f_i$ of arity~$n_i$ and tuples
$\bar a,\bar b\in F_\infty^{n_i}$.  The violating set
$\{R:\bigwedge_j (a_j,b_j)\in R\text{ but }(f_i(\bar a),f_i(\bar b))\notin R\}$
depends on finitely many coordinates and is clopen.  Taking complements and
intersecting gives a closed set.

Thus $\Con_L(F_\infty)$ is a countable intersection of closed sets, hence closed
in $2^{F_\infty\times F_\infty}$.  By compactness of the Cantor cube,
$\Con_L(F_\infty)$ is compact Polish.
\end{proof}

\begin{remark}\label{rem:abelian-closed}
In the group case the subclass of \emph{abelian groups} is the closed subspace
$\{N\in\NSub(F_\infty):[F_\infty,F_\infty]\subseteq N\}$
$=\bigcap_{g,h\in F_\infty}\{N:[g,h]\in N\}$,
and in the ring case \emph{commutative rings} form the closed subspace
$\{I\in\Id(R_\infty):\forall a,b,\ ab-ba\in I\}$.
\end{remark}

\subsection{Canonical encoding of quotients on a fixed domain}

\begin{definition}\label{def:canonical-encoding}
For a subobject~$S$ of the appropriate kind (congruence, normal subgroup, ideal, or
subgroup) such that $F_\infty/S$ is \emph{countably infinite},
we encode the quotient
canonically on the fixed domain~$\N$ as follows.  For each $a\in F_\infty$, let
$[a]_S$ be its $S$-coset or $S$-class.  Define
$\Min_S=\{\min([a]_S):a\in F_\infty\}\subseteq\N$,
the set of minimal representatives.
Since $F_\infty/S$ is infinite, $\Min_S$ is infinite;
let $\rho_S:\N\to\Min_S$ enumerate $\Min_S$ increasingly and let $q_S:F_\infty\to\N$
map each element to the index of its class representative.  For each basic operation
$f_i$ of arity $n_i$, define
\[
f_i^S(k_1,\dots,k_{n_i})
=q_S\bigl(f_i^{F_\infty}(\rho_S(k_1),\dots,\rho_S(k_{n_i}))\bigr).
\]
Then $\Theta(S)=(f_i^S)_{i\in I}$ is an $L$-structure on $\N$ isomorphic
to $F_\infty/S$.

\smallskip\noindent
\emph{Finite quotients.}\
If $F_\infty/S$ is finite (with $n$~elements, say), $\Min_S$ has exactly
$n$~elements and $\rho_S$ is not defined as a map $\N\to\Min_S$.
The canonical encoding $\Theta$ is therefore restricted to
infinite quotients; finite quotients are coded directly
by~$S$ (which determines the quotient structure up to isomorphism).
\end{definition}

\begin{lemma}\label{lem:rho-q-Borel}
On the $G_\delta$ subspace
$\mathcal S_\infty=\{S\in\mathcal S:F_\infty/S\text{ is infinite}\}$, the maps
$S\mapsto\rho_S$ and $S\mapsto q_S$ are continuous, and so is $\Theta$.
Conversely, the map sending an $L$-algebra in $\mathcal V$ on $\N$ to the kernel of the
evaluation homomorphism $x_n\mapsto n$ is continuous.
\end{lemma}
\begin{proof}
For fixed $k,a\in\N$, the equality $\rho_S(k)=a$ is determined by the
equivalence relation induced by $S$ on the finite set $\{0,\ldots,a\}$:
$a$ must be the $(k+1)$st least class representative.
Each of these finitely many equivalence tests is clopen.  Thus every
coordinate of $\rho_S$ is locally constant.  Likewise, $q_S(a)$ is
determined by the induced equivalence relation on $\{0,\ldots,a\}$,
and hence is locally constant.  The formula defining $\Theta$ now
shows continuity of each of its operation coordinates.

For the converse, evaluation of a fixed term in a structure on $\N$
is continuous, by induction on the construction of the term.  Equality
of the values of two fixed terms is therefore clopen.  These equality
tests are precisely the coordinates of the evaluation kernel.
\end{proof}

\begin{mainthm}\label{thm:countable-main}
Let $\mathcal V$ and $F_\infty$ be as above, and let $\mathcal S$ denote any of the
closed coding spaces from Theorem~\ref{thm:free-objects-closed}.
Set
\[
\mathcal S_\infty:=\{S\in\mathcal S:\ F_\infty/S\text{ is infinite}\}.
\]
Then:
\begin{enumerate}
\item $\mathcal S$ is a compact zero-dimensional Polish space, and $\mathcal S_\infty$
is a $G_\delta$ subspace of $\mathcal S$ (hence Polish).

\item The map
\[
\Theta:\mathcal S_\infty\to \mathrm{Str}_L(\N)=\prod_{i\in I}\N^{\N^{n_i}},
\qquad
S\mapsto (f_i^S)_{i\in I},
\]
is continuous.

\item The induced map from $\mathcal S_\infty$ to the set of isomorphism classes
of countably infinite $L$-algebras in~$\mathcal V$ is surjective.  Equivalently,
for every countably infinite $\mathbf A\in\mathcal V$ there exists
$S\in\mathcal S_\infty$ with $\Theta(S)\cong \mathbf A$.

\item For each first-order sentence $\varphi$ in the language~$L$, the set
\[
\{S\in\mathcal S:\ F_\infty/S\models\varphi\}
\]
is Borel in~$\mathcal S$.
Quantifier-free sentences yield clopen sets; one leading existential quantifier
yields open; one leading universal quantifier yields closed; alternations give the
expected Borel ranks.
\end{enumerate}
\end{mainthm}

\begin{proof}
Part~(1): $\mathcal S$ is compact Polish by Theorem~\ref{thm:free-objects-closed}.
Moreover,
\[
\mathcal S_\infty
=\bigcap_{n\in\N}\{S:\ F_\infty/S\text{ has at least }n\text{ classes}\}
\]
is a countable intersection of open sets, hence $G_\delta$.

Part~(2) is Lemma~\ref{lem:rho-q-Borel}.
Part~(3) holds by the universal property of the free object.
Part~(4): atomic formulae are decided by finitely many coordinates of~$S$,
hence clopen.  Existential quantification adds a countable union; universal
quantification adds a countable intersection; further alternations raise the Borel rank accordingly.
\end{proof}

\begin{remark}\label{rem:standard-countable-coding}
For a countable finitary signature $L=\{f_i:i\in I\}$, the usual
space of structures with domain $\N$ is
\[
 \mathrm{Str}_L(\N)=\prod_{i\in I}\N^{\N^{n_i}},
\]
with the product of the discrete topologies.  Let $X_{\mathcal V}$
be its closed subspace consisting of the algebras in $\mathcal V$.
For background on this coding and its isomorphism relation, see
\cite{BeckerKechris,FriedmanStanley,Hjorth,GaoDST}.

Lemma~\ref{lem:rho-q-Borel} gives continuous maps
\[
 \Theta:\mathcal S_\infty\longrightarrow X_{\mathcal V},
 \qquad
 \Psi:X_{\mathcal V}\longrightarrow\mathcal S_\infty,
\]
where $\Psi(\mathbf A)$ is the kernel of the evaluation
homomorphism $x_n\mapsto n$.  Both maps preserve the isomorphism
type of the represented quotient.  Consequently, for an
isomorphism-invariant class $\mathcal C$ of countably infinite
algebras in $\mathcal V$, its locus in $X_{\mathcal V}$ and its
locus $C_\infty$ in $\mathcal S_\infty$ continuously reduce
to one another.  Their Borel ranks agree \emph{relative to these
two parameter spaces}.

This does not assert rank preservation when $C_\infty$ is regarded
as a subset of the full compact space $\mathcal S$.
Since $\mathcal S_\infty$ is $G_\delta$, relative
$\Pi^0_n$ classes remain $\Pi^0_n$ in $\mathcal S$ for
$n\geqslant2$, and relative $\Sigma^0_n$ classes remain
$\Sigma^0_n$ for $n\geqslant3$.
Indeed, a relative class is the intersection of
$\mathcal S_\infty$ with an ambient class of the same pointclass;
the stated conclusions follow from closure under intersections
with a $\Pi^0_2$ set.  A relative $\Sigma^0_2$ class need only
be $\Sigma^0_3$ in the ambient space.
For example, in the group coding the whole infinite-quotient
locus is relatively clopen, but is $\Pi^0_2$-complete in
$\NSub(F_\infty)$ by
Theorem~\ref{thm:finite-sigma2-complete}, and hence is not
$\Sigma^0_2$ there.  Finite quotients are absent from the
domain-$\N$ coding and must be handled separately whenever
a property is considered on all of $\mathcal S$.
\end{remark}

\begin{remark}\label{rem:disc-vs-cont}
For rings, commutativity is \emph{closed} (hence~$\Pi^0_1$):
\[
        \forall a,b\quad (ab,ba)\in\theta .
\]
For Banach algebras the corresponding condition
$\forall a,b\in D^A,\ \Phi(K,ab-ba)=0$
is closed: each norm-zero predicate is closed by continuity of
$K\mapsto\Phi(K,ab-ba)$, and their countable intersection remains closed.
\end{remark}

\section{Examples of quotient generators}\label{sec:examples}

\begin{remark}\label{rem:proj-gens}
We list standard choices of~$P$; all are separable and surject onto every separable
object in the category via a morphism whose kernel is the appropriate admissible ideal.

\begin{enumerate}[label=(\alph*), itemsep=.4em]
\item \emph{Banach algebras.}
Every separable Banach algebra is generated by a sequence of elements of norm at
most~$1$.  By linearising and completing the corresponding free object, one obtains a
projectively universal quotient generator.  Specifically, for the category of
\emph{non-unital} Banach algebras, let $S$ be the free semigroup on countably many
generators and take $P=\ell_1(S)$.  For the category of \emph{unital} Banach algebras,
one must instead take $S$ to be the free monoid on countably many generators (so the
empty word acts as the multiplicative identity), yielding the unital quotient generator
$P=\ell_1(S)$.  The fact that these free objects act as universally free quotient
generators follows from the free construction; see the author and
Balcerzak~\cite{BalcerzakKania2023}.  Choosing the images of the generators dense
in the open unit ball gives a metric quotient: every target vector has an
absolutely summable expansion in these generators with total coefficient norm
arbitrarily close to its norm, by successive approximation.  Thus the quotient
coding also represents the given norm, as needed for norm-sensitive properties.

\item \emph{$C^*$-algebras.}
Let $F_\infty$ be the free group on countably many generators and take
$P=C^*_{\max}(F_\infty)$ \cite[{\S}V.2.1]{Blackadar}.
This generator covers unital separable $C^*$-algebras.  Stabilisation produces
only stable quotients and is not a generator for arbitrary non-unital algebras.

\item \emph{Banach lattices.}
Take $P=\mathrm{FBL}(\ell_1)$; see \cite{FreeBLl1}.

\item \emph{TROs.}
A free separable TRO exists (Section~\ref{sec:TRO}) and is a quotient generator.

\end{enumerate}
\end{remark}

\section{Banach algebras: complexity bounds}\label{sec:banach-complexity}

Let $A$ be a fixed separable complex Banach-algebra quotient generator and work
on $\Ideal(A)$ with the Wijsman topology.  Fix a countable dense
$\Q(i)$-subalgebra $D^A$.  For statements involving a unit, take the unital
generator with $\|1\|=1$ and require $1\in D^A$.

\begin{remark}\label{rem:rank-convention}
All admissible topologies on a given ideal space generate the same standard Borel
$\sigma$-algebra~\cite{GodefroySaintRaymond2018}.  In particular, the statement
``property $\mathcal P$ is Borel'' is \emph{topology-independent}: it holds for
one admissible topology iff it holds for all.
However, the \emph{Borel rank} --- the precise level $\Sigma^0_\alpha$ or
$\Pi^0_\alpha$ at which $\mathcal P$ sits --- depends on the choice of topology,
because it involves the interplay between open and closed sets, which can differ
from one admissible topology to another.

Unless explicitly stated otherwise, all references to specific Borel classes
$\Sigma^0_\alpha$, $\Pi^0_\alpha$ throughout
Sections~\ref{sec:banach-complexity}--\ref{sec:Cstar-properties} are with respect to
the Wijsman topology on the relevant ideal space $\Ideal(P)$.  The Wijsman topology
is the natural choice for rank computations because it makes the quotient-norm
functional $\Phi$ continuous (not merely Borel), so that strict inequalities in
$\Phi$ define open sets and non-strict inequalities define closed sets.
\end{remark}

\begin{convention}\label{conv:proper-quotients}
Whenever a property below explicitly involves the unit
(topological stable rank, direct finiteness, stable finiteness, simplicity, tracial
states, and so on), we tacitly restrict to proper ideals $K\subsetneq A$.
Equivalently, we ignore the singleton corresponding to the zero quotient.
The proper-ideal locus is clopen: $\|1+K\|$ equals $1$ for proper ideals and
$0$ for $K=A$.  Thus this restriction preserves the stated upper bounds.
\end{convention}

\subsection{Commutativity and uniformity}

\begin{proposition}\label{prop:comm}
The set of $K$ such that $A/K$ is commutative is closed in the Wijsman topology.
\end{proposition}

\begin{proof}
$A/K$ is commutative if and only if $\Phi(K,uv-vu)=0$ for all $u,v\in D^A$.
Each condition is closed and we take a countable intersection.
\end{proof}

\begin{proposition}\label{prop:unif}
The set of $K$ such that $A/K$ is an abstract uniform algebra is closed in the Wijsman
topology.
\end{proposition}

\begin{proof}
Here ``abstract uniform algebra'' means a Banach algebra whose norm satisfies
\[
        \|z^2\|=\|z\|^2 \qquad(z\in B).
\]
Equivalently, by the standard abstract characterisation of uniform algebras, such a
Banach algebra is isometrically isomorphic to a closed subalgebra of a commutative
$C(K)$ with the supremum norm.  Thus the descriptive-set-theoretic condition to be
checked is exactly the square-norm identity.

Let
\[
        E_K=\{z\in A/K:\|z^2\|=\|z\|^2\}.
\]
Since multiplication and the norm are continuous in $A/K$, the set $E_K$ is closed in
$A/K$.  Hence the identity holds for every element of $A/K$ if and only if it holds on
the dense set $D^A+K$.  For each fixed $u\in D^A$, the condition
\[
        \Phi(K,u^2)=\Phi(K,u)^2
\]
is closed, since both sides are continuous functions of~$K$ in the Wijsman topology.
The desired class is the countable intersection of these closed sets.
\end{proof}

\subsection{Right topological stable rank}

The \emph{topological stable rank} of a unital Banach algebra $B$, denoted $\mathrm{tsr}(B)$,
is the smallest integer $n\ge 1$ such that the set of right-unimodular $n$-tuples
$U_n(B)$ is dense in $B^n$ (see Rieffel~\cite{Rieffel1983}).  For $C^*$-algebras, the
topological stable rank coincides with the Bass stable
rank~\cite{HermanVaserstein1984}.

\begin{proposition}\label{prop:tsr}
Fix $n\in\N$.  The set of $K\in\Ideal(A)$ such that $A/K$ is unital and
$\mathrm{tsr}(A/K)\le n$ is $G_\delta$ in the Wijsman topology.
\end{proposition}

\begin{proof}
Write $B=A/K$.  The condition $\mathrm{tsr}(B)\le n$ says that the right-unimodular
$n$-tuples are dense in~$B^n$.  It is enough to approximate tuples from $(D^A)^n$.
For fixed $\bar a=(a_1,\ldots,a_n)\in(D^A)^n$ and $m\in\N$, let $U_{\bar a,m}$ be
the set of all $K$ for which there exist $\bar b,\bar v\in(D^A)^n$ such that
\[
        \max_j \Phi(K,a_j-b_j)<\frac1m,
        \qquad
        \Phi\Bigl(K,\sum_{j=1}^n b_jv_j-1\Bigr)<1 .
\]
For fixed witnesses this is open, and therefore $U_{\bar a,m}$ is open.  If the second
inequality holds in $B$, then $s=\sum_j(b_j+K)(v_j+K)$ is invertible and
\[
        \sum_j (b_j+K)\bigl((v_j+K)s^{-1}\bigr)=1,
\]
so $\bar b+K$ is right-unimodular.  Conversely, if a right-unimodular tuple
approximating $\bar a+K$ exists, approximate both that tuple and a right inverse by
members of $D^A$ to obtain the displayed strict inequalities.  Hence
\[
        \{K:\mathrm{tsr}(A/K)\le n\}
        =\bigcap_{\bar a\in(D^A)^n}\bigcap_{m\in\N} U_{\bar a,m},
\]
a countable intersection of open sets.
\end{proof}

\subsection{Uniformly open multiplication}

Uniformly open multiplication in Banach algebras has been studied
in connection with topological stable rank and differential
embeddings; see~\cite{DragaKania,KaniaMaslany}.

\begin{definition}\label{def:UAO-UO}
A continuous map $f:X\to Y$ between metric spaces is \emph{uniformly open}
if for every $\varepsilon>0$ there is $\delta>0$ such that
\[
 B_Y(f(x),\delta)\subseteq f(B_X(x,\varepsilon))\qquad(x\in X).
\]
Replacing the image on the right by its closure defines \emph{uniform almost
openness}.  Schauder's lemma states that a continuous uniformly almost open map
from a complete metric space is uniformly open
\cite[Lemma~2.1]{DragaKania}.
\end{definition}

\begin{proposition}\label{prop:UO}
The class of quotients $A/K$ whose multiplication is uniformly open belongs to
$\Sigma^0_3$ in the Wijsman topology.
\end{proposition}

\begin{proof}
For $q\in\Q\cap(0,1]$, let $\mathcal V_q$ consist of the kernels $K$ for which,
for every $a,b,y\in D^A$ and every positive integer $n$, one has
\begin{multline}\label{eq:UAO-scheme}
 \Phi(K,y-ab)>q\quad\text{or}\quad
 \exists a',b'\in D^A:\\
 \Phi(K,a'-a)<1,\quad \Phi(K,b'-b)<1,\quad
 \Phi(K,a'b'-y)<1/n.
\end{multline}
Each condition is open, since all its inequalities are strict.  Thus
$\mathcal V_q$ is $G_\delta$.

Write $B=A/K$.  If multiplication in $B$ is uniformly open, choose an output
radius $\delta>0$ corresponding to input radius $1$ and choose rational
$0<q<\min\{\delta,1\}$.  Whenever $\|y-ab\|\leqslant q$, there are factors
within distance $1$ of $a,b$ whose product is $y$.  Approximating these factors
by the dense set $D^A+K$ establishes~\eqref{eq:UAO-scheme}; hence
$K\in\mathcal V_q$.

Conversely, suppose $K\in\mathcal V_q$.  For $a,b\in B$, put
\[
 F(a,b)=\overline{B_B(a,1)B_B(b,1)}.
\]
The function $(a,b,y)\mapsto\operatorname{dist}(y,F(a,b))$ is continuous.
Indeed, translating factors from balls centred at $(a,b)$ to balls centred at
$(a+h,b+k)$ changes their products by $xk+hy+hk$; this tends to zero uniformly
on the two bounded balls as $(h,k)\to(0,0)$.  The reverse translation gives
the corresponding Hausdorff estimate, and distance to the target point is
$1$-Lipschitz.

Condition~\eqref{eq:UAO-scheme} implies that every dense-set triple with
$\|y-ab\|<q$ satisfies $y\in F(a,b)$.  By the preceding continuity and the
strict inequality, this remains true for arbitrary $a,b,y\in B$.  Therefore
\[
 B_B(ab,q)\subseteq\overline{B_B(a,1)B_B(b,1)}\qquad(a,b\in B).
\]
Homogeneity of multiplication now gives, for every $\varepsilon>0$,
\[
 B_B(ab,q\varepsilon^2)
 \subseteq\overline{B_B(a,\varepsilon)B_B(b,\varepsilon)}.
\]
Thus multiplication is uniformly almost open, and Schauder's lemma gives
uniform openness.  We have proved
\[
 \{K:A/K\text{ has uniformly open multiplication}\}
 =\bigcup_{q\in\Q\cap(0,1]}\mathcal V_q,
\]
which belongs to $\Sigma^0_3$.
\end{proof}

\subsection{Pure infiniteness and norm control}\label{subsec:PI}

Let $B$ be a unital Banach algebra.  Following Daws--Horv\'ath~\cite{DawsHorvath},
define for $a\in B$
\[
C_{\mathrm{pi}}^B(a):=\inf\{\|b\|\,\|c\|:\ b,c\in B,\ bac=1\}\in(0,\infty].
\]
Thus $C_{\mathrm{pi}}^B(a)<\infty$ precisely when $1=bac$ for some $b,c\in B$;
this asks for a single product, rather than a finite sum of such products.

\begin{definition}\label{def:PI-unif}
A unital Banach algebra $B$ is \emph{purely infinite} (in the sense of
\cite{DawsHorvath}) if $B\not\cong\C$ and $C_{\mathrm{pi}}^B(a)<\infty$ for every
non-zero $a\in B$.

We say that $B$ is \emph{uniformly purely infinite} if $B\not\cong\C$ and there exists
$M\in\N$ such that
\[
\sup\{C_{\mathrm{pi}}^B(a):\ \|a\|=1\}\le M.
\]
By \cite[Prop.~2.2]{DawsHorvath}, uniform pure infiniteness is equivalent to pure
infiniteness of every ultrapower $B^{\mathcal U}$ for countably incomplete ultrafilters.
\end{definition}

\begin{proposition}\label{prop:UPI}
Let $A$ be a unital Banach-algebra quotient generator and work on $\Ideal(A)$ with
the Wijsman topology.  The set of $K$ such that $A/K$ is uniformly purely infinite
is $\Sigma^0_3$.
\end{proposition}

\begin{proof}
Write $B=A/K$.  For $N\in\N$ set
\[
r_N:=\frac{1}{8N^2}\in(0,1)
\]
and let $\mathcal U_N\subseteq\Ideal(A)$ be the set of all $K$ such that
\begin{multline}\label{eq:UPI-net}
\forall a\in D^A:\ 
\bigl|\Phi(K,a)-1\bigr|>r_N
\quad\lor\\
\exists b,c\in D^A:\ 
\Phi(K,b)<N,\ \Phi(K,c)<N,\ \Phi(K,bac-1)<\tfrac18.
\end{multline}
For fixed $a\in D^A$, both disjuncts in \eqref{eq:UPI-net} define open subsets of
$\Ideal(A)$, so $\mathcal U_N$ is $G_\delta$, \emph{i.e.}\ $\Pi^0_2$.

On the proper-ideal locus, let
$\mathrm{NonScal}=\{K:A/K\not\cong\C\}$.
This is the open condition $\dim(A/K)\geqslant2$, by
Lemma~\ref{lem:finite-dimension-closed}.

We \emph{claim} that
\[
\mathrm{UPI}=\mathrm{NonScal}\cap\bigcup_{N\in\N}\mathcal U_N.
\]

\smallskip\noindent
($\Rightarrow$)
Assume that $B$ is uniformly purely infinite.
Choose $M\in\N$ with
$\sup\{C_{\mathrm{pi}}^B(x):\ \|x\|=1\}\le M$,
and then choose $N\in\N$ with $N^2>M$.
Fix $a\in D^A$ with $|\|a+K\|-1|\le r_N$ and put
$u:=(a+K)/\|a+K\|$.
Then $\|u\|=1$ and $\|u-(a+K)\|\le r_N$.
Since $C_{\mathrm{pi}}^B(u)\le M<N^2$, there exist $b_0,c_0\in B$ with
$b_0uc_0=1$ and $\|b_0\|\,\|c_0\|<N^2$.  No attainment of the defining
infimum is required.
Rescale by
\[
\lambda:=\Bigl(\frac{\|c_0\|}{\|b_0\|}\Bigr)^{1/2},
\qquad
b:=\lambda b_0,
\qquad
c:=\lambda^{-1}c_0.
\]
Then $buc=1$ and
$\|b\|=\|c\|=\sqrt{\|b_0\|\,\|c_0\|}<N$.
Hence
\[
\|b(a+K)c-1\|
=\|b(a+K-u)c\|
\le \|b\|\,\|a+K-u\|\,\|c\|
< N^2 r_N
=\tfrac18.
\]
Approximating $b,c$ by elements of $D^A$ preserves the strict bounds
$\Phi(K,b)<N$, $\Phi(K,c)<N$, and $\Phi(K,bac-1)<1/8$.
Thus $K\in\mathcal U_N$.
Since $B\not\cong\C$ by definition, $K\in\mathrm{NonScal}$ as well.

\smallskip\noindent
($\Leftarrow$)
Assume $K\in\mathrm{NonScal}\cap\mathcal U_N$.
Let $x\in B$ with $\|x\|=1$ and choose $a\in D^A$ with
$\|a+K-x\|<r_N/2$.
Then $|\|a+K\|-1|\le r_N/2$, so the left disjunct in \eqref{eq:UPI-net} fails.
Hence there exist $b,c\in D^A$ with
$\|b+K\|<N$, $\|c+K\|<N$, and $\|b(a+K)c-1\|<1/8$.
Therefore
\[
\|bxc-1\|
\le \|b(x-a-K)c\|+\|b(a+K)c-1\|
< N^2\cdot\frac{r_N}{2}+\frac18
=\frac{3}{16}
<\frac12.
\]
So $bxc$ is invertible and $\|(bxc)^{-1}\|<2$.
Putting $c':=c(bxc)^{-1}$ gives $bxc'=1$ and
$\|b\|\,\|c'\|\le \|b\|\,\|c\|\,\|(bxc)^{-1}\|< 2N^2$.
Thus $C_{\mathrm{pi}}^B(x)\le 2N^2$ for every $\|x\|=1$.
Since $K\in\mathrm{NonScal}$, we also have $B\not\cong\C$.
Therefore $B$ is uniformly purely infinite.

\medskip\noindent
Finally, $\mathrm{NonScal}$ is open and $\bigcup_N\mathcal U_N$ is
$\Sigma^0_3$, so
$\mathrm{UPI}=\mathrm{NonScal}\cap\bigcup_{N\in\N}\mathcal U_N$
is $\Sigma^0_3$.
\end{proof}

\begin{proposition}\label{prop:PI-coanalytic}
The class of purely infinite quotients of a fixed separable unital Banach
algebra is coanalytic.
\end{proposition}

\begin{proof}
Work with proper ideals and norm-one units.  For $a\in A$, there exist
$b,c\in A/K$ with $b(a+K)c=1$ if and only if there exist $u,v\in D^A$ such that
\[
 \Phi(K,uav-1)<1.
\]
One implication follows by density.  For the other, the element
$(u+K)(a+K)(v+K)$ is invertible; multiplying $v+K$ on the right by its inverse
gives an exact factorisation of $1$.
Consequently, failure of the factorisation condition at some non-zero element
is the projection onto the first coordinate of
\[
 \left\{(K,a):\Phi(K,a)>0\ \text{and}\quad
    \Phi(K,uav-1)\geqslant1\text{ for all }u,v\in D^A\right\}.
\]
This is Borel in $\Ideal(A)\times A$: quotient norms are jointly continuous in
the kernel and the ambient element.  Its projection is analytic.  Excluding
the scalar quotients is a Borel condition, so the purely infinite class is
coanalytic.
\end{proof}

\begin{remark}\label{rem:PI-nonuniform}
The uniform variant has the finite Borel bound established above.  For the
non-uniform class, the preceding argument gives a coanalytic upper bound;
we do not establish Borelness or a matching lower bound here.
\end{remark}

\subsection{Dedekind (in)finiteness and unbounded witnesses}\label{subsec:df-banach}

A unital ring (or Banach algebra) $B$ is \emph{Dedekind finite} (or \emph{directly
finite}) if $ab=1$ implies $ba=1$ for all $a,b\in B$; otherwise $B$ is
\emph{Dedekind infinite}.

In Banach algebras, a central subtlety is that direct finiteness is not controlled by
uniformly bounded fragments: Daws--Horv\'ath show that ring-theoretic (in)finiteness
behaves delicately under reduced products of Banach algebras, and in particular that
the norms of witnesses to Dedekind infiniteness cannot in general be bounded uniformly
along families~\cite{DawsHorvathCJM}.

For complexity purposes it is useful to isolate bounded
\textit{strict} witness fragments.  For $M\in\N$ let $\mathrm{Bad}_M$ be the set of
$K\in\Ideal(A)$ such that, in $B=A/K$, there exist $a,b\in B$ with
\[
        \|a\|<M,\qquad \|b\|<M,\qquad ab=1,\qquad ba\ne 1 .
\]
Then $B$ is Dedekind infinite iff $K\in\bigcup_M\mathrm{Bad}_M$, because any pair of
bounded witnesses has norm strictly smaller than some integer~$M$.

\begin{proposition}\label{prop:df-banach}
The class of Dedekind finite quotients $A/K$ is closed in the Wijsman topology.
Equivalently, the class of Dedekind infinite quotients is open.
\end{proposition}

\begin{proof}
Write $B=A/K$.  For $M\in\N$ say that $B$ has an \emph{open $M$-bounded
Dedekind-infinite witness} if there are $a,b\in B$ with
\[
        \|a\|<M,
        \qquad \|b\|<M,
        \qquad ab=1,
        \qquad ba\ne 1 .
\]
A unital Banach algebra is Dedekind infinite iff it has such a witness for some
integer~$M$.

Fix $M\in\N$.  For $k,n\in\N$ with $M-1/k>0$, set
\[
\delta_{M,k,n}:=\min\Bigl\{\frac12,\frac{1}{2kM},\frac{1}{4nM^2+1}\Bigr\}.
\]
Consider the following open condition on~$K$:
\[
\begin{aligned}
&\exists u,v\in D^A:\quad
\Phi(K,u)<M-\frac1k,\qquad
\Phi(K,v)<M-\frac1k,\\
&\hspace{35mm}
\Phi(K,uv-1)<\delta_{M,k,n},\qquad
\Phi(K,vu-1)>\frac1n .
\end{aligned}
\]
All inequalities are strict and involve Wijsman-continuous functions, so the condition
is open.

We claim that $B$ has an open $M$-bounded Dedekind-infinite witness iff the displayed
condition holds for some $k,n$.  If $a,b$ are exact witnesses with norms $<M$, choose
$k$ such that both norms are $<M-1/k$ and choose $n$ with $\|ba-1\|>3/n$; density of
$D^A+K$ then gives $u,v\in D^A$ satisfying the displayed strict inequalities.
Conversely, suppose the displayed condition holds and put $u_0=u+K$, $v_0=v+K$.
Since $\|u_0v_0-1\|<\delta_{M,k,n}<1$, the element $u_0v_0$ is invertible.  Set
\[
        \widetilde a=u_0,
        \qquad
        \widetilde b=v_0(u_0v_0)^{-1}.
\]
Then $\widetilde a\widetilde b=1$ and, by the choice of $\delta_{M,k,n}$,
\[
        \|\widetilde b\|
        \le (M-1/k)(1-\delta_{M,k,n})^{-1}<M .
\]
Moreover
\[
\|\widetilde b\widetilde a-1\|
\ge \|v_0u_0-1\|-
     \|v_0((u_0v_0)^{-1}-1)u_0\|
> \frac1n- M^2\frac{\delta_{M,k,n}}{1-\delta_{M,k,n}}>0.
\]
Thus $\widetilde a,\widetilde b$ are genuine open $M$-bounded witnesses.

The Dedekind-infinite class is the countable union of the open sets just described,
over $M,k,n$ and $u,v\in D^A$.  Hence it is open, and its complement is closed.
\end{proof}

\begin{remark}\label{rem:df-banach-nontrivial}
The Daws--Horv\'ath analysis~\cite{DawsHorvathCJM} shows that norms of witnesses to
Dedekind infiniteness cannot in general be bounded uniformly along arbitrary families.
The preceding proof avoids any uniform global bound: it writes Dedekind infiniteness as
a countable union of open bounded fragments, using open balls and a Neumann-series
correction inside each fragment.
\end{remark}

\section{\texorpdfstring{$C^*$}{C*}-algebras: decomposition, approximation, and regularity}
\label{sec:Cstar-properties}

Throughout, let $A=C^*_{\max}(F_\infty)$ and fix a countable dense $\Q(i)$-$*$-subalgebra
$D^A$ closed under multiplication and~$*$.  Work on $\Ideal(A)$ with the Wijsman
topology and write $B=A/K$.

\begin{remark}\label{rem:Cstar-unital-scope}
The generator $A=C^*_{\max}(F_\infty)$ parametrises separable unital
$C^*$-algebras; predicates requiring a non-zero unit are understood on the
clopen set of proper ideals.  Indeed, $\|1+K\|$ is $1$ for a proper ideal
and $0$ for $K=A$.  Stabilising this generator does not produce a generator
for arbitrary non-unital algebras: its quotients are stable.  For the
non-unital constructions used below we instead work in the ideal space of
the specified non-unital algebra and use its standard seminorm code.
\end{remark}

\begin{remark}\label{rem:other-codings}
Kechris introduced a standard Borel parametrisation of separable
$C^*$-algebras~\cite{KechrisAsianCstar}.  Farah--Toms--T\"ornquist
constructed explicit Borel translations between several such
parametrisations~\cite{FarahTomsTornquist,FarahTomsTornquistCrelle}.
One is the concrete coding $\Gamma(H)=B(H)^{\N}$, with the standard
Borel structure on $B(H)$ generated by the weak operator topology;
a sequence codes the algebra it generates.  Another records the norms
of rational $*$-polynomials in a generating sequence.  The latter has
unital and non-unital versions, and variants with bounded generators.
We use only translations that preserve the isomorphism type of the
coded algebra.  Such Borel translations transfer Borelness, but need
not preserve a finite Borel rank.  Every rank asserted here is proved
in the stated topology.
\end{remark}

\begin{lemma}
\label{lem:pullback-ideals}
Let $\sigma:C\twoheadrightarrow D$ be a surjective $*$-homomorphism between separable
$C^*$-algebras.  Then the pullback map
\[
\sigma^*:\Ideal(D)\to\Ideal(C),\qquad J\mapsto\sigma^{-1}(J),
\]
is a Wijsman-continuous injection.
\end{lemma}

\begin{proof}
Injectivity: if $J\neq L$ then $\sigma^{-1}(J)\neq\sigma^{-1}(L)$ because $\sigma$ is
surjective.  For Wijsman continuity, fix $c\in C$.  Because $\sigma$ is a surjective
$*$-homomorphism, the quotient $C/\sigma^{-1}(J)\cong D/J$ isometrically (via the
canonical isomorphism), and therefore
\[
\dist\bigl(c,\,\sigma^{-1}(J)\bigr)
=\|\sigma(c)+J\|_{D/J}
=\dist\bigl(\sigma(c),\,J\bigr).
\]
Since $J\mapsto\dist(\sigma(c),J)$ is continuous in the Wijsman topology on $\Ideal(D)$,
it follows that $J\mapsto\dist(c,\sigma^{-1}(J))$ is continuous on $\Ideal(C)$.
As $c$ was arbitrary, $\sigma^*$ is Wijsman-continuous.
\end{proof}

\begin{remark}\label{rem:Xi-ideals}
Let $U$ be the universal unital $C^*$-algebra on countably many
contractions.  Its ideal space is homeomorphic to the bounded unital
seminorm coding obtained by assigning a $C^*$-seminorm, bounded by
the universal norm, to each rational $*$-polynomial.  The homeomorphism
is $I\mapsto(p\mapsto\|p+I\|)$: agreement on a dense polynomial
algebra determines the ideal and, by uniform approximation, determines
Wijsman convergence.  This assertion concerns this bounded coding,
not the unrestricted seminorm space with arbitrarily large generators.

There are fixed surjective unital $*$-homomorphisms $U\to A$ and
$A\to U$.  The first sends the universal contractions to the free
unitaries; the second exists because every separable unital
$C^*$-algebra is a quotient of $A$.  Lemma~\ref{lem:pullback-ideals}
therefore gives quotient-preserving continuous injections between the
two ideal spaces.  Combining the explicit map to seminorms with the
standard Borel translations in~\cite{FarahTomsTornquistCrelle}
provides the compatibility with the standard unital codings needed below.
\end{remark}

\begin{lemma}\label{lem:fixed-ideal-to-Xi}
Let $C$ be a fixed separable $C^*$-algebra and let $(c_j)_{j\in\N}$ be a dense
sequence in the unit ball generating~$C$ as a $C^*$-algebra.  There is a
Wijsman-continuous map from $\Ideal(C)$ to the appropriate standard seminorm coding
$\Xi$ which sends $I\in\Ideal(C)$ to a code for $C/I$.
\end{lemma}

\begin{proof}
If $C$ is non-unital, use the standard non-unital seminorm coding, obtained by
enumerating rational $*$-polynomials without a constant term; this coding is Borel
equivalent to the concrete coding $\Gamma(H)$ used by Farah--Toms--T\"ornquist.  If
$C$ is unital, use the unital version described above.  In either case, for the fixed
enumeration $(p_n)$ of the relevant rational $*$-polynomials, define
\[
        \xi_I(n):=\|p_n(c_0,c_1,\ldots)+I\|_{C/I}.
\]
For each $n$, the element $p_n(c_0,c_1,\ldots)\in C$ is fixed, and the coordinate
$I\mapsto\xi_I(n)$ is continuous in the Wijsman topology.  The resulting seminorm is
precisely the quotient seminorm of the dense $*$-subalgebra generated by the images of
the $c_j$, so the completion coded by $\xi_I$ is $C/I$.
\end{proof}

\begin{remark}
\label{rem:ideal-coding-history}
The ideal-space coding of separable $C^*$-algebras as quotients of a fixed universal
separable $C^*$-algebra should not be presented as historically new.  Nate Brown
suggested this quotient/ideal coding around 2010, as communicated to the author by
Ilijas Farah.  Austin Shiner has independently announced related work on quotient
encodings and functorial complexity~\cite{ShinerForthcoming}, as discussed
in the Introduction.
\end{remark}

\paragraph{Continuous logic.}
It is also natural to compare the present quotient/Wijsman formalism with the
continuous-logic treatment of axiomatisable and definable classes of $C^*$-algebras;
see \cite{FarahHartLupiniRobertTikuisisVignatiWinter}.  After transferring between
our ideal coding and the standard FTT codings, this gives an alternative abstract
route to Borelness for many axiomatisable classes.  We nevertheless retain the
direct arguments here because they keep track of concrete Borel ranks inside the
present parameter space.

\subsection{Coding projections, unitaries, and matrix units}

\begin{lemma}\label{lem:basic-coding}
For $x\in D^A$ and $K\in\Ideal(A)$:
\begin{enumerate}
\item $x+K$ is a projection in $B$ iff $\Phi(K,x^2-x)=0$ and $\Phi(K,x^*-x)=0$.
\item $u+K$ is unitary in $B$ iff $\Phi(K,u^*u-1)=0$ and $\Phi(K,uu^*-1)=0$.
\end{enumerate}
For fixed finite families, the predicate ``forms a system of matrix units'' is closed.
\end{lemma}

\begin{proof}
All assertions follow directly from identities in $B$ expressed by $\Phi$ and
the continuity of $\Phi$ in~$K$.
\end{proof}

\begin{lemma}\label{lem:positivity-coding}
For every fixed $a\in A$, the set $\{K:a+K\geqslant0\}$ is closed.
The set $\{b^*b+K:b\in D^A\}$ is dense in $(A/K)_+$.
\end{lemma}
\begin{proof}
Fix $R\geqslant\max\{1,\|a\|\}$. Positivity is equivalent to
\[
 \Phi(K,a-a^*)=0,\qquad \Phi(K,a-R1)\leqslant R.
\]
For a self-adjoint element of norm at most $R$, the second condition
places its spectrum in $[0,2R]$. Both conditions are closed.
The density assertion follows by approximating positive square roots.
\end{proof}

\begin{remark}
Density of positive squares alone does not justify replacing an arbitrary
universal quantifier over positive elements by a countable quantifier.
The cutdown arguments below supply the needed justification for simplicity
and property~(SP).
\end{remark}

\begin{lemma}\label{lem:stable-fd}
Let $F=\bigoplus_{l=1}^sM_{d_l}$ with its canonical matrix units
$e^{(l)}_{ab}$. For every $\varepsilon>0$ there is $\delta>0$ such that
the following holds. If elements $y^{(l)}_{ab}$ in a unital $C^*$-algebra
satisfy the matrix-unit relations, the orthogonality relations between
distinct summands, and the unit relation
$\sum_{l,a}y^{(l)}_{aa}=1$, all within $\delta$, with
$\|y^{(l)}_{ab}\|<1+\delta$, then there is a unital $*$-homomorphism
$\pi:F\to B$ with
$\|\pi(e^{(l)}_{ab})-y^{(l)}_{ab}\|<\varepsilon$ for all $l,a,b$.
\end{lemma}
\begin{proof}
This is stability of the finite-dimensional relations; see \cite{Loring}.
\end{proof}

\begin{remark}
The homomorphism in Lemma~\ref{lem:stable-fd} may annihilate summands of $F$.
When $F=M_d$ and the target is non-zero, unitality forces injectivity.
For a general finite direct sum, injectivity requires additional lower norm
bounds on the images of the summand units. The approximation arguments below
only require homomorphic images.
\end{remark}

\begin{remark}\label{rem:stable-relations}
Exact projections, matrix units and unitaries in a quotient need not have
representatives in $D^A$. Where stability is available, we instead search
over approximate relations on $D^A$, including the necessary norm bounds,
and correct the witnesses in the quotient. The tolerance must be chosen
after any finite coefficient or polynomial data whose values must survive
the correction. Each such finite search uses strict inequalities in
continuous quotient-norm functions.
\end{remark}

\begin{lemma}\label{lem:ucp-fd}
Fix $m\geqslant1$, $x_1,\ldots,x_r\in D^A$ and
$Y_1,\ldots,Y_r\in M_m$. The set of $K\in\Ideal(A)$ for which there
is a ucp map $\varphi:A/K\to M_m$ with
$\varphi(x_j+K)=Y_j$ for all $j$ is closed.
\end{lemma}

\begin{proof}
The existence of a ucp map $\varphi:A/K\to M_m$ with
$\varphi(x_j+K)=Y_j$ is equivalent to the existence of a ucp map
$\psi:A\to M_m$ with $\psi(x_j)=Y_j$ for $1\le j\le r$ and $\psi$ vanishing on~$K$.
Since any ucp map is contractive, vanishing on $K$ is equivalent to
$\|\psi(a)\| \le \|a+K\|_{A/K} = \Phi(K,a)$ for all $a\in A$; by density
and continuity it suffices to check this inequality for $a\in D^A$.

Let $\mathrm{UCP}(A,M_m)$ denote the space of all ucp maps $A\to M_m$, equipped
with the point-norm topology.
Since $M_m$ is finite-dimensional, the point-norm topology on $\mathrm{UCP}(A,M_m)$
coincides with pointwise convergence of the $m^2$ matrix-entry functionals in the
weak$^*$ topology of $A^*$; hence $\mathrm{UCP}(A,M_m)$ is compact metrisable by
Banach--Alaoglu and separability of $A$.
Consider the constraint on pairs
$(\psi, K) \in \mathrm{UCP}(A,M_m) \times \Ideal(A)$:
\[
\forall j\le r\ \bigl(\psi(x_j)=Y_j\bigr) \ \wedge \ 
\forall a\in D^A\ \bigl(\|\psi(a)\| \le \Phi(K,a)\bigr).
\]
For each fixed $a\in D^A$, the evaluation $\psi \mapsto \|\psi(a)\|$ is continuous,
and $K \mapsto \Phi(K,a)$ is continuous (Lemma~\ref{lem:Phi-continuous}), so
$\|\psi(a)\|\le\Phi(K,a)$ defines a closed subset of the product.
Similarly, each condition $\psi(x_j)=Y_j$ is closed.
A countable intersection of closed sets is closed, so the full constraint set
is closed in $\mathrm{UCP}(A,M_m) \times \Ideal(A)$.
Because $\mathrm{UCP}(A,M_m)$ is compact, the projection of this closed set onto
$\Ideal(A)$ is closed.

\end{proof}

\subsection{Nuclearity and uniform existential definitions}

The continuity requirement in a uniform family is essential to the
passage from all elements to a named dense set. We make the convention
explicit, following
\cite[\S5.7.1 and Definition~5.7.1]{FarahHartLupiniRobertTikuisisVignatiWinter}.
For a tuple $\bar R=(R_1,\ldots,R_k)$ of positive radii, write
\[
 B_{\bar R}=\prod_{i=1}^k\{b\in B:\|b\|\leqslant R_i\},
 \qquad d_\infty(\bar b,\bar c)=\max_i\|b_i-c_i\|.
\]
A \emph{uniform family} $\mathcal F$ on these fixed bounded sorts is
a countable family of non-negative formulas for which there is a single
modulus $\omega$, with $\omega(t)\to0$ as $t\downarrow0$, such that
\[
 |\varphi^B(\bar b)-\varphi^B(\bar c)|
 \leqslant\omega(d_\infty(\bar b,\bar c))
 \qquad(\varphi\in\mathcal F,\ \bar b,\bar c\in B_{\bar R})
\]
in every algebra under consideration. Consequently
$p_{\mathcal F}^B=\inf_{\varphi\in\mathcal F}\varphi^B$ has the same
continuity modulus. All quantifiers also range over specified bounded
sorts. An existential formula has the form
$\inf_{\bar y}\theta(\bar x,\bar y)$ with $\theta$ quantifier-free.
We allow uniform limits of such formulas, as in the definable-predicate
convention of \cite[\S3.1.1]{FarahHartLupiniRobertTikuisisVignatiWinter}.

\begin{lemma}\label{lem:uniform-existential}
Suppose a class $\mathcal C$ of unital $C^*$-algebras is described by
uniform families $\mathcal F_l$ of non-negative existential formulas,
each on a fixed tuple of bounded sorts
$\bar R_l=(R_{l,1},\ldots,R_{l,k_l})$, in the sense that
\[
 B\in\mathcal C\quad\Longleftrightarrow\quad
 \sup_{\bar x\in B_{\bar R_l}}
      \inf_{\varphi\in\mathcal F_l}\varphi^B(\bar x)=0
       \qquad(l\in\N).
\]
Then $\{K:A/K\in\mathcal C\}$ is $G_\delta$ in the Wijsman topology.
\end{lemma}
\begin{proof}
For $a\in D^A$ and $R>0$, define the section of the radius-$R$ ball by
\[
 s_{a,R}(K)=\frac{R(a+K)}{\max\{R,\Phi(K,a)\}}.
\]
These sections form a dense subset of the radius-$R$ ball of every
quotient. Indeed, $D^A+K$ is dense and radial retraction onto that ball
is continuous and fixes each element of the ball. On a product of sorts
we use the corresponding tuples of sections.

Evaluation of a quantifier-free formula on finitely many such sections
is continuous in $K$. Each atomic predicate is a quotient norm of a
polynomial whose scalar coefficients vary continuously with $K$;
joint continuity of the quotient-norm functional therefore applies.
The same statement holds for quantifier-free definable predicates,
because they are uniform limits of quantifier-free formulas.

For $\varphi(\bar x)=\inf_{\bar y}\theta(\bar x,\bar y)$, continuity
in the bounded quantified variables allows the infimum to be taken
over the countable tuples of named sections in their prescribed sorts.
Thus $K\mapsto\varphi^{A/K}(s_{\bar a,\bar R_l}(K))$ is an infimum
of continuous functions and is upper semicontinuous. Uniform limits
of upper semicontinuous real-valued functions are upper semicontinuous,
so this conclusion also covers existential definable predicates.
Taking the infimum over $\mathcal F_l$ preserves upper semicontinuity.

The common free-variable modulus now shows that
$p_{\mathcal F_l}^{A/K}$ vanishes everywhere on its sorts if and only
if it vanishes on the named dense tuples. Consequently the desired
class is
\[
 \bigcap_{l\in\N}\ \bigcap_{m\geqslant1}
 \ \bigcap_{\bar a\in(D^A)^{k_l}}
 \left\{K:
   \inf_{\varphi\in\mathcal F_l}
   \varphi^{A/K}(s_{\bar a,\bar R_l}(K))<1/m\right\}.
\]
Each set in braces is open, and there are only countably many outer
conditions. This proves the assertion on the non-zero quotient codes;
the isolated zero quotient may be added or removed separately.
\end{proof}

\begin{proposition}\label{prop:nuclear}
The class of nuclear quotients is $G_\delta$ in the Wijsman topology.
\end{proposition}
\begin{proof}
By \cite[Theorem~5.7.3(3)]{FarahHartLupiniRobertTikuisisVignatiWinter},
nuclearity is definable by uniform families in the sense of
\cite[Definition~5.7.1]{FarahHartLupiniRobertTikuisisVignatiWinter}.
The final assertion of that theorem states that the families in
items~(1)--(7), hence in item~(3), consist of existential formulas.
Lemma~\ref{lem:uniform-existential} therefore applies. This uses both
existentiality and the common modulus; bare Borel equivalence between
parameter spaces would not give the asserted rank.
\end{proof}

\subsection{AF and related approximation classes}

A $C^*$-algebra $B$ is \emph{approximately finite-dimensional} (\emph{AF}) if it is the
closure of an increasing union of finite-dimensional $C^*$-subalgebras, or
equivalently, if for every finite $F\subseteq B$ and every $\varepsilon>0$ there exists
a finite-dimensional $C^*$-subalgebra $F_0\subseteq B$ approximating $F$
within~$\varepsilon$ (see \cite[II.8.2.2]{BlackadarOA}).

\begin{proposition}\label{prop:AF}
The class $\mathrm{AF}=\{K\in\Ideal(A):A/K\text{ is AF}\}$ is $\Pi^0_2$ in the Wijsman
topology.
\end{proposition}

\begin{proof}
Fix $\bar x=(x_1,\ldots,x_r)\in(D^A)^{<\N}$ and $m\geqslant1$.
A candidate witness consists of a finite-dimensional algebra type $F$,
its canonical matrix units $(e_\nu)_{\nu\leqslant N}$, and rational
coefficient vectors $\lambda_t=(\lambda_{t,\nu})_{\nu\leqslant N}$.
Choose $\eta>0$ so that
$\eta\max_t\sum_\nu|\lambda_{t,\nu}|<1/(4m)$, and then choose
a stability tolerance $\delta$ from Lemma~\ref{lem:stable-fd}.
Search over tuples $(y_\nu)\in(D^A)^N$ satisfying the finite-dimensional
relations, including the unit relation and the norm bounds, within $\delta$,
and satisfying
\[
 \max_{t\leqslant r}\Phi\left(K,x_t-
                  \sum_\nu\lambda_{t,\nu}y_\nu\right)<1/(2m).
\]
For fixed data this is open. After correction, the resulting unital
homomorphism $\pi:F\to A/K$ has finite-dimensional image approximating
$\bar x+K$ within $3/(4m)<1/m$.

Take the countable union over $F$, the coefficient vectors, and the tuples
$(y_\nu)$. Conversely, if a finite-dimensional subalgebra approximates
$\bar x+K$ within $1/(4m)$, enlarge it by the unit if necessary, choose
rational coefficients with a strict approximation margin, and approximate
its matrix units by images of $D^A$. These choices give a witness in the
open union. Intersecting over $\bar x,m$ therefore gives exactly the AF
class, by its local approximation characterisation.
\end{proof}

The same stability argument applies to local approximation by
homomorphic images of fixed building blocks.

\begin{proposition}
\label{prop:AI-AH}
Let $\mathcal C=(C_j)_{j\in\N}$ be a countable family of separable $C^*$-algebras,
and assume that each $C_j$ is given by a stable finite presentation.  Let
$\mathrm{LocIm}(\mathcal C)$ be the class of $K\in\Ideal(A)$ such that $B=A/K$ has the
following property: for every finite $F\subseteq B$ and every $\varepsilon>0$ there are
$j\in\N$ and a $*$-homomorphism $\pi:C_j\to B$ such that
$\dist(x,\pi(C_j))<\varepsilon$ for all $x\in F$.

Then $\mathrm{LocIm}(\mathcal C)$ is $\Pi^0_2$ in the Wijsman topology.
\end{proposition}

\begin{proof}
Fix $\bar x\in(D^A)^{<\N}$ and $m\geqslant1$.
For each block $C_j$ with fixed finite stable presentation, choose
a tuple of rational $*$-polynomials $\bar p=(p_1,\ldots,p_r)$ in its
generators. Choose a stability tolerance sufficiently small that an
approximate representation can be corrected to a representation while
changing the value of each of these finitely many polynomials by less
than $1/(2m)$. Include all norm constraints from the presentation,
with strict error tolerances. For a tuple $\bar y$ from $D^A$, require
these approximate relations and
\[
 \max_{t\leqslant r}\Phi(K,x_t-p_t(\bar y))<1/(2m).
\]
These conditions define an open set. Their union over $j$, $\bar p$,
and $\bar y$ is open, and every member has a genuine representation
whose range approximates $\bar x$ within $1/m$. Conversely, a genuine
representation approximating $\bar x$ sufficiently closely gives such
a witness by first choosing the polynomials and then approximating its
generators in the named dense subalgebra. Intersect over $\bar x,m$.
\end{proof}

\begin{remark}
The proposition is deliberately formulated for homomorphic images.  Stable finite
presentations let one perturb approximate relations to genuine representations, but
they do not by themselves force the representation to be injective.  Therefore the
argument above does not prove that the usual AI, $A\mathbb T$, AH, or ASH classes are
$\Pi^0_2$ or $\Pi^0_3$ in this coding.  To recover those embedding versions one would
need additional Borel conditions controlling the spectrum or injectivity of the
represented block.
\end{remark}

A separable $C^*$-algebra $B$ is \emph{quasidiagonal} (QD) if it admits a faithful
representation on a separable Hilbert space with an increasing approximate
unit of finite-rank projections that asymptotically commute with every operator in
the image, or equivalently if $B$ admits Voiculescu-type approximate embeddings into
matrix algebras (see \cite[V.4.2.1]{BlackadarOA}).
It is \emph{MF} if it embeds into
$\prod_n M_{k_n}/(\bigoplus_n M_{k_n})_{c_0}$
for some sequence $(k_n)$
(see \cite[V.4.3.1]{BlackadarOA}).

\begin{proposition}\label{prop:QD-MF}
The class of quasidiagonal quotients is $\Pi^0_3$ in the Wijsman topology, and the
class of MF quotients is closed.
\end{proposition}

\begin{proof}
Write $B=A/K$.

\smallskip
\noindent\textbf{(1) Quasidiagonality.}
We use the standard finite-dimensional approximation formulation of
quasidiagonality for separable unital $C^*$-algebras:
$B$ is quasidiagonal iff for every finite set $F\subseteq B$ and every
$\varepsilon>0$ there exist $m\in\N$ and a ucp map $\varphi:B\to M_m$ such that
\begin{equation}\label{eq:QD-approx}
\|\varphi(xy)-\varphi(x)\varphi(y)\|<\varepsilon\quad (x,y\in F),
\qquad
\bigl|\|\varphi(x)\|-\|x\|\bigr|<\varepsilon\quad (x\in F).
\end{equation}
(See e.g.\ \cite[V.4.2.1]{BlackadarOA} for such Voiculescu-type formulations.)

As usual, it suffices to test finite sets drawn from $D^A+K$ and
$\varepsilon=1/n$.
Fix a finite tuple $\bar x=(x_1,\dots,x_r)\subseteq D^A$ and $n\in\N$.
Let $F(\bar x)$ be the finite $*$-closed set consisting of
\[
1,\ x_i,\ x_i^*,\ x_ix_j,\ x_i^*x_j,\ x_ix_j^*,\ x_i^*x_j^*
\qquad (1\le i,j\le r).
\]

For $m\in\N$, let $\mathrm{UCP}(A,M_m)$ be the compact metrisable space of ucp
maps $A\to M_m$.
Consider the set of pairs $(\varphi,K)\in \mathrm{UCP}(A,M_m)\times\Ideal(A)$ such that

\begin{enumerate}[label=\textup{(\alph*)}]
\item $\varphi$ descends to the quotient $A/K$, \emph{i.e.}
$\|\varphi(a)\|\le \Phi(K,a)$ for all $a\in D^A$;
\item the induced map on $B=A/K$ is $1/n$-approximately multiplicative on
$F(\bar x)$:
$\|\varphi(st)-\varphi(s)\varphi(t)\|<1/n$ for all $s,t,st\in F(\bar x)$;
\item the induced map is $1/n$-approximately isometric on $F(\bar x)$:
$\bigl|\|\varphi(s)\|-\Phi(K,s)\bigr|<1/n$ for all $s\in F(\bar x)$.
\end{enumerate}

Condition~(a) is closed in $(\varphi,K)$, by the same compactness argument as in
Lemma~\ref{lem:ucp-fd}.
Conditions~(b) and~(c) are open, since they are strict inequalities involving
continuous functions of $(\varphi,K)$.
Hence the total relation is locally closed in the compact-by-Polish space
$\mathrm{UCP}(A,M_m)\times\Ideal(A)$.
By Lemma~\ref{lem:proj-locally-closed}, its projection onto $\Ideal(A)$ is
$F_\sigma$.

Taking the union over $m\in\N$, we see that for fixed $(\bar x,n)$ the set of
$K$ for which there exists a quasidiagonal witness is $\Sigma^0_2$.
Intersecting over all finite tuples $\bar x$ and all $n\in\N$ yields that the
quasidiagonal class is $\Pi^0_3$.

\smallskip\noindent\textbf{(2) MF.}
In fact, the MF class is closed. More precisely, within the clopen set
of proper ideals it is the Wijsman closure of the set of ideals with
finite-dimensional quotients.

Let $(g_i)$ be the free unitary generators of $A$, and enumerate their
rational $*$-polynomials as $(p_j)$. Suppose $B=A/K$ is MF.
For every finite set of these polynomials and every $\varepsilon>0$,
the defining matrix-product embedding supplies a matrix
size $d$ and unitaries $U_1,\ldots,U_l\in M_d$ for which
\[
 \bigl|\|p_j(U_1,\ldots,U_l)\|-\|p_j(g_1,\ldots,g_l)+K\|\bigr|
 <\varepsilon
\]
for the chosen polynomials. To see this directly, take lifts of the images
of the generating unitaries in a matrix-product embedding, and correct them
to unitaries by polar decomposition on sufficiently late coordinates.
For a finite list of polynomials, choose finitely many late coordinates,
one nearly attaining each of the relevant norm limsups, and take their
direct sum. All chosen coordinates may be taken sufficiently late to
satisfy the upper norm estimates for every polynomial in the list.
If the original embedding is non-unital, first lift the image of the unit
to projections and work in the associated matrix corners.
Extend the assignment $g_i\mapsto U_i$
by sending all remaining generators to the identity. The universal
property gives a unital $*$-homomorphism $\rho:A\to M_d$.
Its kernel has finite-dimensional quotient, and
$\Phi(\ker\rho,p_j)=\|\rho(p_j)\|$. Successively approximating more
polynomials gives $\ker\rho_n\to K$ in the Wijsman topology.

Conversely, suppose $K_n\to K$ and each $A/K_n$ is finite-dimensional.
Choose faithful representations $\rho_n:A/K_n\to M_{d_n}$.
The map
\[
 a\longmapsto[(\rho_n(a+K_n))_n]
 \quad\text{from }A\text{ to }
 \prod_nM_{d_n}/\bigoplus_nM_{d_n}
\]
has norm $\limsup_n\|a+K_n\|=\|a+K\|$ on each $a\in A$.
Its kernel is therefore exactly $K$, and it induces an isometric
$*$-embedding of $B$. This proves the claimed closure characterisation.

\end{proof}

\begin{proposition}\label{prop:approx-div}
The class of approximately divisible unital quotients is $G_\delta$.
\end{proposition}
\begin{proof}
We use the original formulation: for every finite $E\subseteq B$ and
$\varepsilon>0$ there is a unital finite-dimensional subalgebra
$F\subseteq B$, with no one-dimensional summands, such that
$\|[x,z]\|<\varepsilon$ for $x\in E$ and $z$ in the unit ball of $F$;
see \cite{BlackadarKumjianRordam}.
Equivalently, one may use a unital $*$-homomorphism from
$F=\bigoplus_{l=1}^sM_{d_l}$ with every $d_l\geqslant2$:
its image again has no one-dimensional summands.

Fix $\bar x=(x_1,\ldots,x_r)\in(D^A)^{<\N}$ and $m\geqslant1$, and set
$R=\max\{1,\|x_1\|,\ldots,\|x_r\|\}$ and $\mu=1/(8mR)$.
For each finite-dimensional type $F$ as above, fix a finite $\mu$-net
$T_F$ in its unit ball, consisting of elements with rational coefficients
in the canonical matrix units. Choose a matrix-unit correction tolerance
so small that correction changes the evaluated value of every member of
$T_F$ by less than $1/(8mR)$, and take the corresponding stability
tolerance. Search over tuples of approximate matrix units $\bar y$ from
$D^A$ satisfying the unital relations and norm bounds within this
tolerance, and
\[
 \Phi(K,[x_t,z(\bar y)])<1/(2m)
       \qquad(t\leqslant r,\ z\in T_F).
\]
For fixed $F$ and $\bar y$ these conditions are open. Correction gives
a unital homomorphism $\pi:F\to B$ with
$\|[x_t+K,\pi(z)]\|<3/(4m)$ for $z\in T_F$.
Approximating any member of the unit ball of $F$ by this net adds at
most $2R\mu=1/(4m)$, so the commutator bound on the whole unit ball
is strictly less than $1/m$.

Conversely, approximate divisibility applied to $\bar x+K$ with
tolerance $1/(4m)$ gives a suitable $F$. Its entire unit ball, hence
also the fixed net $T_F$, satisfies this centrality bound. Approximating
its matrix units sufficiently closely by elements of $D^A$ produces
the required open witness. The countable union over $F,\bar y$, followed
by the intersection over $\bar x,m$, gives the asserted $G_\delta$ set.
\end{proof}

\subsection{Stable finiteness}

A unital $C^*$-algebra $B$ is \emph{finite} if every isometry in $B$ is unitary, and
\emph{stably finite} if $M_n(B)$ is finite for every $n\in\N$
(see \cite[V.2.1.1]{BlackadarOA}).

\begin{proposition}\label{prop:stably-finite}
The class of stably finite quotients is closed (and in particular $\Pi^0_2$)
in the Wijsman topology.
\end{proposition}

\begin{proof}
Write $B=A/K$.
By definition, $B$ is stably finite iff $M_n(B)$ is finite for every $n\in\N$, \emph{i.e.}\
iff there is no $n$ for which $M_n(B)$ contains a non-unitary isometry.

We use the matrix-level quotient norms $\Phi^{(n)}$ of
Definition~\ref{def:Phi-matrix} and the countable dense set $M_n(D^A)$
(Remark~\ref{rem:Phi-n-dense}).
Fix the concrete constants
\[
\alpha:=\tfrac14,\qquad \beta:=\tfrac12.
\]

\smallskip\noindent
\emph{Claim.} $B$ is \emph{not} stably finite iff there exist $n\in\N$ and
$v\in M_n(D^A)$ such that
\begin{equation}\label{eq:sf-witness}
\Phi^{(n)}(K,v^*v-1)<\alpha
\qquad\text{and}\qquad
\Phi^{(n)}(K,vv^*-1)>\beta.
\end{equation}

\smallskip\noindent
($\Rightarrow$) If $B$ is not stably finite, then for some $n$ there exists an
isometry $u\in M_n(B)$ which is not unitary. Then $p:=uu^*$ is a proper projection,
so $\|p-1\|=1$, \emph{i.e.}\ $\|uu^*-1\|=1$.
Choose $v\in M_n(D^A)$ with $\|v+M_n(K)-u\|<\varepsilon$, where $\varepsilon>0$ is so
small that
\begin{align*}
\|(v+M_n(K))^*(v+M_n(K)) - u^*u\|&<\alpha,\\
\|(v+M_n(K))(v+M_n(K))^* - uu^*\|&<1-\beta.
\end{align*}
Since $u^*u=1$ and $\|uu^*-1\|=1$, this yields \eqref{eq:sf-witness}.

\smallskip\noindent
($\Leftarrow$) Suppose \eqref{eq:sf-witness} holds for some $n$ and $v\in M_n(D^A)$.
Let $x:=v+M_n(K)\in M_n(B)$ and set $s:=x^*x$.
Then $\|s-1\|<\alpha<1$, so $s$ is invertible with
$\|s^{-1}-1\|\le \alpha/(1-\alpha)$.
Define the polar correction $w:=x\,s^{-1/2}\in M_n(B)$.
Then $w^*w=s^{-1/2}s\,s^{-1/2}=1$, so $w$ is an isometry.

We show $w$ is not unitary. Since $ww^*=xs^{-1}x^*$, we have
$ww^*-xx^*=x(s^{-1}-1)x^*$, so
\[
\|ww^*-xx^*\|\le \|x\|^2\,\|s^{-1}-1\|.
\]
Moreover $\|x\|^2=\|s\|\le 1+\alpha$, so with $\alpha=\tfrac14$,
\[
\|ww^*-xx^*\|\le (1+\tfrac14)\cdot\frac{1/4}{1-1/4}
= \tfrac54\cdot\tfrac13=\tfrac{5}{12}<\tfrac12.
\]
Since $\|xx^*-1\|=\Phi^{(n)}(K,vv^*-1)>\beta=\tfrac12$, we conclude
\[
\|ww^*-1\|
\ge \|xx^*-1\|-\|ww^*-xx^*\|
> \tfrac12-\tfrac{5}{12}=\tfrac{1}{12}>0.
\]
Since $ww^*$ is a projection (as $w$ is an isometry), $\|ww^*-1\|>0$ forces
$ww^*\neq 1$, so $w$ is a non-unitary isometry and $B$ is not stably finite.

\smallskip\noindent
\emph{Borel complexity.}
For fixed $n$ and $v$, each inequality in \eqref{eq:sf-witness} is a strict
inequality in a Wijsman-continuous function of $K$ (Lemma~\ref{lem:Phi-n-continuous}),
hence defines an open set in $\Ideal(A)$. Taking the union over the countable set of
pairs $(n,v)$ shows the non-stably-finite codes form an open set. Therefore the class
of stably finite quotients is closed (and in particular $\Pi^0_2$).
\end{proof}

\begin{remark}\label{rem:stably-finite-sharp}
Proposition~\ref{prop:stably-finite} is stronger than the generic rank bound supplied by
Theorem~\ref{thm:definability-rank}: although stable finiteness is \emph{a priori} a universal
property, the complement is witnessed by the \emph{strict} inequalities
\eqref{eq:sf-witness}, so the non-stably-finite class is open and the stably finite class is
actually closed.  In particular, this is a genuine $\Pi^0_1$ phenomenon, not merely a
$\Pi^0_2$ one.
\end{remark}

\subsection{Real rank}

The \emph{real rank} of a unital $C^*$-algebra $B$, denoted $\mathrm{RR}(B)$, is the
smallest non-negative integer $n$ such that for every $(n\!+\!1)$-tuple of self-adjoint
elements $(x_0,\dots,x_n)$ in $B$ and every $\varepsilon>0$, there exists a
self-adjoint tuple $(y_0,\dots,y_n)$ with $\max_j\|x_j-y_j\|<\varepsilon$ such that
$\sum_{j=0}^n y_j^2$ is invertible (see \cite[V.3.2.3]{BlackadarOA}).

\begin{proposition}\label{prop:real-rank}
For fixed $n\in\N$, the class
$\{K\in\Ideal(A):\mathrm{RR}(A/K)\le n\}$ is $G_\delta$ in the Wijsman topology.
\end{proposition}

\begin{proof}
The self-adjoint part of $D^A$ is countable and dense in the self-adjoint part of each
quotient.  For a fixed self-adjoint tuple
$\bar x=(x_0,\ldots,x_n)\in(D^A\cap A_{\mathrm{sa}})^{n+1}$ and $m\in\N$, let
$U_{\bar x,m}$ be the set of $K$ for which there are self-adjoint
$\bar y=(y_0,\ldots,y_n)\in(D^A\cap A_{\mathrm{sa}})^{n+1}$ and $z\in D^A$ such that
\[
\max_j\Phi(K,x_j-y_j)<\frac1m,
\qquad
\Phi\Bigl(K,z\sum_{j=0}^n y_j^2-1\Bigr)<1,
\qquad
\Phi\Bigl(K,\Bigl(\sum_{j=0}^n y_j^2\Bigr)z-1\Bigr)<1 .
\]
For fixed witnesses this is open, and hence $U_{\bar x,m}$ is open.  The two inverse
inequalities imply that $\sum_j(y_j+K)^2$ is invertible after a Neumann-series
correction; conversely, an invertible positive element has an inverse which may be
approximated by an element of $D^A$.  Therefore
\[
        \{K:\mathrm{RR}(A/K)\le n\}
        =\bigcap_{\bar x}\bigcap_m U_{\bar x,m},
\]
which is $G_\delta$.
\end{proof}

\subsection{Simplicity}

\begin{proposition}\label{prop:simple}
The class of non-zero simple quotients is $G_\delta$ in $\Ideal(A)$.
\end{proposition}

\begin{proof}
For $c\in A_+$, let $V_c$ consist of the ideals $K$ for which $c+K$ is full
in $B=A/K$. This is an open set: fullness is equivalent to the existence of
$s\in\N$ and $x_1,\ldots,x_s,y_1,\ldots,y_s\in D^A$ such that
\[
 \Phi\left(K,1-\sum_{j=1}^s x_jcy_j\right)<1.
\]
Indeed, the closed ideal generated by $c+K$ contains the displayed sum;
if it lies within distance $1$ of the unit it is invertible, so that ideal
is all of $B$. The converse follows by density.

For $b\in D^A$ and $r\in\Q_{>0}$, put $c_{b,r}=(b^*b-r)_+$ and
\[
 U_{b,r}=\{K:\Phi(K,b^*b)<2r\}\ \cup\ V_{c_{b,r}}.
\]
Both terms are open. We claim that the required class is
\[
 \{K:\Phi(K,1)>1/2\}\cap
 \bigcap_{b\in D^A}\bigcap_{r\in\Q_{>0}}U_{b,r}.
\]
If $B$ is simple and $\|b^*b+K\|\geqslant2r$, its cutdown
$(b^*b+K-r)_+$ is non-zero and hence full. Conversely, suppose $B$ has
a proper non-zero closed ideal $J$. Choose $a\in J_+$ with $\|a\|=1$,
a rational $0<r<1/4$, and $b\in D^A$ such that
$\|b^*b+K-a\|<r$. Then $\|b^*b+K\|>1-r>2r$.
In $B/J$ the positive image of $b^*b+K$ has norm less than $r$;
functional calculus therefore gives $c_{b,r}+K\in J$. It cannot be full,
so $K\notin U_{b,r}$. This proves the claim.
\end{proof}

\subsection{Nuclear dimension and decomposition rank}

The nuclear dimension of $B$ is at most $n$ if finite subsets of $B$
admit arbitrarily accurate completely positive approximations through
$F=F^{(0)}\oplus\cdots\oplus F^{(n)}$, with a c.p.c. incoming map
$\psi:B\to F$ and outgoing maps $\varphi_i:F^{(i)}\to B$ that are
c.p.c. and preserve orthogonality of positive elements. The latter maps
are called \emph{order zero}; the approximation is
$\sum_i\varphi_i\psi(\cdot)^{(i)}\to\mathrm{id}_B$.
Requiring in addition that the total outgoing map
$\varphi=\sum_i\varphi_i:F\to B$ be contractive gives
\emph{decomposition rank} at most $n$; see \cite{WinterZacharias}.

\begin{proposition}\label{prop:nucdim}
For each $n\geqslant0$, the classes of quotients with nuclear dimension
at most $n$ and decomposition rank at most $n$ are $G_\delta$.
\end{proposition}
\begin{proof}
The two assertions are, respectively, items~(4) and~(5) of
\cite[Theorem~5.7.3]{FarahHartLupiniRobertTikuisisVignatiWinter}:
each fixed nuclear-dimension bound and each fixed decomposition-rank
bound is described by uniform families of formulas. The final assertion
of that theorem makes these families existential. Their uniformity has
the common-modulus meaning specified in Definition~5.7.1, so
Lemma~\ref{lem:uniform-existential} gives both $G_\delta$ bounds.
For $n=0$, both properties coincide with AF-ness, consistently with
Proposition~\ref{prop:AF}.
\end{proof}

\subsection{Property (SP)}

A $C^*$-algebra has \emph{property (SP)} if every non-zero hereditary
$C^*$-subalgebra contains a non-zero projection. We use the notation
$\operatorname{her}(a)=\overline{aBa}$ for $a\in B_+$.

\begin{proposition}\label{prop:SP}
The class of $K\in\Ideal(A)$ such that $A/K$ has property~\textup{(SP)}
is $G_\delta$ in the Wijsman topology.
\end{proposition}

\begin{proof}
For $c\in A_+$ write
\[
 V_c=\{K:\overline{(c+K)(A/K)(c+K)}
                \text{ contains a non-zero projection}\}.
\]
We first observe that $V_c$ is open. Indeed, $K\in V_c$ if and only if
there exists $x\in D^A$ such that, with $h=cx^*xc$,
\[
 \Phi(K,h) >\tfrac12,
 \qquad \Phi(K,h^2-h)<\tfrac18.
\]
For the forward implication, use that the positive cone of
$\overline{(c+K)(A/K)(c+K)}$ is the closure of
$\{(c+K)y^*y(c+K):y\in A/K\}$, and approximate a non-zero projection.
Conversely, $h+K$ is positive, and the displayed inequalities separate its
spectrum into a neighbourhood of zero and a non-empty neighbourhood of one.
Continuous functional calculus, with the cut-off function vanishing at zero,
therefore produces a non-zero projection in the hereditary subalgebra.
The witness conditions are open, so $V_c$ is open.

For $b\in D^A$ and rational $\varepsilon>0$, put
\[
 c_{b,\varepsilon}=(b^*b-\varepsilon)_+,
 \qquad
 U_{b,\varepsilon}
 =\{K:\Phi(K,b^*b)<2\varepsilon\}\cup V_{c_{b,\varepsilon}}.
\]
These sets are open. We claim that the class in question equals
$\bigcap_{b,\varepsilon}U_{b,\varepsilon}$.
If $A/K$ has property~(SP), then whenever
$\|b^*b+K\|\geqslant2\varepsilon$, the positive element
$(b^*b+K-\varepsilon)_+$ is non-zero and its hereditary subalgebra
contains a non-zero projection. Thus every $U_{b,\varepsilon}$ contains~$K$.

For the converse, put $B=A/K$ and let $a\in B_+\setminus\{0\}$.
Choose a positive rational $\varepsilon<\|a\|/4$ and $b\in D^A$ such that
$\|a-(b^*b+K)\|<\varepsilon/2$.
Then $\|b^*b+K\|>2\varepsilon$, so the assumed membership yields a non-zero
projection $p\in\operatorname{her}((b^*b+K-\varepsilon)_+)$.
Writing $d=b^*b+K$, its support projection in $B^{**}$ shows that
$pdp\geqslant\varepsilon p$. Hence
\[
 pap\geqslant\tfrac\varepsilon2p.
\]
The element $pap$ is invertible in $pBp$. Consequently
\[
 v=a^{1/2}p(pap)^{-1/2},\qquad q=vv^*
\]
satisfies $v^*v=p$, and $q$ is a non-zero projection in
$\overline{aBa}$. This proves property~(SP).
\end{proof}

\subsection{Tracial states}

\begin{proposition}\label{prop:trace}
The class of $K\in\Ideal(A)$ such that $A/K$ admits a tracial state is closed
(in particular $\Pi^0_1$) in the Wijsman topology.
\end{proposition}

\begin{proof}
Let $T(A)$ be the compact (metrisable) simplex of tracial states on $A$
with its weak$^*$ topology.
A tracial state on $B=A/K$ is the same as a tracial state $\tau\in T(A)$
vanishing on $K$ (compose with the quotient map $A\to A/K$, and conversely factor
through $A/K$).

Fix the dense $*$-subalgebra $D^A\subseteq A$.
For $\tau\in T(A)$ and $K\in\Ideal(A)$, the condition $K\subseteq\ker(\tau)$ is
equivalent to
\[
\forall a\in D^A:\ |\tau(a)|\le \|a+K\|_{A/K}=\Phi(K,a),
\]
since $\tau$ is contractive and $\Phi(K,k)=0$ for $k\in K$.

For each fixed $a\in D^A$, the map $(\tau,K)\mapsto |\tau(a)|-\Phi(K,a)$ is continuous
on $T(A)\times\Ideal(A)$ (weak$^*$ continuity in $\tau$ and Wijsman continuity in $K$),
hence the inequality defines a closed subset.
Therefore
\[
C:=\{(\tau,K)\in T(A)\times\Ideal(A):\ K\subseteq\ker(\tau)\}
\]
is closed.
Since $T(A)$ is compact, the projection of $C$ onto $\Ideal(A)$ is closed, and this
projection is exactly the set of ideals $K$ for which $A/K$ carries a tracial state.
\end{proof}

\begin{remark}
For a fixed separable $C^*$-algebra $B$, approximately inner
 automorphisms form the point-norm closure of $\Inn(B)$, hence a closed
normal subgroup of $\Aut(B)$.  A uniform treatment of varying actions
requires an additional coding space, which we do not introduce here.
\end{remark}

\begin{remark}\label{rem:summary}
Commutativity, MF-ness, stable finiteness and existence of a tracial state
are closed. AF-ness, simplicity, property~(SP), approximate divisibility,
real-rank bounds and topological stable-rank bounds are $G_\delta$.
Nuclearity and fixed nuclear-dimension and decomposition-rank bounds
are also $G_\delta$. Quasidiagonality is $\Pi^0_3$.
The following reductions establish sharpness for the closed and
$G_\delta$ classes; approximate divisibility is treated together with
CAR absorption in Theorem~\ref{thm:CAR-stability-complete}.
\end{remark}

\subsection{Sharp closed classes}

\begin{lemma}\label{lem:scalar-degeneration}
Let $B$ be a non-zero separable unital $C^*$-algebra. There is a
Wijsman-continuous map $[0,1]\ni t\mapsto K_t\in\Ideal(A)$ such that
$A/K_0\cong\mathbb C$ and $A/K_t\cong B$ for every $t>0$.
\end{lemma}

\begin{proof}
Choose a sequence $(h_n)$ of self-adjoint contractions generating~$B$.
For the canonical unitary generators $(u_n)$ of $A=C^*_{\max}(F_\infty)$,
the universal property gives a unital $*$-homomorphism
\[
 q_t:A\longrightarrow B,\qquad q_t(u_n)=\exp(\mathrm i t h_n),
 \qquad K_t=\ker q_t.
\]
For $0<t\leqslant1<\pi$, the exponential map is injective on
$[-t,t]$. Its continuous inverse on the arc
$\{\exp(\mathrm i s):|s|\leqslant t\}$, followed by division by $t$,
recovers $h_n$ from $\exp(\mathrm i t h_n)$ by functional calculus.
Hence $q_t$ is surjective.
For $t=0$, the range of $q_0$ is $\mathbb C1_B$.
For every $a\in A$, the map $t\mapsto q_t(a)$ is norm-continuous:
this is immediate for non-commutative $*$-polynomials in finitely many $u_n$,
and follows for general $a$ by uniform approximation and contractivity of~$q_t$.
The canonical identification $A/K_t\cong q_t(A)$ is isometric, so
\[
 \operatorname{dist}(a,K_t)=\|q_t(a)\|.
\]
These functions are continuous for every $a$, which proves Wijsman continuity.
\end{proof}

\begin{corollary}\label{cor:closed-sharpness}
Let $\mathcal P\subseteq\Ideal(A)$ be a closed isomorphism-invariant class
containing the scalar algebra and omitting some non-zero separable unital
$C^*$-algebra. Then $\mathcal P$ is $\Pi^0_1$-complete.
In particular, the classes of commutative, MF and stably finite quotients, and
quotients admitting a tracial state, are $\Pi^0_1$-complete.
\end{corollary}

\begin{proof}
Choose an omitted algebra $B$ and the path from
Lemma~\ref{lem:scalar-degeneration}. Its inverse image of $\mathcal P$ is
$\{0\}$. If $F$ is any non-empty closed subset of the Cantor space, equip
that space with a compatible metric bounded by one and put
$t(x)=\operatorname{dist}(x,F)$. Then $x\mapsto K_{t(x)}$ continuously
reduces $F$ to~$\mathcal P$. The empty set is reduced by the constant
map with value $K_1$.
For commutativity take $B=M_2(\C)$. For the remaining applications
take $B=\mathcal O_2$: its two canonical
isometries witness failure of stable finiteness, while a tracial state would
give $1=\tau(s_1s_1^*)+\tau(s_2s_2^*)=2$. The MF case uses closedness
from Proposition~\ref{prop:QD-MF} and the fact that MF algebras are
stably finite. Indeed, a non-unitary isometry cannot occur in a norm
quotient of products of matrix algebras: coordinatewise polar correction
turns an approximate isometry into a unitary, forcing the opposite defect
to tend to zero as well.
\end{proof}

\subsection{Dimension and finite-dimensional representations}

\begin{lemma}\label{lem:finite-dimension-closed}
Let $P$ be a separable Banach space over $\mathbb F\in\{\R,\C\}$ and equip its closed-subspace space
with the Wijsman topology. For every integer $d\geqslant0$, the set
\[
 \{K:\dim(P/K)\leqslant d\}
\]
is closed. In particular this holds on the ideal space of a separable
Banach algebra or $C^*$-algebra.
\end{lemma}

\begin{proof}
Fix a countable dense set $E\subseteq P$. For
$\bar a=(a_0,\ldots,a_d)\in E^{d+1}$ put
\[
 \Delta_{\bar a}(K)
 =\min\left\{
   \operatorname{dist}\left(\sum_{j=0}^d\lambda_ja_j,K\right):
   \lambda\in\mathbb F^{d+1},\ \sum_{j=0}^d|\lambda_j|^2=1
            \right\}.
\]
Joint continuity of the distance functional and compactness of the
coefficient sphere show that $\Delta_{\bar a}$ is continuous.
The images of the $a_j$ are linearly independent precisely when
$\Delta_{\bar a}(K)>0$. If the quotient has dimension greater than $d$,
its dense subset $E+K$ contains a linearly independent $(d+1)$-tuple.
Consequently the complement of the displayed class is the open set
\[
 \bigcup_{\bar a\in E^{d+1}}\{K:\Delta_{\bar a}(K)>0\}.
\]
\end{proof}

\begin{lemma}\label{lem:matrix-CAR-family}
Let $D=\bigotimes_{n\in\N}M_2$ be the CAR algebra. There is a
point-norm-continuous family of unital $*$-homomorphisms
$\pi_x:A\to D$, indexed by $x\in2^\N$, whose images are full matrix
algebras when $x$ has finite support and are isomorphic to $D$ otherwise.
The kernel map $x\mapsto\ker\pi_x$ is Wijsman-continuous.
\end{lemma}
\begin{proof}
In the $n$th tensor factor of $D$, let $u_n,v_n$ be the copies of
\[
u=\begin{pmatrix}1&0\\0&-1\end{pmatrix},\qquad
v=\begin{pmatrix}0&1\\1&0\end{pmatrix}.
\]
These unitaries generate that copy of $M_2$.
Enumerate the free unitary generators of $A$ as
$(s_n,t_n)_{n\in\N}$, and let $\pi_x:A\to D$ be the unital
$*$-homomorphism defined by
\[
\pi_x(s_n)=\begin{cases}u_n,&x(n)=1,\\1,&x(n)=0,\end{cases}
\qquad
\pi_x(t_n)=\begin{cases}v_n,&x(n)=1,\\1,&x(n)=0.\end{cases}
\]
Set $K_x=\ker\pi_x$. Then
\[
A/K_x\cong
\begin{cases}
M_{2^r},&|\{n:x(n)=1\}|=r<\infty,\\
D,&x\text{ has infinite support}.
\end{cases}
\]
For every $*$-polynomial $p$ in finitely many free generators,
the element $\pi_x(p)\in D$ depends on finitely many coordinates of $x$.
Such polynomials are dense in $A$, and the maps $\pi_x$ are
contractive. Uniform approximation therefore shows that
$x\mapsto\pi_x(a)$ is norm-continuous for every $a\in A$. In particular,
$x\mapsto\|\pi_x(a)\|=\dist(a,K_x)$ is continuous.
Thus $x\mapsto K_x$ is Wijsman-continuous.
\end{proof}

\begin{theorem}\label{thm:finite-dimensional-Cstar}
In the Wijsman quotient coding by $A=C^*_{\max}(F_\infty)$:
\begin{enumerate}
\item for each $d\geqslant1$, dimension at most $d$ is
      $\Pi^0_1$-complete;
\item finite dimensionality is $\Sigma^0_2$-complete;
\item infinite dimensionality is $\Pi^0_2$-complete.
\end{enumerate}
The last two lower bounds already hold among simple unital AF quotients.
For each fixed $d$, the locus of dimension exactly $d$ is locally closed.
\end{theorem}

\begin{proof}
Lemma~\ref{lem:finite-dimension-closed} gives the closed upper bounds.
For $d\geqslant1$,
the corresponding class contains $\mathbb C$ and omits
$\mathbb C^{d+1}$, so Corollary~\ref{cor:closed-sharpness} proves
$\Pi^0_1$-completeness. The countable union over $d$ is $F_\sigma$.
Use the continuous family $x\mapsto K_x$ from
Lemma~\ref{lem:matrix-CAR-family}: its quotient is a full matrix
algebra if $x$ has finite support, and is the CAR algebra if $x$ has
infinite support. This reduces the standard $\Sigma^0_2$-complete
finite-support set to finite dimensionality, and its complement to
infinite dimensionality. Every quotient in this family is simple,
unital and AF.
Finally, for $d\geqslant1$ the exact-dimension locus is the intersection
of the closed condition $\dim(A/K)\leqslant d$ with the open condition
$\dim(A/K)>d-1$. The zero-dimensional locus is closed.
\end{proof}

Representations are $*$-homomorphisms and need not be irreducible.
A non-zero finite-dimensional representation becomes unital after
compression to its support projection; the dimension may decrease.

\begin{theorem}\label{thm:finite-dimensional-representations}
In the same coding, for each $n\geqslant1$ both of the following
conditions are $\Pi^0_1$-complete:
\begin{enumerate}
\item the quotient admits a unital representation into $M_n$;
\item the quotient admits a non-zero representation of dimension at most $n$.
\end{enumerate}
In particular, existence of a character is $\Pi^0_1$-complete.
Existence of a non-zero finite-dimensional representation is
$\Sigma^0_2$-complete, and absence of such representations is
$\Pi^0_2$-complete. These last two lower bounds hold already among
simple unital AF quotients.
\end{theorem}

\begin{proof}
For fixed $n$, let $\mathcal R_n$ be the space of unital
$*$-homomorphisms $A\to M_n$, with the point-norm topology.
By the universal property of $A$, this is the compact space
$U(n)^{\mathbb N}$: each tuple of unitaries gives a representation,
and pointwise convergence on the generators gives point-norm convergence
on $A$ by polynomial approximation.
A representation $\pi\in\mathcal R_n$ factors through $A/K$ if and only if
\[
 \|\pi(a)\|\leqslant\Phi(K,a)\qquad(a\in D^A).
\]
This is a closed condition on $(K,\pi)$. Projection along the compact
factor $\mathcal R_n$ shows that the locus $R_n$ of quotients admitting
such a representation is closed. The locus of non-zero representations
of dimension at most $n$ is $R_1\cup\cdots\cup R_n$, also closed.
Both classes contain the scalar algebra and omit $\mathcal O_2$:
in a finite matrix algebra every isometry is unitary, whereas the two
canonical isometries of $\mathcal O_2$ have orthogonal ranges summing
to the identity. Corollary~\ref{cor:closed-sharpness} gives both
$\Pi^0_1$ lower bounds.

The union $\bigcup_nR_n$ is $F_\sigma$. In the continuous family from
Lemma~\ref{lem:matrix-CAR-family}, every finite-support input
codes a full matrix algebra and therefore belongs to this union.
Every infinite-support input codes the CAR algebra, which has no
non-zero finite-dimensional representation: simplicity would make
such a representation injective, contradicting infinite dimensionality.
This proves $\Sigma^0_2$-hardness, and taking complements proves the
last assertion.
\end{proof}

\subsection{Residual finite dimensionality}

\begin{theorem}\label{thm:RFD-Pi3-complete}
In the Wijsman quotient coding of separable unital $C^*$-algebras,
residual finite-dimensionality is $\Pi^0_3$-complete. The lower bound
already holds among AF quotients.
\end{theorem}
\begin{proof}
Write $A=C^*_{\max}(F_\infty)$ and work with non-zero quotients, as
in Convention~\ref{conv:proper-quotients}. The isolated zero quotient
may be added separately. A unital $C^*$-algebra is
\emph{residually finite-dimensional} if its finite-dimensional unital
representations separate points. For each $n\geqslant1$, the space
$H_n=\operatorname{Hom}_1(A,M_n)$ is compact metrisable in the
point-norm topology: choosing the images of the free unitary generators
identifies it with the compact product $U(n)^{\N}$. The relation
\[
 R_n=\{(K,\rho)\in\Ideal(A)\times H_n:K\subseteq\ker\rho\}
\]
is closed, since it is equivalent to
\[
 \|\rho(a)\|\leqslant\Phi(K,a)\qquad(a\in D^A).
\]
Indeed, these inequalities extend by density to all $a\in A$, and
then force $\rho$ to vanish on $K$.

Finite-dimensional representations separate the points of $A/K$ if
and only if, for every $a\in D^A$ and positive integer $m$, there are
$n$ and $\rho\in H_n$ with
\[
 (K,\rho)\in R_n,\qquad
 \|\rho(a)\|>\Phi(K,a)-1/m.
\]
For the forward implication, the product of all separating
representations gives an injective, hence isometric, $*$-homomorphism;
the converse follows by density. For fixed $a,m,n$ the displayed
condition projects a closed set intersected with an open set along the
compact factor $H_n$, and hence is $F_\sigma$ by
Lemma~\ref{lem:proj-locally-closed}. Taking the union over $n$ and then
the intersection over $a,m$ proves the $\Pi^0_3$ upper bound.

For hardness, use the $\Pi^0_3$-complete set
\[
 P_3=\{\alpha\in2^{\N\times\N}:
             \text{every row of }\alpha\text{ has finitely many ones}\}.
\]
Let $D=\bigotimes_{s\in\N}M_2$ be the CAR algebra and put
$E=(\bigoplus_{r\in\N}D)^{\sim}$, where the direct sum is the
$c_0$-sum. Denote the unit of its $r$th summand by $p_r$. In the
$s$th tensor factor of $D$, choose two self-adjoint unitaries
$v_{s,1},v_{s,2}$ generating that copy of $M_2$; for example, the
usual two Pauli matrices. Let $v^{(r)}_{s,j}$ denote the corresponding
unitary in the corner $p_rEp_r$.

For $\alpha\in2^{\N\times\N}$, let $B_\alpha\subseteq E$ be generated
unitally by the fixed unitaries $1-2p_r$ and the unitaries
\[
 w_{r,s,j}(\alpha)=
 \begin{cases}
  1-p_r+v^{(r)}_{s,j},&\alpha(r,s)=1,\\
  1,&\alpha(r,s)=0.
 \end{cases}
\]
Then
\[
 B_\alpha=
 \left(\bigoplus_{r\in\N}D_{\alpha,r}\right)^{\sim},
 \qquad
 D_{\alpha,r}\cong
 \begin{cases}
 M_{2^{k_r}},&k_r=\#\{s:\alpha(r,s)=1\}<\infty,\\
 D,&\#\{s:\alpha(r,s)=1\}=\infty.
 \end{cases}
\]
The fixed generators $1-2p_r$ ensure that the scalar corner $M_1$
is present even when the row has no ones. All these algebras are AF: finite sets can be approximated in
finite tensor stages of finitely many summands, together with the unit.

Enumerate the displayed unitaries and let $q_\alpha:A\to B_\alpha$
be the resulting surjective unital $*$-homomorphism. Set
$K_\alpha=\ker q_\alpha$. For every $*$-polynomial in finitely many
free generators, its image in the fixed algebra $E$ depends on only
finitely many bits of $\alpha$. Contractivity and approximation by
such polynomials therefore make $\alpha\mapsto q_\alpha(a)$
norm-continuous for every $a\in A$. Hence
\[
 \operatorname{dist}(a,K_\alpha)=\|q_\alpha(a)\|
\]
proves Wijsman continuity of $\alpha\mapsto K_\alpha$.

If $\alpha\in P_3$, all summands $D_{\alpha,r}$ are matrix algebras;
the coordinate representations, together with the scalar quotient,
separate points of $B_\alpha$. If a row is infinite, $B_\alpha$ has
an ideal isomorphic to $D$. Every finite-dimensional representation
vanishes on that ideal, since simplicity of the infinite-dimensional
CAR algebra precludes a non-zero homomorphism into a matrix algebra.
Thus
\[
 B_\alpha\text{ is residually finite-dimensional}
 \quad\Longleftrightarrow\quad\alpha\in P_3,
\]
which proves the asserted completeness, including the restriction to
AF quotients.
\end{proof}

\subsection{Sharpness of some \texorpdfstring{$G_\delta$}{G-delta} bounds}

We begin with commutative reductions for AF-ness, property~(SP),
and the rank bounds. Subsequent reductions treat simplicity and nuclearity.
The commutative reductions live inside a small part of the quotient space.  This still proves hardness in the full quotient
coding, because the reduction constructed below is a continuous map into the
ambient Wijsman ideal space.

Let
\[
        P_2=\{\alpha\in 2^{\N\times\N}:\forall r\;\exists s\; \alpha(r,s)=1\}.
\]
This is a standard $\Pi^0_2$-complete subset of the Cantor space $2^{\N\times\N}$.

\begin{lemma}\label{lem:compact-Pi2-test}
There is a continuous map
\[
        \alpha\longmapsto E_\alpha
\]
from $2^{\N\times\N}$ into the hyperspace of non-empty compact subsets of $[0,1]$,
with the Hausdorff topology, such that:
\begin{enumerate}
\item if $\alpha\in P_2$, then $E_\alpha$ is countable, hence zero-dimensional;
\item if $\alpha\notin P_2$, then $E_\alpha$ contains a non-degenerate interval.
\end{enumerate}
\end{lemma}

\begin{proof}
Choose pairwise disjoint closed intervals
\[
        I_r=[2^{-2r-2},2^{-2r-1}]\qquad(r\in\N),
\]
which accumulate only at~$0$.  For each pair $(r,s)$ choose a finite set
$F_{r,s}\subseteq I_r$, containing the endpoints of $I_r$, such that
\[
        d_H(F_{r,s},I_r)<2^{-s},
\]
where $d_H$ denotes Hausdorff distance.  For $\alpha\in2^{\N\times\N}$ put
\[
        m_r(\alpha)=\min\{s:\alpha(r,s)=1\},
\]
with $m_r(\alpha)=\infty$ if the set is empty, and define
\[
        E_\alpha=
        \{0\}
        \cup\bigcup_{m_r(\alpha)=\infty} I_r
        \cup\bigcup_{m_r(\alpha)<\infty} F_{r,m_r(\alpha)} .
\]
This is compact because the intervals $I_r$ accumulate only at~$0$.
If every row of $\alpha$ contains a~$1$, then $E_\alpha$ is a countable compact set.
If some row contains no~$1$, then $E_\alpha$ contains the corresponding interval~$I_r$.

It remains only to check continuity.  Let $\alpha_j\to\alpha$.  Given $\varepsilon>0$,
choose $R$ so large that
\[
        \{0\}\cup\bigcup_{r>R} I_r \subseteq [0,\varepsilon].
\]
For each $r\le R$ with $m_r(\alpha)<\infty$, the first $m_r(\alpha)+1$ entries in row
$r$ are eventually constant, so the corresponding piece of $E_{\alpha_j}$ is eventually
exactly the same as the piece of $E_\alpha$.  For each $r\le R$ with
$m_r(\alpha)=\infty$, choose $S$ so large that $2^{-S}<\varepsilon$; eventually the
first $S$ entries in row~$r$ of $\alpha_j$ are all zero, and therefore the corresponding
piece of $E_{\alpha_j}$ is either $I_r$ or some $F_{r,s}$ with $s\ge S$, in both cases
within Hausdorff distance $<\varepsilon$ from $I_r$.  Combining the finitely many
large intervals with the small tail gives $d_H(E_{\alpha_j},E_\alpha)\to0$.
\end{proof}

For $d\ge0$ set
\[
        Y_{\alpha,d}:=E_\alpha\times [0,1]^d\subseteq [0,1]^{d+1},
\]
where $[0,1]^0$ is a point.  Fix once and for all a surjection
\[
        \rho_d:C^*_{\max}(F_\infty)\longrightarrow C([0,1]^{d+1}).
\]
For a compact set $Y\subseteq [0,1]^{d+1}$ write
\[
        I_Y=\{f\in C([0,1]^{d+1}):f|_Y=0\},
        \qquad
        K_Y=\rho_d^{-1}(I_Y).
\]
The map $Y\mapsto K_Y$ is Wijsman-continuous: for $a\in C^*_{\max}(F_\infty)$,
\[
        \dist(a,K_Y)=\|\rho_d(a)+I_Y\|=\|\rho_d(a)|_Y\|_\infty
        =\max_{y\in Y}|\rho_d(a)(y)|,
\]
and the last expression is continuous in $Y$ for the Hausdorff topology.
Since $Y\mapsto I_Y$ is injective, compactness of the hyperspace and metrisability of the
Wijsman ideal space show that $Y\mapsto K_Y$ is a topological embedding onto its image.
Thus a continuous hardness reduction landing in this commutative subfamily is already a
hardness reduction for the full quotient coding.

\begin{theorem}\label{thm:Cstar-Pi2-complete}
In the Wijsman quotient coding of separable unital $C^*$-algebras by ideals of
$C^*_{\max}(F_\infty)$, the following classes are $\Pi^0_2$-complete:
\begin{enumerate}
\item AF algebras and algebras with property~\textup{(SP)};
\item algebras of real rank, nuclear dimension, or decomposition rank at most $d$, for each fixed $d\ge0$;
\item algebras of topological stable rank at most $n$, for each fixed $n\ge1$.
\end{enumerate}

\end{theorem}

\begin{proof}
Membership in $\Pi^0_2$ for AF, property~(SP), real-rank bounds and topological stable-rank bounds was
proved above.  We prove hardness by reducing $P_2$.

For AF-ness use $Y_{\alpha,0}=E_\alpha$.  If $\alpha\in P_2$, then $E_\alpha$ is
countable and compact, hence zero-dimensional, and $C(E_\alpha)$ is AF.  If
$\alpha\notin P_2$, then $E_\alpha$ contains a non-degenerate interval, so
$C(E_\alpha)$ is not AF.  Thus
\[
        \alpha\in P_2
        \quad\Longleftrightarrow\quad
        C^*_{\max}(F_\infty)/K_{E_\alpha}\cong C(E_\alpha)\text{ is AF}.
\]
Since $\alpha\mapsto K_{E_\alpha}$ is continuous, AF-ness is $\Pi^0_2$-hard.

For real rank, use $Y_{\alpha,d}=E_\alpha\times[0,1]^d$.  If $\alpha\in P_2$, then
$E_\alpha$ is zero-dimensional, so standard dimension theory gives
$\dim(Y_{\alpha,d})=d$.  If $\alpha\notin P_2$, then $Y_{\alpha,d}$ contains a copy of
$[0,1]^{d+1}$, so $\dim(Y_{\alpha,d})\ge d+1$.  Since real rank of a commutative
unital $C^*$-algebra is covering dimension~\cite{BrownPedersen1991}, $\mathrm{rr}(C(Y_{\alpha,d}))\le d$ holds
exactly for $\alpha\in P_2$.  This proves $\Pi^0_2$-hardness of the real-rank bound.

For topological stable rank~$\le n$, put $d=2n-1$.  The same dimension computation
gives dimension $2n-1$ when $\alpha\in P_2$ and dimension $2n$ otherwise; the
choice $2n-1$ is exactly the threshold between stable ranks $n$ and $n+1$.
Rieffel's formula
$\mathrm{tsr}(C(X))=\lfloor \dim(X)/2\rfloor+1$ for compact spaces of finite covering
dimension~\cite{Rieffel1983} therefore yields
\[
        \mathrm{tsr}(C(Y_{\alpha,2n-1}))\le n
        \quad\Longleftrightarrow\quad
        \alpha\in P_2.
\]
This proves $\Pi^0_2$-hardness of the stable-rank bound.

The same commutative family handles nuclear dimension and decomposition
rank, both of which equal $\dim(X)$ on $C(X)$ for compact metrisable
$X$~\cite[Proposition~2.4]{WinterZacharias}. Their $G_\delta$ upper
bounds were proved in Proposition~\ref{prop:nucdim}; hence they too
are $\Pi^0_2$-complete.
For property~(SP), $C(E_\alpha)$ has property~(SP) when $E_\alpha$ is zero-dimensional:
every non-empty open subset contains a non-empty clopen subset, whose characteristic
function is a non-zero projection.  If $E_\alpha$ contains an interval, a hereditary
subalgebra supported in a smaller open interval contains no non-zero projection.  Hence
property~(SP) also has $P_2$ as a continuous preimage.
\end{proof}

\begin{corollary}\label{cor:finite-rank-complete}
Finiteness of real rank, topological stable rank, nuclear dimension
and decomposition rank gives $\Sigma^0_3$-complete classes in the
Wijsman quotient coding of unital separable $C^*$-algebras.
\end{corollary}

\begin{proof}
The upper bounds follow by taking the countable unions of the fixed-rank
classes. For hardness, let $Q=[0,1]^{\mathbb N}$.
Gladdines and van Mill prove that
\[
 \{X\in\mathcal K(Q):\dim X=\infty\}
\]
is an $F_{\sigma\delta}$-absorber, and in particular a
$\Pi^0_3$-complete subset of the hyperspace of non-empty compact subsets
of~$Q$; see~\cite[Corollary~3.5]{GladdinesVanMill1993}.
Thus finite-dimensional compacta form a $\Sigma^0_3$-complete set.
Fix a surjective $*$-homomorphism $\rho:A\to C(Q)$ and map
$X\in\mathcal K(Q)$ to
\[
 K_X=\{a\in A:\rho(a)|_X=0\}.
\]
As before, $\operatorname{dist}(a,K_X)=\max_{x\in X}|\rho(a)(x)|$
proves continuity, and $A/K_X\cong C(X)$.
The commutative formulas for real rank, topological stable rank,
nuclear dimension and decomposition rank identify all four inverse images with the
finite-dimensional compacta. This proves all four lower bounds.
\end{proof}

\begin{corollary}\label{cor:Banach-sharp-ranks}
Let $P=\ell_1(S)$, where $S$ is the free monoid on countably many
symbols. In its Wijsman ideal coding, right topological stable rank at
most $n$ is $\Pi^0_2$-complete for every $n\geqslant1$, and finite right
topological stable rank is $\Sigma^0_3$-complete. Uniformly open
multiplication is $\Pi^0_2$-hard in this coding. On the closed
commutative locus of the unital $C^*$-coding, open multiplication and
uniformly open multiplication are both $\Pi^0_2$-complete.
\end{corollary}
\begin{proof}
Put $Q=[0,1]^{\N}$. Choose the images of the free generators dense in
the open unit ball of $C(Q)$. As in Section~\ref{sec:examples}, the
resulting unital homomorphism $q:P\to C(Q)$ is a metric quotient:
\[
 \inf\{\|a\|:q(a)=f\}=\|f\|\qquad(f\in C(Q)).
\]
For non-empty compact $X\subseteq Q$, put
$L_X=q^{-1}\{f:f|_X=0\}$. The metric quotient property gives
\[
 \operatorname{dist}(a,L_X)
 =\operatorname{dist}(q(a),\{f:f|_X=0\})
 =\max_{x\in X}|q(a)(x)|.
\]
Thus $X\mapsto L_X$ is Wijsman-continuous and
$P/L_X$ is isometrically isomorphic to $C(X)$. The reductions in
Theorem~\ref{thm:Cstar-Pi2-complete} and
Corollary~\ref{cor:finite-rank-complete} therefore give the stable-rank
lower bounds in the Banach-algebra coding. Their upper bounds follow
from Proposition~\ref{prop:tsr} and a countable union.

By the author and Ma\'slany~\cite[Theorem~1.5]{KaniaMaslany}, for a
compact space $X$ and the algebra of complex-valued continuous functions,
the following conditions are equivalent:
\begin{enumerate}
\item $C(X)$ has open multiplication;
\item $C(X)$ has uniformly open multiplication;
\item $\dim X\leqslant1$.
\end{enumerate}
For compact metrisable $X$, the last condition is also equivalent to
$\mathrm{tsr}(C(X))=1$. The stable-rank-one reduction above gives
$\Pi^0_2$-hardness of uniformly open multiplication for Banach
quotients. In the commutative $C^*$-locus the equivalence supplies the
$G_\delta$ upper bounds, completing the proof.
\end{proof}

\begin{proposition}\label{prop:simple-complete}
The class of simple quotients is $\Pi^0_2$-complete. Its restriction to the
AF quotients is already $\Pi^0_2$-hard.
\end{proposition}

\begin{proof}
The upper bound was proved in Proposition~\ref{prop:simple}.
To establish hardness, reduce the standard $\Pi^0_2$-complete set
\[
 \mathrm{INF}=\{\alpha\in2^{\mathbb N}:
                 \alpha(n)=1\text{ for infinitely many }n\}.
\]
Put $F_n=M_{2^n}\oplus M_{2^n}$. For $x,y\in M_{2^n}$ define unital
injective $*$-homomorphisms $\theta_n^0,\theta_n^1:F_n\to F_{n+1}$ by
\[
 \theta_n^0(x,y)=(\operatorname{diag}(x,x),\operatorname{diag}(y,y)),
 \qquad
 \theta_n^1(x,y)=(\operatorname{diag}(x,y),\operatorname{diag}(x,y)).
\]
For $\alpha\in2^{\mathbb N}$ let
\[
 B_\alpha=\varinjlim(F_n,\theta_n^{\alpha(n)}).
\]
Every $B_\alpha$ is a non-zero unital AF algebra.
If $\alpha$ is eventually zero, its inductive system eventually splits
into two independent standard matrix systems. Hence
$B_\alpha\cong M_{2^\infty}\oplus M_{2^\infty}$, which is not simple.
If $\alpha\in\mathrm{INF}$, every non-zero projection in any $F_n$
has non-zero image in both summands after the next connecting map
of type~$1$, and therefore generates the whole algebra at that stage.
Every non-zero ideal of an AF inductive limit with injective maps contains
a non-zero projection from some stage: approximate a non-zero positive
ideal element by a positive element of a stage and take a spectral
projection above a threshold larger than its norm in the quotient by
the ideal. It follows that every non-zero ideal of $B_\alpha$ contains
its unit. Thus $B_\alpha$ is simple exactly when
$\alpha\in\mathrm{INF}$.

We finally describe continuous quotient codes for this family.
Fix a finite unitary generating family for each $F_n$ and enumerate all
these families as one sequence. Let $q_\alpha:A\to B_\alpha$ send the
canonical free generators to the corresponding stage unitaries, and put
$K_\alpha=\ker q_\alpha$. The map $q_\alpha$ is surjective.
For a $*$-polynomial in finitely many free generators, choose $N$ so that
all the corresponding stage unitaries belong to stages at most~$N$.
Its quotient norm is computed already in $F_N$, because the connecting
maps are injective. This norm depends only on
$\alpha(0),\ldots,\alpha(N-1)$ and is therefore locally constant in~$\alpha$.
Density of these polynomials and contractivity of all $q_\alpha$ show
that $\alpha\mapsto\|q_\alpha(a)\|$ is continuous for every $a\in A$.
Thus $\alpha\mapsto K_\alpha$ is Wijsman-continuous, proving the reduction.
\end{proof}

\begin{proposition}\label{prop:nuclear-complete}
Nuclearity is $\Pi^0_2$-complete in the Wijsman quotient coding.
\end{proposition}

\begin{proof}
The upper bound follows from Proposition~\ref{prop:nuclear}.
Let $B=C^*(F_2)$. Choi proved that this unital $C^*$-algebra is
residually finite-dimensional~\cite[Theorem~7]{Choi1980}. It is not nuclear:
its quotient $C_r^*(F_2)$ is non-nuclear by Lance's characterisation
of amenability~\cite{Lance1973}, and nuclearity passes to quotients.
Choose a sequence $(\pi_j)$ of finite-dimensional unital
representations separating points of~$B$, and put
\[
 I_s=\bigcap_{j\leqslant s}\ker\pi_j,
 \qquad I_\infty=\{0\}.
\]
Then every $B/I_s$ is finite-dimensional and non-zero. Moreover,
$\|b+I_s\|\to\|b\|$ for all $b\in B$: the faithful product
representation $b\mapsto(\pi_j(b))_j$ is isometric, and
$\|b+I_s\|=\max_{j\leqslant s}\|\pi_j(b)\|$.
Thus $I_s\to I_\infty$ in the Wijsman topology.

For $\alpha\in2^{\mathbb N\times\mathbb N}$ let $m_r(\alpha)$ be
the first position of a~$1$ in row~$r$, with value~$\infty$ if there
is no such position. The maps $\alpha\mapsto I_{m_r(\alpha)}$ are
continuous. Set
\[
 C=\Bigl(\bigoplus_{r\in\mathbb N}B\Bigr)^{\!\sim},
 \qquad
 J_\alpha=\bigoplus_{r\in\mathbb N}I_{m_r(\alpha)}\subseteq C.
\]
Here the sums are $c_0$-direct sums and the tilde denotes unitisation.
We have
\[
 C/J_\alpha\cong
 \Bigl(\bigoplus_{r\in\mathbb N}B/I_{m_r(\alpha)}\Bigr)^{\!\sim}.
\]
This quotient is nuclear exactly when all its summands are nuclear,
that is, exactly when every row of~$\alpha$ contains a~$1$.

For completeness, continuity of $\alpha\mapsto J_\alpha$ follows
at once from the norm formula. If $c=\lambda1+(b_r)_r\in C$, then
\[
 \|c+J_\alpha\|
 =\max\Bigl\{|\lambda|,
   \sup_r\|\lambda1_B+b_r+I_{m_r(\alpha)}\|\Bigr\}.
\]
The maxima restricted to $r\leqslant R$ are continuous functions
of~$\alpha$, and their error is at most $\sup_{r>R}\|b_r\|$,
which tends to zero uniformly in~$\alpha$.
Finally, fix a surjective $*$-homomorphism $\rho:A\to C$ and take
$K_\alpha=\rho^{-1}(J_\alpha)$.
Lemma~\ref{lem:pullback-ideals} makes this a continuous reduction of
$P_2$ to nuclearity.
\end{proof}

\section{Borel assignment of \texorpdfstring{$K$}{K}-theory}\label{sec:K0}

We give an internal Borel presentation of $K_0$ for the unital quotients
considered here, together with explicit bounds on its membership coordinates.
For $K_1$ we compose the quotient-to-seminorm map with the established
Borel $K_1$ assignment of Farah--Toms--T\"ornquist. Complex Bott
periodicity then supplies the higher groups. An internal treatment of
non-unital $K_0$ would require the relative/unitised construction; it is
not needed for these assignments.

For a unital $C^*$-algebra $B$, $K_0(B)$ is the Grothendieck group
of the monoid $V(B)$ of Murray--von Neumann equivalence classes of
projections in $\bigsqcup_d M_d(B)$, with addition induced by direct sum.
For a non-unital algebra one instead defines
\[
K_0(B)=\ker\bigl(K_0(\widetilde B)\longrightarrow K_0(\C)\bigr),
\]
where the map is induced by the scalar quotient of the unitisation.
For unital $B$, the group $K_1(B)$ is
$\pi_0(\mathrm{GL}_\infty(B))$; in the non-unital case one uses
$\widetilde B$. The elementary suspension isomorphism gives
$K_1(B)\cong K_0(SB)$, where
$SB=C_0((0,1))\otimes B$. Complex Bott periodicity is used below
for the higher groups in Remark~\ref{rem:higher-complex-K}.
For definitions and basic properties we refer to
R{\o}rdam--Larsen--Laustsen~\cite{RordamLarsenLaustsen},
Blackadar~\cite{Blackadar}, and Wegge-Olsen~\cite{WeggeOlsen}.

\begin{remark}\label{rem:K-FTT-compare}
Farah--Toms--T\"ornquist established Borel computation of the
$K_0$ invariant and $K_1$ in the standard concrete parameterisation;
see \cite[Propositions~3.1 and~3.2]{FarahTomsTornquist}.
The Borel translations recalled in Remarks~\ref{rem:other-codings}
and~\ref{rem:Xi-ideals} transfer these results to the present
unital quotient coding. Thus Borel computability is already known.
Our internal $K_0$ presentation additionally records the complexity
of individual subgroup-membership coordinates.
\end{remark}

\subsection{The coding space of countable abelian groups}

Let $\Z^{(\N)}=\bigoplus_{n\in\N}\Z e_n$.  The space
\[
\Sub(\Z^{(\N)})=\{H\subseteq\Z^{(\N)}:H\text{ is a subgroup}\}\subseteq 2^{\Z^{(\N)}}
\]
is a closed subset of a Cantor cube and hence compact Polish.
Each $H$ codes the quotient group $\Z^{(\N)}/H$.

\subsection{An internal Borel presentation of \texorpdfstring{$K_0$}{K0}}

Fix a unital separable $C^*$-algebra quotient generator $A$ and a countable dense
$*$-subalgebra $D^A\subseteq A$.
For each $d\in\N$, let $M_d(D^A)_+$ denote the positive cone of $M_d(D^A)$.
Choose once and for all an enumeration
\[
(a_n)_{n\in\N}
\]
of the countable set
\[
\mathcal P^A:=\bigsqcup_{d\in\N} M_d(D^A)_+
\]
with the following properties:
\begin{enumerate}
\item $a_0=0\in M_1(D^A)$;
\item if $a_m\in M_{d(m)}(D^A)_+$ and $a_n\in M_{d(n)}(D^A)_+$, then there is a
fixed index $m\oplus n$ such that
\[
a_{m\oplus n}=a_m\oplus a_n\in M_{d(m)+d(n)}(D^A)_+.
\]
\end{enumerate}
Since every positive element is of the form $b^*b$ and $D^A$ is dense,
$\mathcal P^A$ is dense in the positive cone of $A\otimes\mathcal K$.

\begin{lemma}
\label{lem:K0-correct-proj}
For every $\varepsilon\in(0,1/4)$ there exist a constant
$\eta(\varepsilon)>0$ and a continuous function
$\chi:[0,\infty)\to[0,1]$ with
\[
\chi|_{[0,1/4]}=0,
\qquad
\chi|_{[3/4,\infty)}=1,
\]
such that for every $C^*$-algebra $B$ and every positive element $x\in B_+$ satisfying
\[
\|x^2-x\|<\eta(\varepsilon),
\]
the element $\widehat x:=\chi(x)$ is a projection and
\[
\|\widehat x-x\|<\varepsilon.
\]
\end{lemma}

\begin{proof}
Choose $\eta(\varepsilon)>0$ small enough that
$|t^2-t|<\eta(\varepsilon)$ for $t\in[0,\infty)$ forces
$t\in[0,\varepsilon)\cup(1-\varepsilon,1+\varepsilon)$.
Functional calculus then gives the asserted spectral projection and the norm estimate.
\end{proof}

\begin{lemma}\label{lem:K0-stable-MvN}
There exists $\delta_{\mathrm{MvN}}>0$ with the following property
for every $C^*$-algebra $B$.
If $p,q\in B_+$ and $v\in B$ satisfy
\[
\|p^2-p\|,\ \|q^2-q\|,\ \|v^*v-p\|,\ \|vv^*-q\|
 <\delta_{\mathrm{MvN}},
\]
then the spectral corrections $P=\chi(p)$ and $Q=\chi(q)$ are
Murray--von Neumann equivalent.
Conversely, if $e,f$ are equivalent projections and $B_0\subseteq B$
is a dense $*$-subalgebra, the above relations have witnesses
$p,q\in(B_0)_+$ and $v\in B_0$ arbitrarily close to $e,f$ and to an
exact partial-isometry witness, respectively.
\end{lemma}

\begin{proof}
Take $\delta_{\mathrm{MvN}}=2^{-10}$ and write
$\rho=3\delta_{\mathrm{MvN}}$.
For a positive $p$ with $\|p^2-p\|<\delta_{\mathrm{MvN}}$,
the scalar spectral estimate gives
$\|p-\chi(p)\|<2\delta_{\mathrm{MvN}}$; similarly for $q$.
Thus
\[
\|v^*v-P\|<\rho,\qquad \|vv^*-Q\|<\rho,
 \qquad \|v\|\leqslant\sqrt{1+\rho}.
\]
In the unitisation, the corner estimates give
\[
\|(1-Q)v\|\leqslant\sqrt\rho,\qquad
\|v(1-P)\|\leqslant\sqrt\rho.
\]
Set $w=QvP$. Then $\|v-w\|\leqslant2\sqrt\rho$ and
$\|w\|\leqslant\|v\|$, whence
\[
\|w^*w-P\|,\ \|ww^*-Q\|
 \leqslant4\sqrt{\rho(1+\rho)}+\rho<1.
\]
Consequently $w^*w$ and $ww^*$ are invertible in their respective
corners. The polar correction $u=w(w^*w)^{-1/2}$ satisfies
$u^*u=P$ and $uu^*=Q$ (with the zero-corner case interpreted
trivially). This proves equivalence.
For the converse, approximate $e$ and $f$ by $x,y\in B_0$ and use
$p=x^*x$, $q=y^*y$, while approximating an exact partial isometry
by $v\in B_0$. Continuity of the finitely many polynomial expressions
gives arbitrarily accurate witnesses.
\end{proof}

Fix $\varepsilon_0\in(0,1/4)$ so small that
$\varepsilon_0<\delta_{\mathrm{MvN}}/4$, and set
\[
\eta_0:=\min\{\eta(\varepsilon_0),\delta_{\mathrm{MvN}}/4\}.
\]
In the sequel, corrected projections are formed using this value of $\eta_0$ and the
corresponding function $\chi$ from Lemma~\ref{lem:K0-correct-proj}.

For $K\in\Ideal(A)$ and $n\in\N$, define the \emph{good approximate projection}
predicate
\[
\mathcal G(K,n)\iff \Phi^{(d(n))}(K,a_n^2-a_n)<\eta_0.
\]
If $\mathcal G(K,n)$ holds, let
\[
p_n^K:=\chi(a_n+M_{d(n)}(K))\in M_{d(n)}(A/K)
\]
be the corrected projection from Lemma~\ref{lem:K0-correct-proj}; if
$\mathcal G(K,n)$ fails, set $p_n^K:=0$.

We also define the \emph{zero-class} predicate
\[
\mathcal Z(K,n)\iff
\neg\mathcal G(K,n)\ \lor\ \Phi^{(d(n))}(K,a_n)<1/2.
\]
If $\mathcal G(K,n)$ holds, then $\mathcal Z(K,n)$ is equivalent to $p_n^K=0$.

\begin{lemma}
\label{lem:K0-dense-proj}
For every $K\in\Ideal(A)$ and every projection
$p\in M_d(A/K)$ there exists $n\in\N$ with $d(n)=d$ such that
$p_n^K$ is Murray--von Neumann equivalent to $p$.
Hence the set of classes $\{[p_n^K]:n\in\N\}$ is all of
$V(A/K)$.
\end{lemma}

\begin{proof}
Fix a projection $p\in M_d(A/K)$.
Choose $\varepsilon>0$ so small that $2\varepsilon+\varepsilon^2<\eta_0$ and
$\varepsilon<1/2$.
Since $\mathcal P^A$ is dense in the positive cone, pick $n$ with $d(n)=d$ and
\[
\|a_n+M_d(K)-p\|<\varepsilon.
\]
Then
\[
\|(a_n+M_d(K))^2-(a_n+M_d(K))\|
\le 2\varepsilon+\varepsilon^2<\eta_0,
\]
so $\mathcal G(K,n)$ holds.
By Lemma~\ref{lem:K0-correct-proj}, the corrected projection $p_n^K$ satisfies
\[
\|p_n^K-(a_n+M_d(K))\|<1/4.
\]
Therefore
\[
\|p_n^K-p\|<1/4+\varepsilon<1.
\]
Projections at distance strictly less than $1$ are unitarily equivalent, hence
Murray--von Neumann equivalent.
\end{proof}

For $m,n\in\N$, let
\[
r(m,n):=d(m)+d(n),
\]
and define
\[
b_{m,n}^{(1)}:=a_m\oplus 0_{d(n)},
\qquad
b_{m,n}^{(2)}:=0_{d(m)}\oplus a_n
\]
as elements of $M_{r(m,n)}(A)_+$.

For $m,n,k\in\N$, let
\[
s(m,n,k):=d(m)+d(n)+d(k),
\]
and define
\[
c_{m,n,k}^{(1)}:=a_m\oplus a_n\oplus 0_{d(k)},
\qquad
c_{m,n,k}^{(2)}:=0_{d(m)+d(n)}\oplus a_k
\]
as elements of $M_{s(m,n,k)}(A)_+$.

Now define relations
\[
\mathcal E\subseteq \Ideal(A)\times\N^2,
\qquad
\mathcal A\subseteq \Ideal(A)\times\N^3
\]
by the following clauses.

\smallskip
\noindent\emph{Equivalence relation.}
We declare $\mathcal E(K,m,n)$ to hold iff either
\[
\mathcal Z(K,m)\wedge \mathcal Z(K,n),
\]
or else there exists
$v\in M_{r(m,n)}(D^A)$ such that
\[
\Phi^{(r(m,n))}\bigl(K,(b_{m,n}^{(1)})^2-b_{m,n}^{(1)}\bigr)<\eta_0,
\]
\[
\Phi^{(r(m,n))}\bigl(K,(b_{m,n}^{(2)})^2-b_{m,n}^{(2)}\bigr)<\eta_0,
\]
and
\[
\Phi^{(r(m,n))}\bigl(K,v^*v-b_{m,n}^{(1)}\bigr)<\delta_{\mathrm{MvN}},
\qquad
\Phi^{(r(m,n))}\bigl(K,vv^*-b_{m,n}^{(2)}\bigr)<\delta_{\mathrm{MvN}}.
\]

\smallskip
\noindent\emph{Addition relation.}
We declare $\mathcal A(K,m,n,k)$ to hold iff one of the following alternatives holds:
\begin{enumerate}[label=\textup{(A\arabic*)}]
\item $\mathcal Z(K,m)$ and $\mathcal E(K,n,k)$;
\item $\mathcal Z(K,n)$ and $\mathcal E(K,m,k)$;
\item there exists $v\in M_{s(m,n,k)}(D^A)$ such that
\[
\Phi^{(s(m,n,k))}\bigl(K,(c_{m,n,k}^{(1)})^2-c_{m,n,k}^{(1)}\bigr)<\eta_0,
\]
\[
\Phi^{(s(m,n,k))}\bigl(K,(c_{m,n,k}^{(2)})^2-c_{m,n,k}^{(2)}\bigr)<\eta_0,
\]
and
\[
\Phi^{(s(m,n,k))}\bigl(K,v^*v-c_{m,n,k}^{(1)}\bigr)<\delta_{\mathrm{MvN}},
\qquad
\Phi^{(s(m,n,k))}\bigl(K,vv^*-c_{m,n,k}^{(2)}\bigr)<\delta_{\mathrm{MvN}}.
\]
\end{enumerate}

\begin{lemma}\label{lem:K0-relations}
The relations $\mathcal Z$, $\mathcal E$, and $\mathcal A$ are Borel in $K$.
Moreover, for every $K\in\Ideal(A)$:
\begin{enumerate}
\item $\mathcal Z(K,n)$ iff $p_n^K=0$;
\item $\mathcal E(K,m,n)$ iff $[p_m^K]=[p_n^K]$ in $V(A/K)$, equivalently iff
\[
p_m^K\oplus 0_{d(n)} \sim 0_{d(m)}\oplus p_n^K
\]
inside $M_{r(m,n)}(A/K)$;
\item $\mathcal A(K,m,n,k)$ iff
\[
[p_m^K]+[p_n^K]=[p_k^K]
\]
in the semigroup $V(A/K)$.
\end{enumerate}
\end{lemma}

\begin{proof}
Borelness is immediate from the definitions: all displayed conditions are strict
inequalities in Wijsman-continuous functions, with existential quantification only over
countable dense sets.

For (1), if $\neg\mathcal G(K,n)$ then by definition $p_n^K=0$.
If $\mathcal G(K,n)$ holds, then $p_n^K=0$ iff the spectrum of $a_n+M_{d(n)}(K)$ is
contained in $[0,1/4]$, which is equivalent to
$\|a_n+M_{d(n)}(K)\|<1/2$.

For (2), the forward implication follows from Lemma~\ref{lem:K0-stable-MvN} applied to
the approximate source/range data in the common amplification $M_{r(m,n)}(A/K)$; the
corrected projections of $b_{m,n}^{(1)}$ and $b_{m,n}^{(2)}$ are exactly
$p_m^K\oplus 0_{d(n)}$ and $0_{d(m)}\oplus p_n^K$.
Conversely, if these two amplified projections are Murray--von Neumann equivalent, approximate
an exact partial-isometry witness by an element of the dense $*$-subalgebra.

For (3), argue exactly as in (2), now comparing
\[
p_m^K\oplus p_n^K\oplus 0_{d(k)}
\quad\text{and}\quad
0_{d(m)+d(n)}\oplus p_k^K.
\]
\end{proof}

\begin{lemma}
\label{lem:Groth-presentation}
Let $M$ be a commutative monoid and let $\nu:\N\to M$ be surjective.
Assume given sets
\[
Z\subseteq \N,\qquad E\subseteq \N^2,\qquad A\subseteq \N^3
\]
such that
\[
\begin{aligned}
n\in Z &\iff \nu(n)=0,\\
(m,n)\in E &\iff \nu(m)=\nu(n),\\
(m,n,k)\in A &\iff \nu(m)+\nu(n)=\nu(k).
\end{aligned}
\]
Let $H\le \Z^{(\N)}$ be the subgroup generated by
\[
\{\mathbf e_n:n\in Z\}\cup
\{\mathbf e_m-\mathbf e_n:(m,n)\in E\}\cup
\{\mathbf e_m+\mathbf e_n-\mathbf e_k:(m,n,k)\in A\}.
\]
Then
\[
\Z^{(\N)}/H\cong G(M),
\]
where $G(M)$ denotes the Grothendieck group of $M$.
\end{lemma}

\begin{proof}
Let $\iota:M\to G(M)$ be the canonical map and define
\[
\pi:\Z^{(\N)}\to G(M),\qquad \pi(\mathbf e_n)=\iota(\nu(n)).
\]
By the defining properties of $Z,E,A$, every generator of $H$ lies in $\ker(\pi)$,
so $\pi$ factors through a surjective homomorphism
\[
\overline\pi:\Z^{(\N)}/H\to G(M).
\]

It remains to prove injectivity.
Take $x\in\ker(\pi)$ and write
\[
x=\sum_{i=1}^r \mathbf e_{m_i}-\sum_{j=1}^s \mathbf e_{n_j}.
\]
Since $\pi(x)=0$ in the Grothendieck group, there exists $c\in M$ such that
\[
\sum_{i=1}^r \nu(m_i)+c=\sum_{j=1}^s \nu(n_j)+c
\]
in $M$.  Choose $\ell\in\N$ with $\nu(\ell)=c$.

We claim that every finite sum of $\nu$-values can be reduced, modulo $H$, to a
single generator.  More precisely, for any indices $q_1,\dots,q_t$ there exists
$u\in\N$ such that
\[
\sum_{i=1}^t \mathbf e_{q_i}-\mathbf e_u\in H
\qquad\text{and}\qquad
\nu(u)=\sum_{i=1}^t \nu(q_i).
\]
This is proved by induction on $t$.
For $t=1$ there is nothing to show.
If $t\ge 2$, choose $u'$ with
\[
\nu(u')=\nu(q_1)+\nu(q_2),
\]
possible by surjectivity of $\nu$.
Then $(q_1,q_2,u')\in A$, so
\[
\mathbf e_{q_1}+\mathbf e_{q_2}-\mathbf e_{u'}\in H,
\]
and the induction hypothesis applies to $u',q_3,\dots,q_t$.

Applying this reduction to the two sums
\[
m_1,\dots,m_r,\ell
\qquad\text{and}\qquad
n_1,\dots,n_s,\ell
\]
yields indices $u,v\in\N$ with
\[
\sum_{i=1}^r \mathbf e_{m_i}+\mathbf e_\ell-\mathbf e_u\in H,
\qquad
\sum_{j=1}^s \mathbf e_{n_j}+\mathbf e_\ell-\mathbf e_v\in H,
\]
and
\[
\nu(u)=\sum_{i=1}^r \nu(m_i)+c=\sum_{j=1}^s \nu(n_j)+c=\nu(v).
\]
Hence $(u,v)\in E$, so $\mathbf e_u-\mathbf e_v\in H$.
Subtracting the two displayed relations and cancelling $\mathbf e_\ell$, we obtain
$x\in H$.
Thus $\ker(\overline\pi)=0$, and $\overline\pi$ is an isomorphism.
\end{proof}

\begin{theorem}\label{thm:K0}
Equip $\Ideal(A)$ with the Wijsman topology.  There is a Borel map
\[
\kappa_{K_0}:\Ideal(A)\to\Sub(\Z^{(\N)})
\]
such that
\[
\Z^{(\N)}/\kappa_{K_0}(K)\cong K_0(A/K)
\qquad (K\in\Ideal(A)).
\]
Moreover, for each fixed $z\in\Z^{(\N)}$, the section
\[
\{K\in\Ideal(A): z\in \kappa_{K_0}(K)\}
\]
is $F_\sigma$ in the Wijsman topology.
\end{theorem}

\begin{proof}
Let $(\mathbf e_n)_{n\in\N}$ denote the standard basis of $\Z^{(\N)}$.
For each $K\in\Ideal(A)$, let $H_K\le \Z^{(\N)}$ be the subgroup generated by the
following countable family:

\begin{enumerate}
\item $\mathbf e_n$ whenever $\mathcal Z(K,n)$ holds;
\item $\mathbf e_m-\mathbf e_n$ whenever $\mathcal E(K,m,n)$ holds;
\item $\mathbf e_m+\mathbf e_n-\mathbf e_k$ whenever $\mathcal A(K,m,n,k)$ holds.
\end{enumerate}

Fix $K$ and write $B=A/K$.
Let
\[
\nu_K:\N\to V(B),\qquad \nu_K(n):=[p_n^K].
\]
By Lemma~\ref{lem:K0-dense-proj}, the map $\nu_K$ is surjective.
By Lemma~\ref{lem:K0-relations}, the relations $\mathcal Z$, $\mathcal E$, and
$\mathcal A$ detect exactly the zero element, equality, and addition in the monoid
$V(B)$.
Therefore Lemma~\ref{lem:Groth-presentation} yields
\[
\Z^{(\N)}/H_K\cong G(V(B))=K_0(B)=K_0(A/K).
\]
Define $\kappa_{K_0}(K):=H_K$.

We next show that the map $K\mapsto H_K$ is Borel as a map into
$\Sub(\Z^{(\N)})$.
Fix $z\in\Z^{(\N)}$.  Then
\[
z\in H_K
\]
iff there exist finitely many generators $g_1,\dots,g_r$ of the above three types and
signs $\varepsilon_j\in\{\pm1\}$ such that
\[
z=\sum_{j=1}^{r}\varepsilon_j g_j
\]
and each of the corresponding activation relations $\mathcal Z$, $\mathcal E$, or
$\mathcal A$ holds at~$K$.
Since the potential generators form a fixed countable set and the activation relations are
Borel by Lemma~\ref{lem:K0-relations}, the set
\[
\{K:z\in H_K\}
\]
is a countable union of Borel sets, hence Borel.
Therefore $K\mapsto H_K$ is Borel as a map into $\Sub(\Z^{(\N)})$.

Finally, the coordinate-complexity claim is read off directly from the construction.
The predicate $\mathcal G(K,n)$ is open, so $\neg\mathcal G(K,n)$ is closed, and
$\mathcal Z(K,n)$ is $F_\sigma$.  The witness alternatives in the definitions of
$\mathcal E$ and $\mathcal A$ are open, while the alternatives involving
$\mathcal Z$ are $F_\sigma$; consequently $\mathcal E(K,m,n)$ and
$\mathcal A(K,m,n,k)$ are $F_\sigma$ predicates in~$K$.
For fixed $z\in\Z^{(\N)}$, membership $z\in H_K$ is equivalent to the existence of a
finite formal derivation of $z$ as a signed sum of generators of types~\textup{(1)}--\textup{(3)},
together with the corresponding activation relations.
Hence
\[
\{K\in\Ideal(A): z\in H_K\}
\]
is a countable union of finite intersections of $F_\sigma$ sets, therefore it is
$F_\sigma$.
\end{proof}

\begin{remark}
\label{rem:K0-general}
Inspection of the proof of Theorem~\ref{thm:K0} shows that it applies verbatim to
any separable \emph{unital} $C^*$-algebra $C$.
For a non-unital algebra $B$, however, one has
\[
K_0(B)=\ker\bigl(K_0(\widetilde B)\to K_0(\C)\cong \Z\bigr),
\]
so an internal treatment requires the relative/unitised picture rather than the
plain projection monoid $V(B)$.
For the unital quotients studied here, the established Borel $K_1$
assignment and complex Bott periodicity avoid this additional construction.
\end{remark}

\subsection{The \texorpdfstring{$K_1$}{K1}-assignment}\label{subsec:K1}

The unital quotients can be treated directly by the standard Borel
$K_1$ assignment. We spell out the final change of group coding, since
it must accommodate finite as well as infinite $K_1$ groups.

\begin{theorem}\label{thm:K1}
Equip $\Ideal(A)$ with the Wijsman topology. There is a Borel map
\[
\kappa_{K_1}:\Ideal(A)\longrightarrow\Sub(\Z^{(\N)})
\]
such that
\[
\Z^{(\N)}/\kappa_{K_1}(K)\cong K_1(A/K)
\qquad(K\in\Ideal(A)).
\]
\end{theorem}

\begin{proof}
For proper ideals $K$, Lemma~\ref{lem:fixed-ideal-to-Xi} gives
continuously a unital seminorm code for $A/K$.
Compose this map with the standard Borel translation to the concrete
coding $\Gamma$, and then with the Borel $K_1$ assignment of
\cite[Proposition~3.2]{FarahTomsTornquist}.
The result is a Borelly assigned abelian group $G_K\cong K_1(A/K)$.
The group coding in \cite[Section~3.1]{FarahTomsTornquist} is a
countable disjoint union of operation-table spaces on
$\{0,\ldots,m-1\}$, $m\geqslant1$, and on $\N$.

For such a group $G$, choose its surjective enumeration
$g_j^G=j$ when its domain is $\N$, and
$g_j^G=j\bmod m$ when its domain has size $m$.
These labels are elements of the coded group, so the enumeration is
Borel in $G$ and permits repetitions for finite groups.
Let
\[
H_G=\left\{z\in\Z^{(\N)}:
             \sum_j z_j g_j^G=0_G\right\}.
\]
The finite sum is evaluated using the coded group operations.
The identity and inverse operations are Borel functions of an
operation-table code, so $\{G:z\in H_G\}$ is Borel for every fixed
$z$. Consequently $G\mapsto H_G$ is Borel as a map into
$\Sub(\Z^{(\N)})$.
Moreover $H_G$ is the kernel of the surjective homomorphism
$e_j\mapsto g_j^G$, and hence $\Z^{(\N)}/H_G\cong G$.
Set $\kappa_{K_1}(K)=H_{G_K}$ for proper $K$.
Finally, $\{A\}$ is clopen in $\Ideal(A)$, since
$\dist(1,K)$ equals $1$ for proper ideals and $0$ for $K=A$.
Defining $\kappa_{K_1}(A)=\Z^{(\N)}$ handles the zero quotient
and completes the construction.
\end{proof}

\begin{remark}\label{rem:higher-complex-K}
For $m\geqslant0$, let $S^m(B)$ be the $m$-fold suspension,
with $S^0(B)=B$, and define the higher complex groups by
\[
K_m(B):=K_0(S^m(B)).
\]
For $m=1$ this agrees with the usual $K_1$ through the suspension
isomorphism recalled above. Complex Bott periodicity gives
\[
K_{m+2}(B)\cong K_m(B).
\]
Consequently a Borel subgroup code for $K_m(A/K)$ is obtained by
setting
\[
\kappa_{K_m}(K)=
\begin{cases}
\kappa_{K_0}(K),&m\text{ even},\\
\kappa_{K_1}(K),&m\text{ odd}.
\end{cases}
\]
This proves Borel computability of every higher complex $K$-group
without constructing codes for the non-unital suspensions.
No finite Wijsman rank is claimed for the standard-coding
$K_1$ assignment used here.
\end{remark}

\begin{lemma}\label{lem:Sub-map-complexity}
Let $X$ be Polish and $\Theta:X\to\Sub(\Z^{(\N)})$.
Suppose $A_z:=\{x:z\in\Theta(x)\}$ is $F_\sigma$ for each
$z\in\Z^{(\N)}$. Then preimages of basic clopen cylinders are
$\Delta^0_3$, and $\Theta$ is $\Sigma^0_3$-measurable.
\end{lemma}

\begin{proof}
For finite $E,F\subseteq\Z^{(\N)}$, the preimage of
\[
U_{E,F}=\{H:E\subseteq H,\ F\cap H=\varnothing\}
\]
is
\[
\Bigl(\bigcap_{z\in E}A_z\Bigr)
 \setminus\Bigl(\bigcup_{w\in F}A_w\Bigr),
\]
a difference of two $F_\sigma$ sets and hence a $\Delta^0_3$ set.
Every open set in the target is a countable union of such cylinders,
so its preimage is $\Sigma^0_3$.
\end{proof}

\begin{corollary}\label{cor:K0-map-complexity}
The map $\kappa_{K_0}$ of Theorem~\ref{thm:K0} is
$\Sigma^0_3$-measurable. More precisely, positive membership coordinates
have $F_\sigma$ preimages, basic clopen cylinders have $\Delta^0_3$
preimages, and open sets have $\Sigma^0_3$ preimages.
\end{corollary}

\begin{proof}
Apply Lemma~\ref{lem:Sub-map-complexity} to Theorem~\ref{thm:K0}.
\end{proof}

\section{Tensor products and kernel maps}\label{sec:tensor-products}

Tensor products require a distinction between the ideal generated by the
factor ideals and the kernel of the tensor product of the quotient maps.
Let $A$ and $B$ be separable $C^*$-algebras. For
$\bullet\in\{\min,\max\}$ put
\[
\mathcal J_\bullet(I,J)
 :=\overline{I\odot B+A\odot J}^{\,\|\cdot\|_\bullet}
 \triangleleft A\otimes_\bullet B.
\]
For the maximal tensor product this always equals
$\ker(q_I\otimes_{\max}q_J)$. For the minimal tensor product we
assert this equality only when the relevant exactness hypotheses have been
checked. In particular, for a fixed exact algebra $D$,
\begin{equation}\label{eq:fixed-exact-kernel}
\ker(q_I\otimes\mathrm{id}_D)
 =\overline{I\odot D},\qquad
(A\otimes_{\min}D)/\overline{I\odot D}
 \cong(A/I)\otimes_{\min}D.
\end{equation}
This is precisely the defining preservation-of-exact-sequences property of
$D$; see \cite[Chapter~2]{BrownOzawa}.
No exactness hypothesis is needed for the regularity of the elementary ideal
assignment proved next.

\begin{proposition}\label{prop:tensor-usc}
For each $z\in A\otimes_\bullet B$, the function
\[
(I,J)\longmapsto\dist(z,\mathcal J_\bullet(I,J))
\]
is upper semicontinuous for the Wijsman topologies.
\end{proposition}

\begin{proof}
Suppose $(I_n,J_n)\to(I,J)$ Wijsman and fix $\varepsilon>0$.
Choose
\[
w=\sum_{k=1}^{m}(i_k\otimes b_k+a_k\otimes j_k)
\]
with $i_k\in I$, $j_k\in J$, and
$\|z-w\|_\bullet<\dist(z,\mathcal J_\bullet(I,J))+\varepsilon$.
Wijsman convergence permits choices $i_k^{(n)}\in I_n$ and
$j_k^{(n)}\in J_n$ converging in norm to $i_k$ and $j_k$.
The corresponding tensors $w_n$ belong to
$I_n\odot B+A\odot J_n$ and converge in norm to $w$.
Consequently
\[
\limsup_n\dist(z,\mathcal J_\bullet(I_n,J_n))
 \leqslant\|z-w\|_\bullet
 <\dist(z,\mathcal J_\bullet(I,J))+\varepsilon.
\]
Letting $\varepsilon\downarrow0$ proves the assertion.
\end{proof}

\begin{proposition}\label{prop:tensor-borel}
The map
\[
\mathfrak T_\bullet:\Ideal(A)\times\Ideal(B)
 \longrightarrow\Ideal(A\otimes_\bullet B),\qquad
 (I,J)\longmapsto\mathcal J_\bullet(I,J),
\]
is $\Sigma^0_2$-measurable for the Wijsman topologies: the preimage
of every open set is $F_\sigma$. In particular it is Borel.
For a fixed exact $D$, the same conclusion holds for the kernel map
$I\mapsto\ker(q_I\otimes\mathrm{id}_D)$ for the minimal tensor product.
\end{proposition}

\begin{proof}
For each fixed $z$ and real $r$, Proposition~\ref{prop:tensor-usc}
shows that the preimage of $\{L:\dist(z,L)<r\}$ is open.
The preimage of $\{L:\dist(z,L)>r\}$ is $F_\sigma$, since
\[
\{(I,J):\dist(z,\mathcal J_\bullet(I,J))>r\}
 =\bigcup_{\substack{q\in\Q\\q>r}}
 \{(I,J):\dist(z,\mathcal J_\bullet(I,J))\geqslant q\}
\]
and every set on the right is closed. A countable Wijsman subbasis,
finite intersections, and countable unions give the asserted bound for
arbitrary open sets. The last assertion follows from
\eqref{eq:fixed-exact-kernel}.
\end{proof}

\begin{theorem}\label{thm:tensor-continuity}
Let $D$ be a fixed separable nuclear $C^*$-algebra. Then
\[
\Ideal(A)\longrightarrow\Ideal(A\otimes D),\qquad
 I\longmapsto\overline{I\odot D}
\]
is Wijsman-continuous. Here the minimal and maximal tensor products agree.
\end{theorem}

\begin{proof}
Let $I_n\to I$ Wijsman and first consider
$w=\sum_{r=1}^{s}a_r\otimes d_r\in A\odot D$.
Nuclearity gives completely positive contractions
$\alpha:D\to M_k$ and $\beta:M_k\to D$ such that
\[
\sum_{r=1}^{s}\|a_r\|\,\|d_r-\beta\alpha(d_r)\|<\varepsilon.
\]
For every ideal $L\triangleleft A$ write
$w_L=\sum_r(a_r+L)\otimes d_r$ and
$u_L=(\mathrm{id}_{A/L}\otimes\alpha)(w_L)$.
Contractivity of both tensorised maps gives
\[
\|u_L\|\leqslant\|w_L\|
 \leqslant\|u_L\|+\varepsilon.
\]
The element $u_L$ is the image of the fixed matrix
$u=(\mathrm{id}_A\otimes\alpha)(w)\in M_k(A)$.
Continuity of matrix quotient norms therefore gives
$\|u_{I_n}\|\to\|u_I\|$.
Since $\varepsilon>0$ is arbitrary,
$\|w_{I_n}\|\to\|w_I\|$.
Approximating an arbitrary $z\in A\otimes D$ by algebraic tensors,
using contractivity of all quotient maps, gives the same conclusion for $z$.
Nuclearity implies exactness, so \eqref{eq:fixed-exact-kernel} identifies
these norms with the distances to $\overline{I_n\odot D}$ and
$\overline{I\odot D}$, as required.
\end{proof}

\begin{remark}\label{rem:tensor-cont}
The fixed nuclear-factor case includes suspension and stabilisation.
Joint continuity for two varying ideals is a different question and is not
asserted here. Nor does an arbitrary maximal tensor product of continuous
fields have continuous fibre norms: the relevant continuity results have
substantive hypotheses; see \cite{KirchbergWassermann1995}.
We state the descriptive conclusion in Proposition~\ref{prop:tensor-borel}
as $\Sigma^0_2$-measurability. For maps with a general metric target this
must not be identified automatically with being a pointwise limit of
continuous maps.
\end{remark}

\subsection{Projective and injective tensor products of Banach algebras}

Let $A$ and $B$ be separable Banach algebras and
$\alpha\in\{\pi,\varepsilon\}$. Define
\[
\mathcal J_\alpha(I,J)
 :=\overline{I\odot B+A\odot J}^{\,\|\cdot\|_\alpha}
 \subseteq A\widehat\otimes_\alpha B.
\]
For $\alpha=\pi$ the completed tensor product is a Banach algebra and
$\mathcal J_\pi(I,J)$ is an ideal. For $\alpha=\varepsilon$ this is
asserted only when the algebraic multiplication extends continuously to
the completion. Without that assumption, the same construction is a map
into the space of closed linear subspaces.

\begin{proposition}\label{prop:banach-tensor-usc}
For $\alpha\in\{\pi,\varepsilon\}$ and
$z\in A\widehat\otimes_\alpha B$, the function
\[
(I,J)\longmapsto\dist(z,\mathcal J_\alpha(I,J))
\]
is upper semicontinuous in the Wijsman topologies.
\end{proposition}

\begin{proof}
The proof of Proposition~\ref{prop:tensor-usc} applies unchanged:
for either norm, $\|x\otimes y\|_\alpha=\|x\|\,\|y\|$,
so the finite tensor witnesses converge in norm.
\end{proof}

\begin{proposition}\label{prop:banach-tensor-borel}
For either tensor norm, the elementary ideal assignment is
$\Sigma^0_2$-measurable in the Wijsman topologies, regarded as a map into
closed linear subspaces, or into closed ideals when the completed tensor
product is a Banach algebra. It is therefore Borel for any of the
admissible topologies having the same Borel sets.
\end{proposition}

\begin{proof}
Apply the subbasic-open-set argument of
Proposition~\ref{prop:tensor-borel} to
Proposition~\ref{prop:banach-tensor-usc}.
\end{proof}

\section{\texorpdfstring{$D$}{D}-absorption}\label{sec:D-stability}

In this section $A=C^*_{\max}(F_\infty)$ is again the fixed unital
quotient generator. For a fixed separable unital $C^*$-algebra $D$, a separable unital
$C^*$-algebra $B$ is \emph{$D$-absorbing} if
$B\cong B\otimes_{\min}D$. For strongly self-absorbing $D$ this
is the usual notion of $D$-stability \cite{TomsWinter}.
The complexity of the general isomorphism condition and that of
strongly self-absorbing stability should be distinguished: the former
gives an analytic upper bound, whereas the latter has a $G_\delta$
description.

\begin{remark}\label{rem:isom-analytic}
Isomorphism of separable $C^*$-algebras is analytic in a standard coding
with named dense sequences. Indeed, the images of a countable dense
$\Q(i)$-$*$-subalgebra can be coded by sequences of Cauchy sequences
of names in the other algebra. Being Cauchy, respecting the algebraic
operations, preserving norms, and having dense range are countable
Borel conditions. Their satisfaction is equivalent to extension to a
surjective isometric $*$-homomorphism. Projection over the sequence
codes therefore gives an analytic relation. This argument uses only
countably coded maps. For comparison, Sabok proved that isomorphism
of separable $C^*$-algebras, already of simple separable AI algebras,
is complete among orbit equivalence relations
\cite{SabokCompleteness}.
\end{remark}

\begin{theorem}\label{thm:Dstab}
For a fixed separable unital exact $C^*$-algebra $D$, the set
\[
\{K\in\Ideal(A):A/K\cong(A/K)\otimes_{\min}D\}
\]
is analytic.
\end{theorem}

\begin{proof}
By \eqref{eq:fixed-exact-kernel} and
Proposition~\ref{prop:tensor-borel}, the assignment
\[
K\longmapsto\overline{K\odot D}
 \triangleleft A\otimes_{\min}D
\]
is Borel and its associated quotient is $(A/K)\otimes_{\min}D$.
Passing to standard seminorm codes, both $A/K$ and its tensor product
therefore depend Borelly on $K$. Pull back the analytic isomorphism
relation of Remark~\ref{rem:isom-analytic} along this pair of Borel
assignments.
\end{proof}

\begin{theorem}\label{thm:Dstab-SSA}\label{rem:Dstab-SSA}
For every fixed strongly self-absorbing $C^*$-algebra $D$, the class
of ideals $K$ for which $A/K$ is $D$-stable is $\Pi^0_2$ in the
Wijsman topology.
\end{theorem}

\begin{proof}
Corollary~4.3.2 of
\cite{FarahHartLupiniRobertTikuisisVignatiWinter} states that, for a
fixed strongly self-absorbing $D$, $D$-stability is
$\forall\exists$-axiomatisable among separable $C^*$-algebras;
the same assertion appears in Theorem~2.5.2(21) there.
Restrict to the unital algebras coded here and choose a countable
axiomatisation
\[
 \sup_{\bar x}\inf_{\bar y}\theta_j(\bar x,\bar y)=0,
       \qquad j\in\N,
\]
where every variable ranges over its specified bounded sort and
$\theta_j$ is non-negative and quantifier-free.
For each $j$, the singleton family consisting of the existential
formula $\inf_{\bar y}\theta_j(\bar x,\bar y)$ is uniform: every
formula has a modulus on its prescribed sorts. Apply
Lemma~\ref{lem:uniform-existential}. The possible isolated zero quotient
is handled separately as in that lemma.
\end{proof}

This applies, in particular, to infinite-type UHF algebras,
$\mathcal O_2$, $\mathcal O_\infty$ and the Jiang--Su algebra.

\begin{theorem}\label{thm:CAR-stability-complete}
Let $A=C^*_{\max}(F_\infty)$ and let
$D=\bigotimes_{n\in\N}M_2$ be the CAR algebra.
The following classes are $\Pi^0_2$-complete in $\Ideal(A)$:
\begin{enumerate}
\item the ideals $K$ for which $A/K$ is $D$-stable;
\item the ideals $K$ for which $A/K$ is approximately divisible.
\end{enumerate}
The hardness already holds within the class of simple unital AF quotients.
\end{theorem}

\begin{proof}
The first upper bound is Theorem~\ref{thm:Dstab-SSA}, and the second
is Proposition~\ref{prop:approx-div}.
The latter is also covered by
\cite[Theorem~2.5.2(22)]{FarahHartLupiniRobertTikuisisVignatiWinter},
whose proof in \S4.5 uses the same definition adopted here:
finite sets are approximately centralised by the entire unit ball of
a unital finite-dimensional subalgebra with no abelian summand.
That section explains that the approximating algebra can be taken to
be $M_2$, $M_3$ or $M_2\oplus M_3$; it need not have two non-zero
summands.

For hardness, use the standard $\Pi^0_2$-complete set
\[
\mathrm{Inf}=\{x\in2^{\N}:x(n)=1\text{ for infinitely many }n\}.
\]
Apply Lemma~\ref{lem:matrix-CAR-family} and write $K_x=\ker\pi_x$.
Its quotient is a full matrix
algebra when $x$ has finite support and the CAR algebra when
$x\in\mathrm{Inf}$.

The CAR algebra absorbs itself, whereas no non-zero finite-dimensional
algebra does. It is approximately divisible under the definition used
here: given a finite set $F\subseteq D$ and $\varepsilon>0$, choose
elements $a_x$ in an initial tensor block with
$\|x-a_x\|<\varepsilon/2$ for $x\in F$. A single subsequent tensor
factor is a unital copy of $M_2$ with no abelian summand and commutes
with that initial block. Hence, for every $z$ in its unit ball,
\[
 \|[x,z]\|=\|[x-a_x,z]\|
 \leqslant2\|x-a_x\|\|z\|<\varepsilon.
\]
No full matrix algebra is approximately divisible.
Indeed, choose a finite spanning set of $M_k$. Uniformly small
commutators with this set force a contraction to be close to a scalar,
by equivalence of norms on $M_k/\C1$. This is impossible for a
non-trivial projection in a proposed approximately central subalgebra
with no abelian summand: its distance to $\C1$ is $1/2$.
Consequently $x\in\mathrm{Inf}$ is equivalent to either property
of $A/K_x$, proving both completeness assertions.
\end{proof}

\section{TROs and ternary ideals}\label{sec:TRO}

A ternary ring of operators (TRO) is a closed subspace
$T\subseteq\mathcal B(H,K)$ closed under $[x,y,z]=xy^*z$.
A closed subspace $J\subseteq T$ is a ternary ideal if
\[
    [T,T,J]+[T,J,T]+[J,T,T]\subseteq J.
\]

A ternary homomorphism means a complex-linear map preserving the
ternary product. We record the quotient-norm consequence of the
linking-algebra extension.

\begin{lemma}\label{lem:TRO-quotient-norm}
Every complex-linear ternary homomorphism $\varphi:T\to S$ between
TROs is contractive, has closed range, and induces an isometric ternary
isomorphism $T/\ker\varphi\to\varphi(T)$.
\end{lemma}
\begin{proof}
Hamana's linking-algebra extension theorem
\cite{Hamana1999}, as recalled explicitly in
\cite[Section~1]{SalmiSkalski2017}, extends $\varphi$ to a
$*$-homomorphism
\[
 \widetilde\varphi:\mathcal L(T)\longrightarrow\mathcal L(S),
 \qquad
 \mathcal L(T)=
 \begin{pmatrix}
  \overline{\operatorname{span}}TT^*&T\\
  T^*&\overline{\operatorname{span}}T^*T
 \end{pmatrix},
\]
whose upper-right corner is $\varphi$ and which preserves all four
corners.  Contractivity follows from that of $*$-homomorphisms.
Identify $T$ with the upper-right corner and put
$I=\ker\widetilde\varphi$.
The contractive upper-right corner projection sends $I$ onto
$\ker\varphi$, so the $C^*$-algebra quotient norm formula gives
\[
 \|\varphi(t)\|
 =\operatorname{dist}_{\mathcal L(T)}(t,I)
 =\operatorname{dist}_T(t,\ker\varphi),\qquad t\in T.
\]
Thus the induced ternary map is isometric, and its range is complete,
hence closed.
\end{proof}

\begin{lemma}\label{lem:freeTRO}
There exists a separable TRO $T_{\mathrm{free}}$ generated by a sequence
$(u_n)$ of contractions such that for every separable TRO $T$ and every
sequence $(t_n)\subseteq \operatorname{Ball}(T)$ there exists a unique
contractive ternary homomorphism
$\pi:T_{\mathrm{free}}\to T$
with $\pi(u_n)=t_n$ for all $n$.
In particular every separable TRO is a quotient of $T_{\mathrm{free}}$ by a
closed ternary ideal.
\end{lemma}

\begin{proof}
Let $\mathcal T_{\mathrm{alg}}$ be the algebraic free ternary system on symbols
$u_n$.  Define
\begin{multline*}
\|x\|_{\max}
:=
\sup\bigl\{\|\rho(x)\|:\ \rho:\mathcal T_{\mathrm{alg}}\to S
\text{ is a ternary homomorphism}\\
\text{into a TRO }S,\ 
\|\rho(u_n)\|\le 1\ \forall n\bigr\}.
\end{multline*}
An induction on ternary word length shows that $\|x\|_{\max}<\infty$ for every
$x\in \mathcal T_{\mathrm{alg}}$: if $x=[x_1,x_2,x_3]$, then
$\|\rho(x)\|\le\|\rho(x_1)\|\,\|\rho(x_2)\|\,\|\rho(x_3)\|$,
and the right-hand side is uniformly bounded over all admissible $\rho$ by the
inductive hypothesis.

Let $N=\{x:\|x\|_{\max}=0\}$, set $T_0=\mathcal T_{\mathrm{alg}}/N$, and let
$T_{\mathrm{free}}$ be the completion of $T_0$.  By construction the images of
the $u_n$ are contractions. To see that the completion is a concrete TRO,
choose a countable collection of representations which norms every rational
ternary polynomial (choose witnesses within $1/m$ of each supremum).
Their diagonal product gives an isometric ternary representation of $T_0$;
its closure in the product TRO is the required concrete realisation.

Given $(t_n)\subseteq \operatorname{Ball}(T)$, the algebraic universal property
gives a unique ternary homomorphism
$\rho:\mathcal T_{\mathrm{alg}}\to T$ with $\rho(u_n)=t_n$.
By definition of $\|\cdot\|_{\max}$, $\rho$ is contractive, so it factors
through $T_0$ and extends uniquely to a contractive ternary homomorphism
$\pi:T_{\mathrm{free}}\to T$.

Finally, if $T$ is separable, choose a dense sequence $(t_n)$ in
$\operatorname{Ball}(T)$.  The resulting map $\pi$ has dense range.
Lemma~\ref{lem:TRO-quotient-norm} shows that its range is also closed,
so $\pi$ is surjective and induces an isometric ternary isomorphism
$T_{\mathrm{free}}/\ker\pi\cong T$.
\end{proof}

\begin{proposition}\label{prop:TRO-polish}
The space of closed ternary ideals in $T_{\mathrm{free}}$ is closed in
$\SB(T_{\mathrm{free}})$ for any admissible topology, hence Polish.
Every separable TRO is isometrically ternary isomorphic to
$T_{\mathrm{free}}/J$ for some closed ternary ideal $J$.
\end{proposition}

\begin{proof}
Apply Proposition~\ref{prop:ideal-closed} to the ternary operation
$[x,y,z]=xy^*z$.
The final assertion follows from Lemma~\ref{lem:freeTRO}.
\end{proof}

Write $V=T_{\mathrm{free}}/J$ and let $D^T\subseteq T_{\mathrm{free}}$ be a fixed
countable dense subset; set $\Phi_T(J,x)=\|x+J\|_{V}$.

\begin{proposition}\label{prop:TRO-complexity}
The class of ternary ideals $J$ for which
$[x,y,z]=[z,y,x]$ in $T_{\mathrm{free}}/J$ for all $x,y,z$
is Wijsman-closed.
\end{proposition}
\begin{proof}
It suffices by continuity to impose
$\Phi_T(J,[x,y,z]-[z,y,x])=0$ for $x,y,z\in D^T$.
This is a countable intersection of closed conditions.
\end{proof}

\section{Countable structures: a complexity catalogue}
\label{sec:countable-complexity}

All spaces in this section are compact zero-dimensional Polish spaces (closed subspaces
of Cantor cubes), and atomic predicates are clopen.  For a fixed word $w$ in the free
object, the predicate ``$[w]=1$ in the quotient'' is clopen.  Consequently:
one universal quantifier over a countable domain yields a closed
($\Pi^0_1$) condition;
one existential quantifier yields an open ($\Sigma^0_1$) condition;
alternations proceed in the expected way.

\subsection{Countable groups}\label{subsec:groups}

We code $G=F_\infty/N$ with $N\in\NSub(F_\infty)\subseteq 2^{F_\infty}$.

\paragraph{Order and structure.}
\begin{itemize}
\item \emph{Finite:} $\Sigma^0_2$ ($\exists$ finite transversal
$T\subseteq F_\infty$ s.t.\ $\forall g\in F_\infty\;\bigvee_{t\in T} gt^{-1}\in N$;
each membership test is clopen, so the inner condition is closed, and the outer
existential over finite $T$ gives $\Sigma^0_2$).
\item \emph{Finitely generated:} $\Sigma^0_3$ ($\exists$ finite $S\subseteq F_\infty$
s.t.\ $\forall g\in F_\infty$, $g\in\langle S\rangle N$).
\item \emph{Cyclic:} $\Sigma^0_3$ ($\exists g\;\forall h\;\exists n:\ h\in\langle g\rangle N$).
\item \emph{Abelian:} closed ($[F_\infty,F_\infty]\subseteq N$).
\item \emph{Simple:} $\Pi^0_2$, with the trivial group excluded.
A non-trivial $G_N$ is simple iff for every $g$ and every generator $x_i$,
either $g\in N$ or there exist
$w_1,\dots,w_k\in F_\infty$ and signs $\varepsilon_j\in\{\pm1\}$ such that
$x_i(\prod_{j=1}^k w_j g^{\varepsilon_j} w_j^{-1})^{-1}\in N$.
Since $g\in N$ is clopen and the existentially quantified membership test is
a countable union of clopen conditions (hence open), each instance
of the outer $\forall g\;\forall x_i$ is open, and the intersection is $G_\delta=\Pi^0_2$.
Non-triviality is open, so adjoining it preserves this bound.
\end{itemize}

\paragraph{Commutativity and centrality.}
\begin{itemize}
\item \emph{Nilpotent of class at most $c$:} closed, since
\[
 F_\infty/N\text{ has nilpotency class at most }c
 \quad\Longleftrightarrow\quad
 \gamma_{c+1}(F_\infty)\subseteq N.
\]
Here $\gamma_1(F_\infty)=F_\infty$ and
$\gamma_{j+1}(F_\infty)=[\gamma_j(F_\infty),F_\infty]$;
the lower central series commutes with surjective homomorphisms.
The inclusion on the right is an intersection of clopen
membership conditions.  Nilpotence is the union of these
closed classes over finite $c$, hence is $\Sigma^0_2$.
\item \emph{Soluble of derived length $\le d$:} closed.  ``Soluble'' is $\Sigma^0_2$.
\item \emph{Perfect ($G=[G,G]$):} $G_\delta$ ($\Pi^0_2$).
For each generator~$x_i$, the condition
$\exists w\in[F_\infty,F_\infty]:\ x_iw^{-1}\in N$ is a countable union of clopen sets,
hence open.
Intersecting over all generators gives $G_\delta$.
\item \emph{Centreless:} $\Pi^0_2$
($\forall g\;\exists h:\ g\in N\lor [g,h]\notin N$; the disjunction is
clopen $\cup$ clopen = clopen, so $\forall\exists(\text{clopen})$ gives $\Pi^0_2$).
\item \emph{Non-trivial centre:} $\Sigma^0_2$
($\exists g:\ g\notin N\land\forall h:\ [g,h]\in N$; the conjunction of a
clopen set with a countable intersection of clopen sets is closed,
so $\exists(\text{closed})$ gives $\Sigma^0_2$).
\end{itemize}

\paragraph{Torsion.}
\begin{itemize}
\item \emph{Torsion-free:} closed
($\forall g\;\forall n\ge 2:\ g^n\in N\implies g\in N$).
\item \emph{Periodic:} $\Pi^0_2$ ($\forall g\;\exists n\ge 2:\ g^n\in N$).
\item \emph{Locally finite:} $\Pi^0_2$-complete. For each fixed finite
tuple, finiteness of its generated subgroup is open: a finite multiplication
and inverse table gives a clopen witness. See
Theorem~\ref{thm:locally-finite-complete}.
\end{itemize}

\paragraph{Group extensions and virtual properties.}
The remaining properties in this direction
(finite-by-cyclic, finite-by-abelian, polycyclic, cyclic-by-finite, virtually
cyclic, virtually abelian, virtually nilpotent, residually finite, Hopfian, and
related variants) require additional coding of finite quotients, finite-index
subgroups, or endomorphisms.  Their precise Borel complexity is not established
in the present paper, so we do not assign ranks to them here.

We record only the easy derived-series case:
\begin{itemize}
\item \emph{Metabelian:} closed, since $G_N$ is metabelian iff
$F_\infty^{(2)}\subseteq N$.
\end{itemize}

\subsubsection{Quotient universality}

Every countable group is a quotient of $F_\infty$; its normal-subgroup
space therefore provides a single compact parameter space for all countable
groups.  This is a quotient-universality statement, directly parallel to
the use of $C^*_{\max}(F_\infty)$ for separable unital $C^*$-algebras.
It is unrelated to universality for embeddings.
For a fixed finite generating set, the normal-subgroup topology is the
usual topology of marked groups.  Its systematic study goes back to
Grigorchuk~\cite{Grigorchuk1984}; the present coding permits countably
many generators and fits the same local-equation viewpoint.

\begin{proposition}\label{prop:closed-group-hardness}
Let $\mathcal C$ be a proper isomorphism-invariant class of countable
groups containing the trivial group.  If its locus in
$\NSub(F_\infty)$ is closed, then this locus is $\Pi^0_1$-complete.
\end{proposition}
\begin{proof}
Choose a countable group $G\notin\mathcal C$ and a sequence
$(g_j)_{j\in\N}$ generating $G$.  For $\alpha\in2^\N$,
define a homomorphism $\pi_\alpha:F_\infty\to G$ as follows.
If $\alpha=0^\N$, send every free generator to $1$.
Otherwise, let $m$ be the first index with $\alpha(m)=1$ and put
\[
 \pi_\alpha(x_j)=
 \begin{cases}
  1,&j\leqslant m,\\
  g_{j-m-1},&j>m.
 \end{cases}
\]
Set $N_\alpha=\ker\pi_\alpha$.  The quotient is trivial when
$\alpha=0^\N$ and is isomorphic to $G$ otherwise.
For a word involving only $x_0,\ldots,x_r$, its image is determined
by $\alpha(0),\ldots,\alpha(r)$: if no one occurs in this initial
segment, its image is $1$, regardless of the later coordinates.
Thus $\alpha\mapsto N_\alpha$ is continuous and reduces the
$\Pi^0_1$-complete singleton $\{0^\N\}$ to $\mathcal C$.
\end{proof}

\begin{corollary}\label{cor:closed-group-complete}
Abelianity and torsion-freeness are $\Pi^0_1$-complete in
$\NSub(F_\infty)$.  The same holds for nilpotency of class at most
$c$ and solubility of derived length at most $d$, for each fixed
finite $c,d$.  Torsion-freeness is also $\Pi^0_1$-complete in
$\Sub(\Z^{(\N)})$.
\end{corollary}
\begin{proof}
For the group assertions apply Proposition~\ref{prop:closed-group-hardness}
to the closed classes recorded above.  All contain the trivial group
and all are proper.  For the abelian coding, the same proof with
the free abelian source and the excluded group $\Z/2\Z$ applies.
\end{proof}

\subsubsection{Sofic groups}

Soficity was introduced by Gromov~\cite{Gromov1999}; we use its
finite-permutation approximation characterisation, as presented in
Pestov's survey~\cite[Theorem~3.5]{PestovSofic2008}.
Closedness of the sofic locus is a classical consequence of this local
characterisation.  We include the short argument in the present
countably generated quotient coding, and then obtain completeness from
the existence of a non-sofic group.

The latter input is supplied by OpenAI's construction in
\cite[Chapter~3, Theorem~1.1]{OpenAINonsofic2026}: the countable
unit group $L_{\mathbb F_2}(1,2)^\times$ of the binary Leavitt algebra
is non-sofic. The argument uses Kun's expander decomposition for sofic approximations
of property~$(T)$ groups~\cite[Theorem~1]{KunSofic2016} and the
expander-centraliser obstruction of Kun and Thom
\cite[Theorem~1.1]{KunThomInapproximability2019}.
Fournier-Facio~\cite{FournierFacio2026} gives a torsion-free construction
using the same technical criterion, while Kun and Thom
\cite{KunThom2026} develop the underlying rigidity mechanism to obtain
further examples.\footnote{These 2026 constructions are cited as preprints; no journal
versions are recorded in the cited sources as of 8 October 2026.
Only the existence of a countable non-sofic group is used for completeness.}

\begin{definition}\label{def:sofic}
A countable group $G$ is \emph{sofic} if for every finite subset $F\subseteq G$ and
every $\varepsilon>0$, there exist a natural number $d$ and a map
$\sigma:F\to\mathrm{Sym}(d)$ (not necessarily a homomorphism) such that:
\begin{enumerate}
\item $\sigma(1)=\mathrm{id}$ (if $1\in F$);
\item for all $u,v,w\in F$ with $uv=w$ in $G$, the normalised Hamming distance
satisfies $d_H(\sigma(u)\sigma(v),\sigma(w))<\varepsilon$;
\item for all distinct $u,v\in F$,
$d_H(\sigma(u),\sigma(v))>1-\varepsilon$.
\end{enumerate}
\end{definition}

\begin{proposition}\label{prop:sofic}
The sofic locus in $\NSub(F_\infty)$ is $\Pi^0_1$-complete.
\end{proposition}
\begin{proof}
Fix a finite $S\subseteq F_\infty$ containing $1$ and
$k\geqslant1$.  Let $U_{S,k}$ consist of normal subgroups $N$
for which there are $d\geqslant1$ and
$\sigma:S\to\operatorname{Sym}(d)$ satisfying
$\sigma(1)=\mathrm{id}$, together with the following conditions:
\begin{enumerate}
\item if $uv^{-1}\in N$, then $\sigma(u)=\sigma(v)$;
\item if $uvw^{-1}\in N$, then
$d_H(\sigma(u)\sigma(v),\sigma(w))<1/k$;
\item if $uv^{-1}\notin N$, then
$d_H(\sigma(u),\sigma(v))>1-1/k$.
\end{enumerate}
Here the variables range over $S$.  The first condition makes
$\sigma$ a well-defined map on the image of $S$ in the quotient.
For fixed $S,k$, membership in $U_{S,k}$ is determined solely by
the membership in $N$ of the finite set
\[
 W_S=\{uv^{-1}:u,v\in S\}\cup
     \{uvw^{-1}:u,v,w\in S\}.
\]
Crucially, $W_S$ is independent of $d$ and $\sigma$.
There are only finitely many possible patterns on $W_S$, each
defining a clopen cylinder; $U_{S,k}$ is a union of some of
these finitely many cylinders.  Thus $U_{S,k}$ is clopen.
The sofic locus is $\bigcap_{S,k}U_{S,k}$, and hence is closed.
The class contains the trivial group and is proper by the non-sofic
construction of~\cite[Chapter~3, Theorem~1.1]{OpenAINonsofic2026}.
Proposition~\ref{prop:closed-group-hardness} therefore gives
$\Pi^0_1$-completeness.
\end{proof}

\subsection{Countable abelian groups}\label{subsec:abelian}

We code $A=\Z^{(\N)}/H$ with $H\in\Sub(\Z^{(\N)})$.

\paragraph{Basic structure.}
\begin{itemize}
\item \emph{Finite:} $\Sigma^0_2$
($\exists$ finite transversal $T$ s.t.\ $\forall x\;\bigvee_{t\in T} x-t\in H$).
\item \emph{Finitely generated:} $\Sigma^0_3$
($\exists$ finite $S$ s.t.\ $\forall x\;\exists\vec n:\ x-\sum n_ss\in H$).
\item \emph{Cyclic:} $\Sigma^0_3$
($\exists g\;\forall x\;\exists n:\ x-ng\in H$).
\item \emph{Torsion-free:} closed
($\forall x\;\forall n\ge 2:\ nx\in H\implies x\in H$;
each instance is clopen, so the intersection is closed).
\item \emph{Torsion:} $\Pi^0_2$
($\forall x\;\exists n\ge 1:\ nx\in H$).
\item \emph{Divisible:} $\Pi^0_2$
($\forall a\;\forall n\ge 1\;\exists b:\ nb-a\in H$).
\item \emph{First Ulm subgroup is zero ($G^1=0$):} $\Pi^0_3$.
An abelian group $A_H=\Z^{(\N)}/H$ has $G^1=0$ (\emph{i.e.}, has no non-zero element divisible by every positive integer) iff
\[
\forall x\in \Z^{(\N)}\ 
\Bigl(
x\in H\ \lor\ \exists n\ge 1\ \forall y\in \Z^{(\N)}:\ ny-x\notin H
\Bigr).
\]
For fixed $x$ and $n$, the inner universal condition is closed; hence the
displayed formula has complexity $\Pi^0_3$.
(For torsion-free groups, being reduced is equivalent to $G^1=0$.
We do not need, and therefore do not discuss, the descriptive complexity of
reducedness for general countable abelian groups here.)
\end{itemize}

\paragraph{Torsion decomposition.}
\begin{itemize}
\item \emph{$p$-primary:} $\Pi^0_2$ ($\forall x\;\exists k:\ p^kx\in H$).
\item \emph{Direct sum of cyclic groups:} analytic ($\Sigma^1_1$) in general, since it
quantifies over an entire sequence of potential cyclic generators.
\end{itemize}

\paragraph{As $\Z$-modules.}
\begin{itemize}
\item \emph{Free abelian of finite rank $r$:} $\Sigma^0_3$
($\exists$ independent set of size $r$ generating modulo~$H$).
\item \emph{Injective $\Leftrightarrow$ divisible:} $\Pi^0_2$.
\end{itemize}

\paragraph{Slenderness.}
A torsion-free abelian group $A$ is \emph{slender} if every homomorphism
$\Z^\N\to A$ depends on finitely many coordinates.

\begin{proposition}\label{prop:slender-Pi03}
For countable abelian groups,
\[
\mathrm{Slender}=\mathrm{TorsionFree}\cap \{G : G^1=0\}.
\]
Consequently, in the coding by subgroups of $\Z^{(\N)}$, the class of slender groups is
$\Pi^0_3$.
\end{proposition}

\begin{proof}
By Nunke's theorem~\cite{Nunke1962}, an abelian group $G$ is slender iff it contains no subgroup
isomorphic to any of $\Q$, $\Z^\omega$, $\Z(p)$, or $J_p$ (for any prime $p$).
If $G$ is countable, then $\Z^\omega$ and $J_p$ cannot embed into $G$ because they are
uncountable. The obstruction $\Z(p)$ is exactly torsion, so for countable groups the
criterion reduces to: $G$ is slender iff $G$ is torsion-free and contains no copy of
$\Q$.

For torsion-free abelian groups, ``contains no copy of $\Q$'' is equivalent to $G^1=0$.
Indeed, every copy of $\Q$ is divisible, so a group containing it has non-zero elements
divisible by every integer.
Conversely, if a torsion-free group $G$ has a non-zero element $x$ divisible by every
integer, the unique roots of $x$ generate a subgroup isomorphic to $\Q$.

Thus, for countable abelian groups,
\[
\mathrm{Slender}=\mathrm{TorsionFree}\cap \{G : G^1=0\}.
\]
By the complexity bounds recorded above, torsion-free is closed ($\Pi^0_1$) and $G^1=0$
is $\Pi^0_3$. Hence slenderness is $\Pi^0_3$.
\end{proof}

\begin{theorem}\label{thm:slender-Pi03-complete}
In $\Sub(\Z^{(\N)})$, the class of slender countable abelian groups
and the class $\{H:(\Z^{(\N)}/H)^1=0\}$ are $\Pi^0_3$-complete.
Both hardness statements hold already on the closed subspace of
torsion-free quotients.  On this subspace, reducedness is also
$\Pi^0_3$-complete.
\end{theorem}
\begin{proof}
The upper bounds were proved above.  We reduce the standard
$\Pi^0_3$-complete set
\[
 \mathrm{Fin}^{\N}
 =\{\alpha\in2^{\N\times\N}:
       \forall i\;\exists m\;\forall n\geqslant m\;
                      \alpha(i,n)=0\}.
\]
Identify $\Z^{(\N)}$ with the free abelian group on generators
$e_{i,n}$, where $i,n\in\N$.  For $\alpha\in2^{\N\times\N}$ put
\[
 d_{i,n}(\alpha)=
 \begin{cases}
  1,&\alpha(i,n)=0,\\
  (n+2)!,&\alpha(i,n)=1,
 \end{cases}
 \qquad
 D_{i,0}(\alpha)=1,
 \qquad
 D_{i,n}(\alpha)=\prod_{k<n}d_{i,k}(\alpha).
\]
Let $H_\alpha$ be generated by all
$e_{i,n}-d_{i,n}(\alpha)e_{i,n+1}$.
Equivalently, $H_\alpha$ is the kernel of the homomorphism
\[
 e_{i,n}\longmapsto D_{i,n}(\alpha)^{-1}u_i
 \quad\hbox{into }\Q^{(\N)},
\]
where $(u_i)$ is the standard basis.  To verify the kernel identity,
move any finite linear combination within each row to a common
terminal generator using the displayed relations; its image vanishes
exactly when every resulting integer coefficient vanishes.
Thus
\[
 A_\alpha:=\Z^{(\N)}/H_\alpha
 \cong\bigoplus_{i\in\N} A_{\alpha,i},
 \qquad
 A_{\alpha,i}=\bigcup_{n\in\N}D_{i,n}(\alpha)^{-1}\Z\subseteq\Q.
\]
In particular, every $A_\alpha$ is torsion-free.
If row $i$ contains only finitely many ones, $D_{i,n}$ is eventually
constant, so $A_{\alpha,i}$ is infinite cyclic.  If that row contains
infinitely many ones, its factorial factors have unbounded indices;
every positive integer divides some $D_{i,n}$, and consequently
$A_{\alpha,i}=\Q$.

If $\alpha\in\mathrm{Fin}^{\N}$, the group $A_\alpha$ is a direct
sum of infinite cyclic groups.  A non-zero element has a non-zero
coordinate in one of these cyclic groups, and cannot be divisible
by every positive integer.  Hence $A_\alpha^1=0$ and $A_\alpha$ is
slender.  If $\alpha\notin\mathrm{Fin}^{\N}$, the group
$A_\alpha$ has a direct summand isomorphic to $\Q$, so it has
non-zero first Ulm subgroup and is neither reduced nor slender.

Finally, if $z=\sum_{i,n}c_{i,n}e_{i,n}$ has finite support, then
\[
 z\in H_\alpha
 \quad\Longleftrightarrow\quad
 \sum_n\frac{c_{i,n}}{D_{i,n}(\alpha)}=0
 \quad\hbox{for every row occurring in the support of }z.
\]
This condition depends on finitely many coordinates of $\alpha$.
Thus $\alpha\mapsto H_\alpha$ is continuous, proving all the
asserted lower bounds.
\end{proof}

\subsection{Countable rings}\label{subsec:rings}

We code $R=R_\infty/I$ where $R_\infty=\Z\langle X_n:n\in\N\rangle$ and
$I\in\Id(R_\infty)$.
The terms integral domain, division ring, field, simple ring, prime ring,
and local ring exclude the zero ring.  In the formulae below this adds
the clopen condition $1\notin I$, which does not affect the stated bounds.

\paragraph{Basic algebraic structure.}
\begin{itemize}
\item \emph{Commutative:} closed ($\forall a,b:\ ab-ba\in I$).
\item \emph{Finite:} $\Sigma^0_2$ (finite transversal).
\item \emph{Integral domain (commutative case):} closed
($\forall a,b:\ ab\in I\implies a\in I\lor b\in I$).
\item \emph{Division ring:} $\Pi^0_2$
($\forall a\;\exists b:\ a\in I\lor(ab-1\in I\text{ and }ba-1\in I)$;
the first disjunct is clopen and the second is a countable union of clopen conditions,
so each instance is open and the intersection is $G_\delta$).
\item \emph{Field:} $\Pi^0_2$ (commutative $\cap$ division ring; intersection of closed and $\Pi^0_2$).
\item \emph{Reduced:} closed
($\forall a\;\forall n\ge 1:\ a^n\in I\implies a\in I$).
\item \emph{Nil ring:} clopen.  In the unital signature, nilness forces $1=0$, so this is the singleton $\{R_\infty\}$.
\item \emph{Boolean ring:} closed ($\forall a:\ a^2-a\in I$).
\end{itemize}

\paragraph{Ideals and factor structure.}
\begin{itemize}
\item \emph{Simple:} $\Pi^0_2$
($\forall a\;\exists m,b_j,c_j:\ a\in I\lor\sum_{j=1}^mb_jac_j-1\in I$;
$a\in I$ is clopen, the existential clause is a countable union of clopen conditions,
hence open; $\forall(\text{open})=G_\delta$).
\item \emph{Prime (two-sided):} $\Pi^0_2$
($\forall a,b\;\exists r:\ a\in I\lor b\in I\lor arb\notin I$;
each disjunct is clopen, so the disjunction is open and $\forall(\text{open})=G_\delta$).
\item \emph{Semiprime:} $\Pi^0_2$
($\forall a:\ (\forall r:\ ara\in I)\implies a\in I$).
\item \emph{Local:} $\Pi^0_2$
($\forall a:\ [\exists u:\ au-1,ua-1\in I]\lor[\exists v:\ (1\!-\!a)v-1,v(1\!-\!a)-1\in I]$;
each $\exists$ clause is open, so the disjunction is open and $\forall(\text{open})=G_\delta$).
\item \emph{Von Neumann regular:} $\Pi^0_2$
($\forall a\;\exists x:\ a-axa\in I$; $\forall\exists(\text{clopen})=G_\delta$).
\item \emph{Unit-regular:} $\Pi^0_2$
($\forall a\;\exists u,v:\ uv-1,vu-1,a-aua\in I$;
$\forall\exists(\text{clopen})=G_\delta$).
\end{itemize}

\paragraph{Dedekind finiteness.}

\begin{proposition}\label{prop:ring-df}
Let $R=R_\infty/I$ be coded by $I\in\Id(R_\infty)$.  The class of codes $I$ such that
$R$ is \emph{Dedekind finite} (directly finite), \emph{i.e.}\
$\forall a,b\in R\ (ab=1\ \Rightarrow\ ba=1)$,
is a closed subset of $\Id(R_\infty)$ (hence $\Pi^0_1$).
Its complement (Dedekind infiniteness) is open.
\end{proposition}

\begin{proof}
Fix $u,v\in R_\infty$.  In the quotient $R_\infty/I$ we have $uv=1$ iff $uv-1\in I$,
and $vu=1$ iff $vu-1\in I$.  For fixed $(u,v)$ the set
\[
E_{u,v}:=\{I:\ (uv-1\in I)\Rightarrow(vu-1\in I)\}
\]
is the complement of the clopen cylinder
$\{I:\ uv-1\in I\ \wedge\ vu-1\notin I\}$, hence is clopen.
Dedekind finiteness is $\bigcap_{u,v\in R_\infty}E_{u,v}$, a countable intersection
of clopen sets, therefore closed.
The complement is the corresponding countable union of clopen sets, hence open.
\end{proof}

\begin{proposition}\label{prop:ring-df-complete}
Dedekind finiteness is $\Pi^0_1$-complete in $\Id(R_\infty)$.
\end{proposition}
\begin{proof}
Closedness was proved in Proposition~\ref{prop:ring-df}.
Let $R=\Z\langle u,v\rangle/(uv-1)$.  This ring has characteristic
zero and $vu\ne1$: represent $u$ and $v$ on $\Z^{(\N)}$ as the
left and right shifts, respectively, so that $uv=1$ and
$vu$ annihilates the first basis vector.
Choose a sequence $(r_j)$ generating $R$ as a unital ring.
For $\alpha=0^\N$, evaluate all generators $X_j$ at $0$ in $R$;
the image is its prime subring $\Z$.  Otherwise, if $m$ is the first
index with $\alpha(m)=1$, evaluate $X_j$ at $0$ for
$j\leqslant m$ and at $r_{j-m-1}$ for $j>m$.
The image is now $R$.  Membership of a fixed polynomial in the
evaluation kernel depends on finitely many coordinates of
$\alpha$, exactly as in Proposition~\ref{prop:closed-group-hardness}.
Thus this is a continuous reduction of $\{0^\N\}$ to the
Dedekind-finite locus.
\end{proof}

\begin{remark}\label{rem:ring-df-contrast}
The contrast with the Banach-algebra case (Proposition~\ref{prop:df-banach}) is
instructive: in countable rings the Dedekind finiteness condition involves no norm
bounds and the atomic predicates are clopen, yielding a closed ($\Pi^0_1$) set.
In Banach algebras the atomic predicates involve continuous (not clopen) norms, but
the open-fragment/Neumann-series argument above still yields a closed set.
\end{remark}

\subsection{Lattices and Boolean algebras}\label{subsec:lattices}

We code lattices as quotients $L_\infty/\theta$ with $\theta$ a lattice congruence
on the free lattice $L_\infty$ in the signature $(\wedge,\vee)$.  If considering
bounded lattices, we use $(\wedge,\vee,0,1)$.

\paragraph{Lattice properties.}

\begin{itemize}
\item \emph{Finite:} $\Sigma^0_2$ (finite transversal).
\item \emph{Distributive:} closed (satisfies the identity
$x\wedge(y\vee z)=(x\wedge y)\vee(x\wedge z)$).
\item \emph{Modular:} closed (satisfies the Horn condition
$x\le b\implies x\vee(a\wedge b)=(x\vee a)\wedge b$).
\item \emph{Complemented (bounded case):} $\Pi^0_2$.
For a fixed $x$, the condition $\exists y:\ (x\wedge y=0)\text{ and }(x\vee y=1)$
is open (a single existential over the countable domain).  Intersecting over all~$x$
gives~$\Pi^0_2$.
\end{itemize}

\paragraph{Boolean algebras.}
If we use the Boolean-algebra signature $(\wedge,\vee,{}',0,1)$, Boolean algebras form
a variety and hence a closed subspace of the corresponding congruence space on
$B_\infty$.  If we use only the lattice signature $(\wedge,\vee,0,1)$, the property of
being a Boolean algebra is $\Pi^0_2$ (distributive, which is closed, and complemented,
which is $\Pi^0_2$).  Working within the closed space of Boolean algebras $\Id(B_\infty)$:

\begin{itemize}
\item \emph{Atomless:} $\Pi^0_2$
($\forall x>0\;\exists y:\ 0<y<x$).
\item \emph{Atomic:} $\Pi^0_3$
($\forall x>0\;\exists y\le x:\ y\text{ is an atom}$, where ``$y$ is an atom'' is the
closed condition $y>0\wedge\forall z\le y\ (z=0\vee z=y)$; the $\forall\exists\forall$
scheme yields~$\Pi^0_3$).
\item \emph{Complete:} $\Sigma^0_2$.  A countable Boolean algebra is complete if and only
if it is finite, so this reduces to the property of being finite.
\end{itemize}

\subsection{Optimality of Borel bounds}
\label{subsec:optimality}

The bounds established above are not merely upper estimates: several fundamental
properties are \emph{complete} for their respective Borel classes, demonstrating
that the definability method assigns optimal ranks.

\begin{theorem}\label{thm:finite-sigma2-complete}
In the parameter space $\NSub(F_\infty)$, the property of being
a finite group is $\Sigma^0_2$-complete.  The same holds for rings, abelian groups,
lattices, and Boolean algebras in their respective parameter spaces.
\end{theorem}

\begin{proof}
Membership in $\Sigma^0_2$ was noted in Section~\ref{subsec:groups}.  For completeness,
we reduce from the canonical $\Sigma^0_2$-complete set
$\mathrm{Fin} = \{ \alpha \in 2^\N : \exists N\; \forall n \ge N\; (\alpha(n) = 0) \}$.

Define $V:=\bigoplus_{k\in\N}\Z/2\Z$ with canonical basis $(e_k)_{k\in\N}$.
For $\alpha\in 2^\N$ set
\[
k_\alpha(n):=\sum_{i<n}\alpha(i)\in\N,
\]
and define a homomorphism $\pi_\alpha:F_\infty\to V$ by $\pi_\alpha(x_n)=e_{k_\alpha(n)}$.
Let 
\[
    N_\alpha:=\ker(\pi_\alpha)\in\NSub(F_\infty).
\]
For the abelian-group coding, define the homomorphism
\[
\widetilde\pi_\alpha:\Z^{(\N)}\to V,
\qquad
\widetilde\pi_\alpha(e_n)=e_{k_\alpha(n)},
\]
and let
\[
H_\alpha:=\ker(\widetilde\pi_\alpha)\in \Sub(\Z^{(\N)}).
\]
Then
\[
\Z^{(\N)}/H_\alpha\cong \bigoplus_{j\in\operatorname{im}(k_\alpha)} \Z/2\Z,
\]
which is finite iff $\alpha\in\mathrm{Fin}$ and countably infinite otherwise.
Exactly as in the free-group coding, membership of a fixed element of $\Z^{(\N)}$ in
$H_\alpha$ depends only on finitely many values of $\alpha$, so the map
$\alpha\mapsto H_\alpha$ is continuous.

Then $G_\alpha:=F_\infty/N_\alpha$ is abelian of exponent dividing $2$, and in $G_\alpha$
we have $x_n=x_{n+1}$ whenever $\alpha(n)=0$ (since then $k_\alpha(n)=k_\alpha(n+1)$).
Thus $G_\alpha$ has one $\Z/2\Z$-generator for each $\sim_\alpha$-equivalence class
of $\N$ generated by the edges $n\sim_\alpha n{+}1$ when $\alpha(n)=0$.

If $\alpha\in\mathrm{Fin}$, then $\alpha(n)=0$ for all $n\ge N$, so $k_\alpha(n)$ is bounded
and $G_\alpha$ is generated by finitely many elements of order $2$, hence is finite.
If $\alpha\notin\mathrm{Fin}$, then $k_\alpha(n)\to\infty$, so $G_\alpha\cong\bigoplus_\N\Z/2\Z$
and is infinite.

Finally, the map $\alpha\mapsto N_\alpha$ is continuous.
Indeed, for a fixed reduced word $w\in F_\infty$ let $M$ be the largest index of a generator
appearing in $w$.  Then $\pi_\alpha(w)$ depends only on the values $\alpha(0),\dots,\alpha(M{-}1)$,
hence the set $\{\alpha: w\in N_\alpha\}=\{\alpha:\pi_\alpha(w)=0\}$ is clopen in $2^\N$.
Therefore $\alpha\mapsto N_\alpha$ is continuous as a map into $2^{F_\infty}$.

For rings, let
\[
V_\alpha:=\bigoplus_{j\in\operatorname{im}(k_\alpha)}\mathbb F_2 e_j,
\qquad
R_\alpha:=\mathbb F_2\oplus V_\alpha,
\]
with multiplication
\[
(a,v)(b,w):=(ab,aw+bv),
\]
so that $V_\alpha^2=0$.
Define the unique unital ring homomorphism
\[
\rho_\alpha:R_\infty\to R_\alpha
\]
by
\[
\rho_\alpha(1)=(1,0),
\qquad
\rho_\alpha(X_n)=(0,e_{k_\alpha(n)}) \quad (n\in\N).
\]
This map is surjective, since $(1,0)$ is the unit of $R_\alpha$ and the elements
$(0,e_{k_\alpha(n)})$ generate the square-zero ideal $V_\alpha$.
Then $R_\alpha$ is finite iff $\operatorname{im}(k_\alpha)$ is finite, \emph{i.e.}\ if and only if
$\alpha\in\mathrm{Fin}$; otherwise $R_\alpha$ is countably infinite.

For lattices, put $I_\alpha=\operatorname{im}(k_\alpha)\subseteq\N$
and equip $I_\alpha$ with the chain operations $\min$ and $\max$.
The assignment $x_n\mapsto k_\alpha(n)$ defines a surjective
homomorphism $\lambda_\alpha:L_\infty\to I_\alpha$.
The target is finite exactly when $\alpha\in\mathrm{Fin}$.
Equality of two fixed lattice terms is decided by finitely many
values $k_\alpha(n)$, so $\alpha\mapsto\ker\lambda_\alpha$
is continuous.  For the bounded-lattice signature, adjoin a formal
largest element $\infty$ to $I_\alpha$, use $0$ as the bottom,
and interpret the constant $1$ as $\infty$; the same argument applies.

For Boolean algebras, let $B(I_\alpha)$ be the free Boolean algebra
on generators $(b_j)_{j\in I_\alpha}$ and evaluate the $n$th free
generator of $B_\infty$ at $b_{k_\alpha(n)}$.
The resulting quotient is finite exactly when $I_\alpha$ is finite.
Equality of two terms depends only on the equality pattern of the
finitely many generator indices occurring in those terms, and hence
the map to its kernel ideal is continuous.  Consequently finiteness,
and equivalently completeness, of countable Boolean algebras is
$\Sigma^0_2$-complete.

In each case, membership of a fixed word / polynomial / lattice term in the
corresponding kernel depends only on finitely many values
$k_\alpha(0),\dots,k_\alpha(M)$, hence only on finitely many coordinates of
$\alpha$.  Therefore the maps
$\alpha\mapsto N_\alpha$, $\alpha\mapsto H_\alpha$,
$\alpha\mapsto \ker(\rho_\alpha)$, and $\alpha\mapsto \ker(\lambda_\alpha)$,
as well as the Boolean-algebra kernel map, are continuous, completing the proof.
\end{proof}

\begin{theorem}\label{thm:torsion-divisible-pi2-complete}
In the parameter space $\Sub(\Z^{(\N)})$, the properties
of being a torsion group and of being a divisible group are
both $\Pi^0_2$-complete.
\end{theorem}

\begin{proof}
Membership in $\Pi^0_2$ is recorded in Section~\ref{subsec:abelian}.  We reduce from
the canonical $\Pi^0_2$-complete set
\[
        P_\infty=\{\alpha\in2^\N: \exists^\infty n\, \alpha(n)=1\}.
\]

\emph{Torsion.}
For $\alpha\in2^\N$, let $H_\alpha\le\Z^{(\N)}$ be generated by
\[
        \{2^{m-n}e_n:n\le m\text{ and }\alpha(m)=1\}.
\]
Then $\Z^{(\N)}/H_\alpha$ is torsion iff $\alpha\in P_\infty$.  Indeed, if
$\alpha$ has infinitely many ones, then for every $n$ there is $m\ge n$ with
$\alpha(m)=1$, so a power of $2$ annihilates the image of $e_n$; every element has
finite support, hence is torsion.  If $\alpha$ has only finitely many ones and $n$ is
larger than the last one, then no non-zero multiple of $e_n$ belongs to $H_\alpha$, so
the image of $e_n$ has infinite order.

The map $\alpha\mapsto H_\alpha$ is continuous.  For a fixed
$z=\sum c_ne_n\in\Z^{(\N)}$ with finite support, membership $z\in H_\alpha$ holds
iff, for every $n$ with $c_n\ne0$, there is some
\[
        m\in\{n,n+1,\dots,n+\nu_2(|c_n|)\}
\]
with $\alpha(m)=1$; here $\nu_2(q)$ is the exponent of $2$ in $q$.  Thus the cylinder
$\{\alpha:z\in H_\alpha\}$ depends on only finitely many coordinates of
$\alpha$.

\emph{Divisible.}
For $\alpha\in2^\N$, define
\[
        d_n(\alpha)=
        \begin{cases}
        1, & \alpha(n)=0,\\
        (n+2)!, & \alpha(n)=1,
        \end{cases}
\]
and let $J_\alpha\le\Z^{(\N)}$ be generated by
\[
        e_n-d_n(\alpha)e_{n+1}\qquad(n\in\N).
\]
Equivalently, $\Z^{(\N)}/J_\alpha$ is the direct limit of copies of $\Z$ with
bonding maps multiplication by $d_n(\alpha)$.

If $\alpha\in P_\infty$, then the indices with $\alpha(n)=1$ are unbounded.  Given
an element represented by $ae_i$ and an integer $q\ge1$, choose $t>i$ such that
$\prod_{k=i}^{t-1}d_k(\alpha)$ is divisible by $q$; this is possible because one of
the factorial factors can be chosen large enough.  Since
\[
        e_i=\Bigl(\prod_{k=i}^{t-1}d_k(\alpha)\Bigr)e_t
\]
in the quotient, $ae_i$ is divisible by $q$.  Thus the whole quotient is divisible.
If $\alpha\notin P_\infty$, then all but finitely many $d_n(\alpha)$ are equal to
$1$, and the quotient embeds as the cyclic subgroup $M^{-1}\Z\subseteq\Q$
for some integer $M$; this group is not divisible.

The map $\alpha\mapsto J_\alpha$ is continuous.  Put
$D_0=1$ and $D_i=\prod_{k<i}d_k(\alpha)$.  The homomorphism
\[
        e_i\longmapsto 1/D_i
\]
from $\Z^{(\N)}$ into $\Q$ has kernel exactly $J_\alpha$.  Hence, for a fixed
$z=\sum_{i\le r}c_ie_i$, the condition $z\in J_\alpha$ is
\[
        \sum_{i\le r} c_i/D_i=0,
\]
which depends only on $\alpha(0),\dots,\alpha(r-1)$.  This proves continuity and
therefore $\Pi^0_2$-hardness for both properties.
\end{proof}

\begin{theorem}\label{thm:locally-finite-complete}
The locally finite groups form a $\Pi^0_2$-complete subset of
$\NSub(F_\infty)$.  Hardness holds already among abelian groups.
\end{theorem}
\begin{proof}
Fix a finite $S\subseteq F_\infty$.  Its image generates a finite
subgroup of $F_\infty/N$ if and only if there is a finite set
$T\subseteq\langle S\rangle$ containing $S\cup\{1\}$ such that
\[
 \forall u,v\in T\;\exists w\in T:\ uvw^{-1}\in N,
 \qquad
 \forall u\in T\;\exists v\in T:\ uv\in N.
\]
Indeed, these conditions say that the image of $T$ is a finite
subgroup containing the image of $S$; conversely, choose word
representatives when that generated subgroup is finite.
For fixed $T$ the conditions are clopen, so finiteness of the
subgroup generated by $S$ is open.  Intersecting over the
countably many finite $S$ proves the $\Pi^0_2$ upper bound.

For hardness, use the torsion construction in
Theorem~\ref{thm:torsion-divisible-pi2-complete}:
\[
 H_\alpha=
 \big\langle2^{m-n}e_n:n\leqslant m,\ \alpha(m)=1\big\rangle
 \leqslant\Z^{(\N)}.
\]
That theorem proves continuity of $\alpha\mapsto H_\alpha$.
If $\alpha$ has infinitely many ones, the quotient is a direct
sum of finite cyclic $2$-groups, hence locally finite.  Otherwise,
some $e_n+H_\alpha$ has infinite order.  Taking inverse images
under the abelianisation $F_\infty\to\Z^{(\N)}$ is continuous
and preserves the quotient.  This reduces the $\Pi^0_2$-complete
set of sequences with infinitely many ones to local finiteness.
\end{proof}

\begin{theorem}\label{thm:ring-localisation-complete}
Each of the following loci is $\Pi^0_2$-complete in
$\Id(R_\infty)$: fields, division rings, simple rings, local
rings, von Neumann regular rings, and unit-regular rings.
All six hardness statements hold already among commutative
integral domains of characteristic zero.
\end{theorem}
\begin{proof}
The upper bounds were proved in Section~\ref{subsec:rings}.
For $\alpha\in2^\N$ set
\[
 d_n(\alpha)=
 \begin{cases}
  1,&\alpha(n)=0,\\
  (n+2)!,&\alpha(n)=1,
 \end{cases}
 \qquad
 D_0(\alpha)=1,
 \qquad D_n(\alpha)=\prod_{j<n}d_j(\alpha),
\]
and let $I_\alpha$ be the kernel of the unital homomorphism
\[
 R_\infty\longrightarrow\Q,
 \qquad X_n\longmapsto D_n(\alpha)^{-1}.
\]
Its image is the subring
$R_\alpha=\Z[D_n(\alpha)^{-1}:n\in\N]$ of $\Q$.
Evaluation of any fixed polynomial depends on finitely many
coordinates of $\alpha$, so $\alpha\mapsto I_\alpha$ is
continuous.

If $\alpha$ has infinitely many ones, every positive integer
divides some $D_n(\alpha)$, and $R_\alpha=\Q$, which has all
six properties.  Otherwise $D_n(\alpha)$ is eventually a
fixed integer $D$, and $R_\alpha=\Z[1/D]$.
Choose distinct primes $p,q$ not dividing $D$.
The ideals $pR_\alpha$ and $qR_\alpha$ are distinct maximal
ideals, as reduction modulo $p$ or $q$ gives the respective
prime fields.  Hence $R_\alpha$ is neither a field nor a
division ring, simple ring, or local ring.
Moreover, $p=pxp$ would imply $px=1$ in this integral domain,
contrary to the non-invertibility of $p$.  Thus $R_\alpha$
is not von Neumann regular, and therefore is not unit-regular.
The same continuous map consequently reduces the
$\Pi^0_2$-complete set of sequences with infinitely many ones
to each of the six loci.
\end{proof}

\begin{theorem}\label{thm:nilpotent-soluble-complete}
Nilpotence and solubility are $\Sigma^0_2$-complete properties
in $\NSub(F_\infty)$.  Both lower bounds hold already among
locally finite $2$-groups.
\end{theorem}
\begin{proof}
Both classes are countable unions of the closed classes with
a fixed bound on nilpotency class or derived length.
For $n\in\N$ let $G_n=\operatorname{UT}_{2^n+1}(\mathbb F_2)$,
the finite group of upper unitriangular matrices.
These are nilpotent $2$-groups, and their derived lengths
are unbounded.  For completeness, write $t_{ij}=1+E_{ij}$.
The identity $[t_{ij},t_{jk}]=t_{ik}$ shows inductively that
$t_{i,i+2^r}\in G_n^{(r)}$ whenever $i+2^r\leqslant2^n+1$.
In particular $G_n^{(n)}\ne\{1\}$.
Nilpotence follows directly from the filtration by vanishing
initial superdiagonals: commutators add their vanishing widths.

Choose a generating tuple $g_{n,1},\ldots,g_{n,r_n}$ for each
$G_n$, and index the free generators of $F_\infty$ by the
pairs $(n,j)$ with $1\leqslant j\leqslant r_n$.
For $\alpha\in2^\N$, define a homomorphism into
$\bigoplus_nG_n$ by sending $x_{n,j}$ to $g_{n,j}$ in the
$n$th coordinate when $\alpha(n)=1$, and to the identity
otherwise.  Its image is
\[
 G_\alpha=\bigoplus_{\{n:\alpha(n)=1\}}G_n.
\]
The kernel depends continuously on $\alpha$, because a fixed
word involves only finitely many coordinates.
Every $G_\alpha$ is a locally finite $2$-group.  If $\alpha$
has finite support, it is a finite direct product of nilpotent
groups, hence nilpotent and soluble.  If the support is infinite,
it contains factors $G_n$ of unbounded derived length, so
$G_\alpha$ is not soluble and consequently is not nilpotent.
Thus this map reduces the $\Sigma^0_2$-complete set
$\mathrm{Fin}$ simultaneously to both classes.
\end{proof}

\begin{theorem}\label{thm:boolean-low-rank-complete}
In $\Id(B_\infty)$, atomlessness is $\Pi^0_2$-complete and
atomicity is $\Pi^0_3$-complete.
\end{theorem}
\begin{proof}
The upper bounds were recorded above.  We use the homeomorphism
between closed subsets of $2^\N$ and ideals of $B_\infty$
established in Lemma~\ref{lem:F-to-I-borel}.
For $\beta\in2^\N$ put
\[
 K_\beta=\{x\in2^\N:x(n)=0\text{ whenever }\beta(n)=0\}.
\]
For a basic cylinder $[s]$, whether $K_\beta$ meets $[s]$ is
determined by the initial segment $\beta\upharpoonright |s|$.
Since hit and miss conditions for clopen cylinders generate the
Vietoris topology, $\beta\mapsto K_\beta$ is continuous.
If $\beta$ has finitely many ones, $K_\beta$ is finite and non-empty;
if it has infinitely many ones, $K_\beta$ has no isolated points.
Atoms of $\operatorname{Clop}(K)$ correspond exactly to isolated
points of $K$.  Consequently this construction reduces the
$\Pi^0_2$-complete set of sequences with infinitely many ones
to atomlessness.

For atomicity, index rows by $i=0,1,\ldots$, take
$\alpha\in2^{\N\times\N}$ and write
$\alpha_i(n)=\alpha(i,n)$.  Using concatenation of binary sequences,
define
\[
 F_\alpha=\{0^\N\}\cup
      \bigcup_{i\in\N}\{0^i1x:x\in K_{\alpha_i}\}.
\]
This set is closed: its displayed compact components lie in
pairwise disjoint cylinders and can accumulate outside those
cylinders only at $0^\N$, which has been included.
For a finite binary word $s$, if $s$ contains no one then
$F_\alpha\cap[s]\ne\varnothing$ automatically.  Otherwise the
first one in $s$ identifies a unique component, and the hit
condition depends on finitely many bits of that row.  Thus
$\alpha\mapsto F_\alpha$ is continuous in the Vietoris topology.

A Boolean algebra $\operatorname{Clop}(K)$ is atomic exactly when
the isolated points are dense in $K$: every non-empty clopen set
must contain a singleton that is clopen in $K$.
If every row of $\alpha$ contains only finitely many ones, every
component of $F_\alpha$ is finite and consists of isolated points;
their union is dense because it accumulates at $0^\N$.
If some row contains infinitely many ones, its component is a
non-empty clopen subset of $F_\alpha$ with no isolated points,
so $\operatorname{Clop}(F_\alpha)$ is not atomic.
Thus $\alpha\mapsto F_\alpha$ reduces the $\Pi^0_3$-complete set
$\mathrm{Fin}^{\N}$ to atomicity.  Composing both reductions with
the ideal homeomorphism proves the theorem.
\end{proof}

\section{Coanalytic-complete properties in quotient codings}\label{sec:Pi11-example}

We give an explicit non-Borel example within the quotient paradigm, showing that some
natural Banach-space properties of quotients cannot be located within the Borel
hierarchy.  This calibrates the reach of the definability method developed in earlier
sections.

\subsection{Closed subsets of Cantor space and commutative quotients}

Let $K_0=2^\N$ be the Cantor space and let $\mathscr F(K_0)$ be the hyperspace of
all closed subsets of $K_0$, including the empty set, with the Vietoris
topology.  Its Borel structure is the Effros Borel structure; the empty
set is isolated.

For $F\in\mathscr F(K_0)$ define the closed ideal
\[
I_F=\{f\in C(K_0): f|_F=0\}\triangleleft C(K_0).
\]
Then the restriction map $C(K_0)\to C(F)$ is a surjective $*$-homomorphism with kernel
$I_F$, hence $C(K_0)/I_F\cong C(F)$ as $C^*$-algebras.

\begin{lemma}\label{lem:F-to-I-borel}
Let $\mathscr F(K_0)$ denote the Vietoris hyperspace of all closed
subsets of $K_0=2^\N$, including the empty set.  The assignments
\[
 F\longmapsto\{f\in C(K_0):f|_F=0\},\qquad
 F\longmapsto\{U\in\operatorname{Clop}(K_0):U\cap F=\varnothing\}
\]
are homeomorphisms onto the Wijsman ideal space of $C(K_0)$ and
the ideal space of $B_\infty\cong\operatorname{Clop}(K_0)$,
respectively.
\end{lemma}
\begin{proof}
For the first assignment,
\[
 \operatorname{dist}(f,I_F)=\max_{x\in F}|f(x)|,
\]
with value $0$ when $F=\varnothing$.  This is continuous in the
Vietoris topology for every $f\in C(K_0)$; the empty set is an
isolated point.  Hence the assignment is Wijsman-continuous.
For the second assignment, the membership test $U\cap F=\varnothing$
is clopen when $U$ is clopen, so the assignment is continuous.
Both maps are bijective by commutative ideal theory and Stone
duality, respectively.  A continuous bijection from a compact
space to a Hausdorff space is a homeomorphism.
\end{proof}

\subsection{Countability of closed sets is \texorpdfstring{$\Pi^1_1$}{Pi11}-complete}

Let $\CTbl\subseteq\mathscr F(K_0)$ be the set of countable closed subsets
of~$K_0$.

\begin{theorem}\label{thm:countable-closed-Pi11}
$\CTbl$ is $\Pi^1_1$-complete.
\end{theorem}

\begin{proof}
It is classical that $\CTbl$ is coanalytic: a closed set $F\subseteq K_0$ is
uncountable iff it contains a perfect subset.

$\Pi^1_1$-hardness goes back to Hurewicz; see
\cite[Theorem~27.5 and Section~33.B]{Kechris}.  Applied to Cantor space, this theorem says directly that the countable
compact subsets form a $\Pi^1_1$-complete subset of the Vietoris
hyperspace.  Adding the isolated empty set does not change completeness.
\end{proof}

\subsection{Separable dual of \texorpdfstring{$C(F)$}{C(F)}}

\begin{lemma}\label{lem:CF-dual-separable}
For $F\in\mathscr F(K_0)$, the dual Banach space $C(F)^*$ is separable iff $F$ is
countable.
\end{lemma}

\begin{proof}
If $F$ is countable, then every regular Borel measure on $F$ is atomic and
$C(F)^*\cong \ell_1(F)$, which is separable.

If $F$ is uncountable, then the family of point masses $(\delta_x)_{x\in F}\subseteq C(F)^*$
is uncountable and satisfies $\|\delta_x-\delta_y\|=2$ for $x\neq y$.
Hence $C(F)^*$ contains an uncountable $2$-separated subset and is non-separable.
\end{proof}

\begin{theorem}\label{thm:dual-separable-Pi11}
Let $A=C^*_{\max}(F_\infty)$. The set
\[
 \{K\in\Ideal(A):(A/K)^*\text{ is separable}\}
\]
is $\Pi^1_1$-complete in the Wijsman topology. Its restriction to
commutative AF quotients is already $\Pi^1_1$-complete. The same
completeness statement holds in the ideal coding of $C(2^\N)$.
\end{theorem}
\begin{proof}
For the upper bound, we make an explicit Borel translation to the
standard quotient coding of real separable Banach spaces. Write
$A_{\R}$ for the underlying real Banach space and fix a real-linear
metric quotient $T:\ell_1(\N;\R)\to A_{\R}$. Such a map is obtained
by sending the unit vectors to a sequence dense in the open unit ball
of $A_{\R}$. For $K\in\Ideal(A)$, put $J_K=T^{-1}(K)$. Then
\[
 \operatorname{dist}(x,J_K)=\operatorname{dist}(Tx,K)
 \qquad(x\in\ell_1(\N;\R)),
\]
so $K\mapsto J_K$ is Wijsman-continuous, and hence Borel for the
Effros Borel structure on the closed real subspaces of $\ell_1$.
Moreover,
\[
 \ell_1(\N;\R)/J_K\cong(A/K)_{\R}
\]
isometrically. Bossard's Corollary~3.3(ii), together with the
comparison of the standard subspace and quotient codings in his
Proposition~2.8~\cite{Bossard2002}, shows that
\[
 \{J\in\SB(\ell_1(\N;\R)):(\ell_1/J)^*\text{ is separable}\}
\]
is coanalytic. This applies to the complex quotient as well: taking
real parts is a real-linear isometric bijection from its complex dual,
viewed over $\R$, onto the dual of its realification. The inverse sends
$f$ to the complex-linear functional $x\mapsto f(x)-\mathrm i f(\mathrm i x)$.
Thus the claimed class in $\Ideal(A)$ is a Borel preimage of a
coanalytic set.

For hardness, fix a surjective unital $*$-homomorphism
$\rho:A\to C(K_0)$, where $K_0=2^\N$. For a non-empty compact
$F\subseteq K_0$, let
\[
 K_F=\rho^{-1}(I_F),\qquad I_F=\{f\in C(K_0):f|_F=0\}.
\]
The quotient is $A/K_F\cong C(F)$, and
\[
 \operatorname{dist}(a,K_F)
   =\max_{t\in F}|\rho(a)(t)|\qquad(a\in A).
\]
Consequently $F\mapsto K_F$ is continuous from the Vietoris
hyperspace to $\Ideal(A)$ with its Wijsman topology. By
Lemma~\ref{lem:CF-dual-separable}, its inverse image of the
separable-dual class is the class of countable compact subsets of
$K_0$, which is $\Pi^1_1$-complete by
Theorem~\ref{thm:countable-closed-Pi11}. Every $F\subseteq K_0$ is
zero-dimensional, so every $C(F)$ in this reduction is commutative
and AF. This proves the asserted hardness within that locus, as well
as completeness in the full ideal space. Finally, the homeomorphism
$F\mapsto I_F$ from Lemma~\ref{lem:F-to-I-borel} gives the same result
in $\Ideal(C(K_0))$.
\end{proof}

\begin{remark}\label{rem:Pi11-calibration}
Theorem~\ref{thm:dual-separable-Pi11} shows that certain natural properties of
quotient algebras are \emph{provably not Borel}, even for commutative $C^*$-algebras.
This goes beyond the finite Borel bounds in Theorem~B.
These examples show that the finite-rank catalogues do not exhaust
natural quotient properties.  The source of non-Borelness is established
by the continuous reductions, rather than by quantifier counting alone.
\end{remark}

\subsection{Superatomic Boolean algebras}\label{subsec:superatomic}

The parameterisation of countable Boolean algebras provides another application
of the $\CTbl$ space.  Let $B_\infty$ be the free Boolean algebra on countably many
generators.  The Stone space of $B_\infty$ is canonically homeomorphic to the Cantor
space $K_0=2^\N$.

By Stone duality, there is a canonical bijection between ideals
$I\in\Id(B_\infty)$ and closed subsets $F\in\mathscr F(K_0)$: the quotient Boolean
algebra $B_\infty/I$ is isomorphic to the algebra of clopen subsets of~$F$.

A Boolean algebra is called \emph{superatomic} if every non-zero homomorphic image contains
an atom.  Indeed, quotient Boolean algebras correspond to closed subspaces of the
Stone space, and atoms correspond to isolated points.  Thus superatomicity
is equivalent to scatteredness of the Stone space.  A compact metrisable
space is scattered if and only if it is countable, by the perfect-set theorem.

\begin{theorem}\label{thm:superatomic-Pi11}
In the parameter space of countable Boolean algebras coded by ideals in the free
Boolean algebra $\Id(B_\infty)$, the property
\[
\{I\in\Id(B_\infty):\ B_\infty/I\ \text{is superatomic}\}
\]
is $\Pi^1_1$-complete.
\end{theorem}

\begin{proof}
Under the Stone duality correspondence $I\leftrightarrow F$, the map
$F\mapsto I_F$ from $\mathscr F(K_0)$ to $\Id(B_\infty)$ is a homeomorphism by Lemma~\ref{lem:F-to-I-borel}.
By the characterisation above, $B_\infty/I_F$ is superatomic if and only if
$F\in\CTbl$.

Since $\CTbl$ is $\Pi^1_1$-complete
(Theorem~\ref{thm:countable-closed-Pi11}), this homeomorphism sends $\CTbl$ exactly onto the superatomic locus
and continuously reduces one class to the other in both directions.
Therefore superatomicity is $\Pi^1_1$-complete.
\end{proof}

\begin{remark}\label{rem:slenderness-resolved}
In contrast with superatomicity, slenderness for countable abelian groups
is Borel, and Theorem~\ref{thm:slender-Pi03-complete} determines its
exact complexity as $\Pi^0_3$-complete, even for torsion-free quotients.
\end{remark}

\section{Further questions}

The preceding results settle the fixed and finite rank questions for
real rank, stable rank, nuclear dimension and decomposition rank.
Nuclearity is also $\Pi^0_2$-complete. The sharpness of the $\Pi^0_3$
bound for quasidiagonality remains open in this treatment. A lower bound
requires families beyond the commutative dimension tests, since every
commutative $C^*$-algebra is quasidiagonal. The row construction used
for residual finite dimensionality suggests seeking a continuous family
that switches between quasidiagonal and non-quasidiagonal quotients;
the matrix/CAR family itself cannot do this, since both outcomes are
quasidiagonal.

Uniformly open multiplication is $\Sigma^0_3$ in the Banach-algebra
coding and is $\Pi^0_2$-hard. On the commutative $C^*$-locus it is
$\Pi^0_2$-complete, but we do not know whether the general
$\Sigma^0_3$ bound is optimal. A third-level lower bound would have to leave the commutative
$C^*$-locus; varying the quantitative openness radius in
Proposition~\ref{prop:UO} is a possible source of the extra complexity.
The non-uniform pure-infiniteness condition is coanalytic; its Borelness and possible
coanalytic hardness remain open in this framework. Simplicity of general
Banach-algebra quotients likewise requires arguments beyond the positive
cutdowns available in $C^*$-algebras.

For tensor products, the elementary ideal assignment is
$\Sigma^0_2$-measurable and a fixed nuclear factor gives continuity.
It would be useful to determine optimal hypotheses for joint continuity
of the two-ideal assignment and for finer regularity with a fixed exact
factor. The distinction between generated ideals and minimal tensor
kernels must be retained in such questions.

Finally, the internal $K_0$ construction gives $F_\sigma$ membership
coordinates. Are these coordinate bounds optimal for natural quotient
presentations? A corresponding direct construction for relative
$K$-theory, including non-unital suspensions, would complement the direct unital assignments used here.

\section*{Acknowledgements}
The author is grateful to Ilijas Farah for his interest in this work and for
helpful comments on an earlier draft, especially concerning substructure versions
of the Wijsman-closedness argument, congruence-based parametrisations, and the
comparison with existing codings of separable $C^*$-algebras. The author also thanks Christian Rosendal for perceptive comments on a presentation of this work. The author also thanks Jennifer Pi and Stuart White for communicating Austin Shiner's related work. The author gratefully acknowledges support received from
NCN Sonata-Bis~13 (2023/50/E/ST1/00067).

\bibliographystyle{amsplain}

\end{document}